\def\bu{\bullet}
\def\marker{\>\hbox{${\vcenter{\vbox{
    \hrule height 0.4pt\hbox{\vrule width 0.4pt height 6pt
    \kern6pt\vrule width 0.4pt}\hrule height 0.4pt}}}$}\>}
\def\gpic#1{#1
     \smallskip\par\noindent{\centerline{\box\graph}} \medskip}
\documentclass[12pt]{article}
\usepackage{array,amsmath,amssymb,amsthm}  				    
\usepackage{fullpage,lineno}
\usepackage[utf8]{inputenc}
\usepackage{color,tikz}
\usepackage{fullpage}
\usepackage{hyperref}
\usepackage{mathrsfs}
\usetikzlibrary{decorations.pathmorphing}
\usetikzlibrary{decorations.pathreplacing}
\tikzset{snake it/.style={decorate, decoration=snake}}

\newtheorem{theorem}{Theorem}[section]

\newtheorem{conjecture}[theorem]{Conjecture}

\newtheorem{lemma}[theorem]{Lemma}
\newtheorem{lem}[theorem]{Lemma}

\newtheorem{corollary}[theorem]{Corollary}

\theoremstyle{definition}
\newtheorem{definition}[theorem]{Definition}

\newtheorem{remark}[theorem]{Remark}

\def\st{\colon\,}   
\def\VEC#1#2#3{#1_{#2},\ldots,#1_{#3}}

\def\SE#1#2#3{\sum_{#1=#2}^{#3}}

\def\CH#1#2{\binom{#1}{#2}} 
\def\FR#1#2{\frac{#1}{#2}}
\def\FL#1{\left\lfloor{#1}\right\rfloor} 
\def\CL#1{\left\lceil{#1}\right\rceil}   

\def\NN{{\mathbb N}}

\def\cD{{\mathcal D}}

\def\C#1{\left | #1 \right |}    
\def\ov#1{\overline{#1}}        

\def\cD{{\mathcal D}}
\def\DD{{\mathcal D}}
\def\cF{{\mathcal F}}

\def\esub{\subseteq}
\def\nosub{\not\subseteq}

\def\eps{\varepsilon}

\def\ov#1{\overline{#1}}
\def\la{\langle}
\def\ra{\rangle}
\def\deg{d}
\def\bb{{\ov b}}
\def\hb{{\ov h}}
\def\ib{{\ov i}}
\def\jb{{\ov j}}
\def\kb{{\ov k}}

\def\qb{{\ov q}}
\def\ub{{\ov u}}
\def\hF{{\hat F}}

\long\def\skipit#1{}

\def\qed{\hfill\ifhmode\unskip\nobreak\fi\quad\ifmmode\Box\else\hfill$\Box$\fi\\}

\begin{document}

\title{Caterpillars with $n$ vertices are reconstructible from subgraphs
with at most $n/2+1$ vertices}

\author{
Alexandr V. Kostochka\thanks{University of Illinois, Urbana, IL:
\texttt{kostochk@illinois.edu}.  Research partially supported by NSF grant
DMS-2153507 and NSF RTG grant DMS-1937241.}\,,
Zishen Qu\thanks{University of Illinois, Urbana, IL:
\texttt{zishenq2@illinois.edu}.
Research supported by Natural Sciences and Engineering Research Council of
Canada (NSERC), [funding reference number PGSD-568936-2022];
financ\'ee par le Conseil de recherches en sciences naturelles
et en gnie du Canada (CRSNG), [num\'ero de r\'ef\'erence PGSD-568936-2022]
}\,,
Maddy Ritter\thanks{University of Illinois, Urbana, IL:
\texttt{mritter5@illinois.edu}.}\,,
Douglas B. West\thanks{Zhejiang Normal Univ., Jinhua, China
and University of Illinois, Urbana, IL:
\texttt{dwest@illinois.edu}.  Supported by National Natural Science Foundation
of China grants NSFC 11871439, 11971439, and U20A2068.}\,,
}

\date{\today}
\maketitle

\baselineskip 16pt

\begin{abstract}
The {\it $m$-deck} of an $n$-vertex graph is the multiset of unlabeled induced
subgraphs with $m$ vertices.  Caterpillars are trees in which all nonleaf
vertices lie on a single path.  We prove for $n\ge48$ that any $n$-vertex
caterpillar is reconstructible (up to isomorphism) from its $m$-deck when
$m>n/2$.  The result is sharp, since for $n\ge6$ there are two $n$-vertex
caterpillars having the same $\FL{n/2}$-deck.  Our result proves the special
case for caterpillars of a 1990 conjecture by N\'ydl about trees.
\end{abstract}

\section{Introduction}

The {\it $m$-deck} of a graph $G$ is the multiset of its $m$-vertex induced
subgraphs (as isomorphism classes); each member is an {\it $m$-card}, but the
vertices are unlabeled.  We write the $m$-deck as $\cD_m(G)$ or simply as
$\cD_m$.  Generalizing the classical problem of graph reconstruction, we say
that an $n$-vertex graph $G$ is {\it $\ell$-reconstructible} if it is
determined by $\cD_{n-\ell}(G)$, meaning that no graph not isomorphic to $G$
has the same $(n-\ell)$-deck.

Since every member of $\cD_{m-1}$ arises $n-m+1$ times by deleting a vertex
from a member of $\cD_m$, the $m$-deck of a graph determines its $(m-1)$-deck
(and all decks of smaller subgraphs).  The natural problem is then to find for
each graph the maximum $\ell$ such that it is $\ell$-reconstructible, or
equivalently the minimum $m$ such that it is reconstructible from its $m$-deck.
In this context, Manvel~\cite{M69,M74} extended
the classical Reconstruction Conjecture of Kelly~\cite{Kel1} and Ulam~\cite{U}.

\begin{conjecture}[{\rm Manvel~\cite{M69,M74}}]
For $\ell\in\NN$, there exists a threshold $M_\ell$ such that every graph with
at least $M_\ell$ vertices is $\ell$-reconstructible.
\end{conjecture}

\noindent
Manvel named this ``Kelly's Conjecture'' in honor of the final sentence in
Kelly~\cite{Kel2}, which suggested studying reconstruction from the
$(n-2)$-deck.  The statement that $\cD_m$ determines $\cD_j$ when $j<m$, which
Kelly noted for $m=n-1$, is often called ``Kelly's Lemma''.

The classical Reconstruction Conjecture is $M_1=3$.  Since Manvel's conjecture
is open for all $\ell$, researchers have approached the question by studying
thresholds on the number of vertices for $\ell$-reconstructibility of various
classes or properties of graphs.  The survey by Kostochka and West~\cite{KW}
describes some of these results.  Of particular interest is the threshold for
trees, to extend the seminal result of Kelly~\cite{Kel2} that trees with at
least three vertices are $1$-reconstructible.  N\'ydl~\cite{N90} stated a
conjecture for general $\ell$.

\begin{conjecture}[\cite{N90}]\label{Nconj}
For $n\ge2\ell+1$, no two $n$-vertex trees have the same $(n-\ell)$-deck.
\end{conjecture}

One isolated counterexample to Conjecture~\ref{Nconj} is known;
Groenland et al.~\cite{GJST} found two $13$-vertex trees with the same
$7$-deck.  For general $\ell$, it is known that $n\ge2\ell$ does not suffice.
N\'ydl~\cite{N90} provided two trees with $2\ell$ vertices having the same
$\ell$-deck; a short proof is given in~\cite{KW} using the results
of~\cite{SW}.  The two trees arise from a path with $2\ell-1$ vertices by
adding one vertex adjacent to the central vertex of the path or to a neighbor
of that vertex.  We generalize this example in Theorem~\ref{catsharp}.

Many arguments for reconstructibility of the graphs in a particular class have
two parts, as articulated by Bondy and Hemminger~\cite{BH} for the case
$\ell=1$.  First, one proves that the class is {\it $\ell$-recognizable}; this
means that for the $(n-\ell)$-deck of any $n$-vertex graph, all graphs or no
graphs having that deck belong to the class.  Separately, one proves that the
family is {\it weakly $\ell$-reconstructible}, meaning that no two graphs in
the family have the same deck.  Thus N\'ydl's conjecture states that trees with
at least $2\ell+1$ vertices are weakly $\ell$-reconstructible.  For the other
step, Kostochka, Nahvi, West, and Zirlin~\cite{KNWZa} proved the following
result.

\begin{theorem}[{\rm\cite{KNWZa}}]\label{acyc}
For $n\ge2\ell+1$, except when $(n,\ell)=(5,2)$, the family of $n$-vertex
acyclic graphs is $\ell$-recognizable.
\end{theorem}

\noindent
The case $(n,\ell)=(5,2)$ is excluded because the tree $T_5$ obtained from the
complete bipartite graph $K_{1,3}$ by subdividing one edge has the same
$3$-deck as the disjoint union of a $4$-cycle with an isolated vertex.

By Kelly's Lemma, the $(n-\ell)$ deck determines the $2$-deck and hence the
number of edges.  Since $n$-vertex trees are the $n$-vertex acyclic graphs with
$n-1$ edges, Theorem~\ref{acyc} implies that proving Conjecture~\ref{Nconj}
(excluding $(n,\ell)\in\{(5,2),(13,6)\}$) would show that trees with at least
$2\ell+1$ vertices are $\ell$-reconstructible (except when $\ell\in\{2,6\}$).
Due to $T_5$ not being $2$-reconstructible, at least six vertices are needed to
guarantee $2$-reconstructibility of trees, and Giles~\cite{Gil} proved that
this suffices.

A general threshold for $\ell$-reconstructibility of $n$-vertex trees has been
proved.  Groenland, Johnston, Scott, and Tan~\cite{GJST} proved that $n$-vertex
trees are $\ell$-reconstructible when $n\ge 9\ell+24\sqrt{2\ell}+o(\sqrt\ell)$.
Kostochka, Nahvi, West and Zirlin~\cite{KNWZt} improved the sufficient
condition to $n\ge 6\ell+11$.

Since finding the least number of vertices to guarantee
$\ell$-reconstructibility of trees seems to be quite difficult,
it is natural to consider $\ell$-reconstructibility
of trees with special structure.
A {\it caterpillar} is a tree in which the subgraph induced by the
nonleaf vertices is a path called the \emph{spine}.
Z. Hunter~\cite{Hun} proved that an $n$-vertex
caterpillar is reconstructible from its $m$-deck when
$m>(n/2)+C$ for some constant $C$.  Our result reduces
$C$ to $0$, proving N\'ydl's conjecture for the special case of caterpillars
when $n$ is sufficiently large.

\begin{theorem}\label{main}
When $n\ge48$ and $m>n/2$, every $n$-vertex caterpillar is
reconstructible from its $m$-deck, and this is sharp.
\end{theorem}

We will prove the sharpness of Theorem~\ref{main} at the end of this section.
Because the number of vertices in subgraphs is always an integer, this phrasing
for $m$ in terms of $n$ is equivalent to the threshold on $n$ in terms of
$\ell$ in Conjecture~\ref{Nconj}.  Similarly, the statement of the result
in~\cite{GJST} was that an $n$-vertex tree is reconstructible from its
$m$-deck when $m=8n/9+(4/9)\sqrt{8n+5}$.

In light of Kelly's Lemma, we can speak of ``the deck'' of given information as
being the multiset of all induced subgraphs with at most $(n+2)/2$ vertices.
We will write it as $\cD$.  The individual subgraphs are the ``cards'' in the
deck.  We aim to retrieve the original caterpillar when given $\cD$.

It is easy to see that a tree is a caterpillar if and only if it does not
contain the $7$-vertex tree obtained from $K_{1,3}$ by subdividing each edge.
Hence when $n\ge12$ we can tell from $\cD$ whether the unknown tree is a
caterpillar, and it then remains only to determine which caterpillar has
the given deck.

It is still possible that when $n\ge6$ the conclusion of Theorem~\ref{main} has
no exceptions.  Giles~\cite{Gil} proved that all trees are $2$-reconstructible,
handling the case $n=6$.  When $n=7$ there are $11$ isomorphism classes of
trees, and they are distinguished by their numbers of $4$-vertex paths and
independent $4$-sets.  Only one is not a caterpillar.  The threshold $n\ge48$
in the statement of Theorem~\ref{main} is an artifact of the proof.  Some of
our lemmas invoke thresholds on $n$ to ensure that certain subgraphs are
visible in cards.  The only one that relies on $n\ge48$ is Lemma~\ref{Case3};
otherwise $n\ge27$ would suffice.  We will include ``$n\ge48$'' in statements
as a flag indicating that a threshold on $n$ will be invoked; varying the
threshold in various lemmas would lead to unnecessary distraction from the
main result.

Within the family of caterpillars, a particular caterpillar is specified by its
ordered list of vertex degrees along the spine (the reverse list specifies the
same caterpillar).  Hence it is natural to seek first the multiset of vertex
degrees.  Groenland et al.~\cite{GJST} used algebraic arguments about
polynomial equations (obtained by counting cards that are stars of various
sizes) to prove that the multiset of vertex degrees of an $n$-vertex graph is
determined by its $m$-deck when $m\ge\sqrt{2n\ln2n}$.  Although this is the
strongest known result for large $n$ for reconstructing vertex degrees, note
that $(n+1)/2\ge\sqrt{2n\ln2n}$ only when $n\ge67$, so for $n$ of moderate size
this result does not give us the degree list.  Meanwhile, Kostochka et
al.~\cite{KNWZa} gave a combinatorial argument implying that the degree list of
any $n$-vertex acyclic graph is determined by its $m$-deck when $m>n/2$ (again
for $n\ge6$).  Hence we may assume that the multiset of vertex degrees is known
from the deck.

In particular, given the deck $\cD$ of an unknown caterpillar $G$, we know the
number of nonleaf vertices in $G$; we call this $r$.  Let $\VEC d1r$ denote the
degrees of the nonleaf vertices in nonincreasing order.  Let $\VEC v1r$ be the
vertices of the spine in order as a path, and let $d(v)$ denote the degree of
vertex $v$.  Although $\VEC d1r$ and $d(v_1),\ldots,d(v_r)$ form the same
multiset, they are rarely in the same order; our task is to use $\cD$ to
discover the second order.



Before closing this section, we note one fundamental counting tool that
will be useful at many points in the paper.

\begin{lemma}\label{leaves}
The vertex degrees of an $n$-vertex caterpillar with $n-r$ leaves satisfy
\begin{equation}\label{d'}
n-r-2=\SE i1r (d_i-2) .
\end{equation}
With $m$ vertices of maximum degree,
\begin{equation}\label{up}
d_1\le \FL{\frac{n-r-2}{m}}+2.
\end{equation}
\end{lemma}
\begin{proof}
A longest path has $r+2$ vertices.  Counting the $n-r-2$ vertices not on the
path according to their neighbors on the spine yields~\eqref{d'}.  Keeping only
the $m$ vertices of maximum degree in~\eqref{d'} yields~\eqref{up}.
\end{proof}

In Section~\ref{lowr} we consider the ``low-diameter'' case $r\le(n-6)/2$.
In the subsequent sections we consider the more difficult ``high-diameter''
case, beginning with a general strategy in Section~\ref{highr}.  We then
consider cases in terms of the number of vertices of maximum degree and the
number of vertices of degree at least $3$.

As mentioned after Conjecture~\ref{Nconj}, for even $n$ N\'ydl~\cite{N90}
presented two $n$-vertex caterpillars having the same $n/2$-deck.
A proof of this appears in~\cite{KW} using a special case of the results
of~\cite{SW}.  Here we present a direct self-contained proof generalizing the
sharpness claim of Theorem~\ref{main}.

Let $T_{a,b}$ denote the tree that is the union of paths of lengths $1$, $a$,
and $b$ having a common vertex.  Note that $T_{a,b}$ is a caterpillar with
$a+b+2$ vertices.  A special case of the result below is that when $k=\FL{n/2}$,
the caterpillars $T_{k-1,n-k-1}$ and $T_{k-2,n-k}$ have the same $k$-deck,
which proves sharpness of Theorem~\ref{main} for all $n$ at least $6$.

\begin{theorem}\label{catsharp}
When $b> a\ge k-2$, the caterpillars $T_{a,b}$ and $T_{a+1,b-1}$ have the
same $k$-deck.
\end{theorem}
\begin{proof}
Let $n=a+b+2$; both caterpillars are $n$-vertex trees.
We use induction on $n$.  When $n=3$, we have $(b,a,k)=(1,0,2)$ and the
claim holds because $T_{0,1}\cong T_{1,0}$.

For $n\ge4$, always the trees are the same when $b=a+1$, so we may assume
$b\ge a+2\ge k$.  Let $G=T_{a,b}$ and $H=T_{a+1,b-1}$, and let $v$ denote in
both graphs the leaf at the end of the path of length $b$ or $b-1$ from the
branch vertex.  For $j\le k$, in both $G$ and $H$ there is exactly one
$j$-vertex path containing $v$, because $b\ge k$.

Let $\cD_{k,j}(G)$ and $\cD_{k,j}(H)$ be the subsets of the two $k$-decks in
which the component containing $v$ has $j$ vertices, for $0\le j\le k$ ($j=0$
corresponds to $v$ being omitted).  When $j=k$, both $\cD_{k,j}(G)$ and
$\cD_{k,j}(H)$ consist of a single card that is a $k$-vertex path.  When $j<k$,
we have specified $j$ vertices to be used and one vertex to be omitted.  When
$j=k-1$, one more vertex must be picked from the remaining $n-k$ vertices, so
both $\cD_{k,j}(G)$ and $\cD_{k,j}(H)$ consist of $n-k$ cards, each a
$(k-1)$-vertex path plus an isolated vertex.

For $j\le k-2$, the cards in the two subsets correspond to induced subgraphs
with $k-j$ vertices chosen from $G'$ or $H'$, where since $b\ge k\ge j+2$
we can write $G'=T_{a,b-j-1}$ and $H'=T_{a+1,b-j-2}$.
Since $b\ge k$, we have $b-j-2\ge k-j-2$.
Hence $\min\{a,b-j-2\}\ge (k-j)-2$, so the induction hypothesis applies to
yield $\cD_{k-j}(G')=\cD_{k-j}(H')$.  Adding a $j$-vertex path to each card
yields $\cD_{k,j}(G)=\cD_{k,j}(H)$, and the union over all $j$ completes the
proof.
\end{proof}

\section{Low-diameter Caterpillars}\label{lowr}


\begin{definition}\label{basic}
Given the deck of an unknown caterpillar $G$, let $\#H$ denote the number of
induced subgraphs of $G$ isomorphic to $H$; we are given $\#H$ whenever
$\C{V(H)}\le(n+2)/2$.  Let $P_t$ denote an $t$-vertex path, and let
$\la \VEC u1t\ra$ denote a path with vertices $\VEC u1t$ in order.  Let the
spine of our unknown caterpillar $G$ with $r$ nonleaf vertices be
$\la\VEC v1r\ra$.  Let $d'(v_i)$ denote the number of leaf neighbors of $v_i$;
that is, $d'(v_i)=d(v_i)-1$ when $i\in\{1,r\}$, and otherwise
$d'(v_i)=d(v_i)-2$.  For $1\le i\le r$, let $\ib=r+1-i$; thus $v_i$ and
$v_{\ib}$ are the same distance from the nearest end of the spine.  The $i$th
\emph{level pair} is the unordered pair $\{d(v_i),d(v_{\ib})\}$ of vertex
degrees.  The motivation for the term ``level pair'' is picturing the spine as
in Figure~\ref{3maxfig} in Section~\ref{3max}.
\end{definition}

\begin{theorem}
When $r\le (n-6)/2$, every $n$-vertex caterpillar $G$ with $r$ nonleaf vertices
is reconstructible from $\cD$.
\end{theorem}

\begin{proof}
Longest paths have $r+2$ vertices and contain the spine.  Consider all cards
consisting of a longest path plus one more leaf.  Since we cannot distinguish
the two ends, we index a longest path in such a card as $\la\VEC u0{r+1}\ra$ so
that the neighbor $u_i$ of the extra leaf satisfies $1\le i\le \CL{r/2}$.  Let
$Y_i$ denote the resulting graph.  Note that $u_i\in\{v_i,v_\ib\}$.

We first determine the level pairs $\{d(v_i),d(v_\ib)\}$.
For $i=1$, we have $\#P_{r+2}=d'(v_1)d'(v_r)$ and
$\#Y_1=\CH{d'(v_1)}2d'(v_r)+\CH{d'(v_r)}2d'(v_1)$.
Substituting $d'(v_r)=\#P_{r+2}/d'(v_1)$ into the second equation yields
a quadratic equation, the two roots of which are $d'(v_1)$ and $d'(v_r)$.
We thus obtain the level pair $\{d(v_1),d(v_r)\}$.

Let $M=d'(v_1)d'(v_r)$; note that $M$ is the number of longest paths in $G$.
Consider $i$ with $2<i\le\FL{r/2}$.  Recall that $d'(v_i)=d(v_i)-2$ for
$i\notin\{1,r\}$, so $\#Y_i=[d'(v_i)+d'(v_\ib)]M$.
We also seek the product $d'(v_i)d'(v_\ib)$.
Let $Z_i$ be the graph with $r+4$ vertices obtained from the path
$\la \VEC u0{r+1}\ra$ with $r+2$ vertices by giving each of $u_i$ and
$u_\ib$ one leaf neighbor.  Note that $\#Z_i=d'(v_i)d'(v_\ib)M$.
Since $r+4\le (n+2)/2$, we now know the sum and product of $d'(v_i)$ and
$d'(v_\ib)$ and find the pair as the zeros of a quadratic polynomial.
We thus obtain $\{d(v_i),d(v_\ib)\}$.

If $d(v_i)=d(v_\ib)$ for $1\le i\le\FL{r/2}$, then $G$ is symmetric and the
level pairs suffice to determine $G$.  Otherwise, let $j$ be the least index
such that $d(v_j)\ne d(v_\jb)$; note that $j\le{r/2}$.  By symmetry, we may
assume $d(v_j)<d(v_\jb)$.  Let $X_{i,k}$ be the graph with $r+4$ vertices
obtained from $\VEC u0{r+1}$ by giving each of $u_i$ and $u_k$ one leaf
neighbor; thus $Z_i=X_{i,\ib}$.  With $j$ as above, we focus on $X_{j,k}$ with
$j<k\le{r/2}$.  If $j\ge2$, then
$\#X_{j,k}=[d'(v_j)d'(v_k)+d'(v_\jb)d'(v_\kb)]M$, but
$\#X_{1,k}=\CH{d'(v_1)}2d'(v_k)+\CH{d'(v_r)}2d'(v_\kb)$.

If $d'(v_k)$ and $d'(v_\kb)$ are equal, then we already know them; suppose they
differ.  Since $d'(v_j)$ and $d'(v_\jb)$ are known and unequal, we can
determine which of the pair $\{d'(v_k),d'(v_\kb)\}$ is $d'(v_k)$, because
$ac+bd$ differs from $ad+bc$ when $a\ne b$ and $c\ne d$.  Choosing $d'(v_k)$
correctly is the only way to obtain $\#X_{j,k}/M$ (or $\#X_{1,k}$ when $j=1$)
as the value.  Doing this for each $k$ completes the reconstruction (when $r$
is odd, $d(v_{(r+1)/2})$ is the remaining degree).
\end{proof}

Henceforth we may assume $r\ge(n-5)/2$.

\section{High-diameter Strategy}\label{highr}

The low-diameter case is easy because we can see the entire spine on a card.
We used longest paths to anchor the cards so that we knew (up to reflecting the
spine) the identity of any vertex of degree $3$ on the card.  In the
high-diameter case $r\ge(n-5)/2$, we will instead use high-degree vertices as
anchors.

\begin{remark}
{\it The plan.}
Knowing $r$ and the vertex degrees, we first seek the unordered degree pairs
for all pairs of vertices separated by a fixed distance $j$, aiming to do so
when $j\le (r-1)/2$.  From that, we will determine the level pairs
$\{d(v_k),d(v_\kb)\}$ for all $k$.  In some cases, we also need triples of
vertex degrees, which we will describe more precisely later.  From this
information we will determine the ordered list $(d(v_1),\ldots,d(v_r))$ to
complete the reconstruction (up to reflection).
\end{remark}

In order to obtain degree pairs and degree triples, we will apply a general
counting tool.

\begin{definition}
Given a family $\cF$ of graphs, an {\it $\cF$-subgraph} of a graph $G$ is an
induced subgraph of $G$ belonging to $\cF$.  Let $s(F,G)$ denote the number of
occurrences of $F$ as an induced subgraph in $G$.  Let $\hat s(F,G)$ be the
number of occurrences of $F$ as a maximal $\cF$-subgraph in $G$ (maximality
with respect to induced subgraphs).
\end{definition}

The special case of the next lemma for classical reconstruction ($\ell=1$) is
due to Greenwell and Hemminger~\cite{GH}.  Similar statements for general
$\ell$ appear for example in~\cite{GJST,KNWZa,KNWZt} and elsewhere.  For
completeness, we include the proof from~\cite{KNWZa} rephrased for $m$-decks;
it is slightly simpler than proofs in the literature involving inclusion chains
of subgraphs because we do not need an explicit formula for $\hat s(F,G)$.

\begin{lemma}\label{counting}
Fix an $n$-vertex graph $G$, and let $\cF$ be a family of graphs such that
every $\cF$-subgraph of $G$ is an induced subgraph of a unique maximal
$\cF$-subgraph in $G$.  If the value of $\hat s(F,G)$ is known for every
$F\in\cF$ with at least $m$ vertices, then for all $F\in \cF$ the $m$-deck of
$G$ determines $\hat s(F,G)$.
\end{lemma}
\begin{proof}
Let $t=\C{V(G)}-\C{V(F)}$; we use induction on $t$.  When $t\le n-m$, the value
$\hat s(F,G)$ is given in the hypothesis.  When $t> n-m$, each induced subgraph
of $G$ isomorphic to $F$ lies in a unique maximal $\cF$-subgraph $H$ of $G$.
Grouping the induced copies of $F$ according to the choices for the isomorphism
class of $H$ yields
$$s(F,G)=\sum_{H\in{\cF}} s(F,H)\hat s(H,G).$$
Since $\C{V(F)}< m$, we know $s(F,G)$ from the deck, and we know
$s(F,H)$ when $F$ and $H$ are known.  By the induction hypothesis, we know all
values of the form $\hat s(H,G)$ when $F$ is an induced subgraph of $H$ except
$\hat s(F,G)$.  Therefore, we can solve for $\hat s(F,G)$.
\end{proof}

Applying Lemma~\ref{counting} to determine all the maximal $\cF$-subgraphs
from the $m$-deck requires determining those with at least $m$ vertices.
This is often the difficult part.  We refer to the use of Lemma~\ref{counting}
to iteratively obtain the smaller maximal $\cF$-subgraphs (with multiplicity)
by discarding those contained in larger maximal $\cF$-subgraphs as the
{\bf exclusion argument}.  We then find $\hat s(F,G)$ for each $F\in\cF$ in
nonincreasing order of $\C{V(F)}$.

We will apply Lemma~\ref{counting} to several families of subgraphs.

\begin{definition}\label{brooms}
A \emph{baton} is a caterpillar obtained from a path by attaching some positive
number of leaves at both ends.  A \emph{triton} is obtained from a baton by
also appending a positive number of leaves at one internal vertex.
A \emph{$j$-baton} is a baton in which the original path has $j$ edges, and
$B_{j:a,b}$ is the $j$-baton with $j+a+b-1$ vertices having vertices of degrees
$a$ and $b$ (adjacent to the leaves and called the \emph{key vertices})
separated by distance $j$.  For a baton we require $a,b\ge2$.
A \emph{$j,j'$-triton} is a triton in which the distances from the
internal specified vertex to the other vertices where leaves have been added
are $j$ and $j'$.  Let $B_{j,j':a,b,c}$ denote the $j,j'$-triton having
vertices of degrees $a$, $b$, and $c$ (called the \emph{key vertices}) in
order as shown in Figure~\ref{tritonfig} ($a,c\ge2$); it has
$j+j'+a+b+c-3$ vertices.
\end{definition}

\begin{figure}[h]
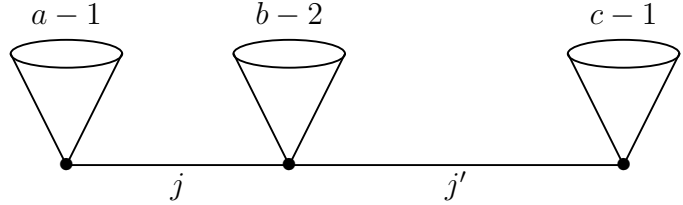

\begin{center}
\gpic{
\expandafter\ifx\csname graph\endcsname\relax
   \csname newbox\expandafter\endcsname\csname graph\endcsname
\fi
\ifx\graphtemp\undefined
  \csname newdimen\endcsname\graphtemp
\fi
\expandafter\setbox\csname graph\endcsname
 =\vtop{\vskip 0pt\hbox{%
    \graphtemp=.5ex
    \advance\graphtemp by 0.773in
    \rlap{\kern 0.292in\lower\graphtemp\hbox to 0pt{\hss $\bu$\hss}}%
    \graphtemp=.5ex
    \advance\graphtemp by 0.773in
    \rlap{\kern 1.458in\lower\graphtemp\hbox to 0pt{\hss $\bu$\hss}}%
    \graphtemp=.5ex
    \advance\graphtemp by 0.773in
    \rlap{\kern 3.208in\lower\graphtemp\hbox to 0pt{\hss $\bu$\hss}}%
    \special{pn 11}%
    \special{pa 292 773}%
    \special{pa 3208 773}%
    \special{fp}%
    \graphtemp=.5ex
    \advance\graphtemp by 0.890in
    \rlap{\kern 0.875in\lower\graphtemp\hbox to 0pt{\hss $j$\hss}}%
    \graphtemp=.5ex
    \advance\graphtemp by 0.890in
    \rlap{\kern 2.333in\lower\graphtemp\hbox to 0pt{\hss $j'$\hss}}%
    \special{ar 292 190 292 73 0 6.28319}%
    \special{ar 1458 190 292 73 0 6.28319}%
    \special{ar 3208 190 292 73 0 6.28319}%
    \special{pa 0 190}%
    \special{pa 292 773}%
    \special{fp}%
    \special{pa 292 773}%
    \special{pa 583 190}%
    \special{fp}%
    \special{pa 1167 190}%
    \special{pa 1458 773}%
    \special{fp}%
    \special{pa 1458 773}%
    \special{pa 1750 190}%
    \special{fp}%
    \special{pa 2917 190}%
    \special{pa 3208 773}%
    \special{fp}%
    \special{pa 3208 773}%
    \special{pa 3500 190}%
    \special{fp}%
    \graphtemp=.5ex
    \advance\graphtemp by 0.000in
    \rlap{\kern 0.292in\lower\graphtemp\hbox to 0pt{\hss $a-1$\hss}}%
    \graphtemp=.5ex
    \advance\graphtemp by 0.000in
    \rlap{\kern 1.458in\lower\graphtemp\hbox to 0pt{\hss $b-2$\hss}}%
    \graphtemp=.5ex
    \advance\graphtemp by 0.000in
    \rlap{\kern 3.208in\lower\graphtemp\hbox to 0pt{\hss $c-1$\hss}}%
    \hbox{\vrule depth0.890in width0pt height 0pt}%
    \kern 3.500in
  }%
}%
}
\vspace{-.5pc}
\caption{The $j,j'$-triton $B_{j,j':a,b,c}$\label{tritonfig}}
\end{center}
\end{figure}
\vspace{-1pc}

\begin{lemma}\label{jbaton}
Fix $j,j'\in\NN$, with $r$ being the number of nonleaf vertices in the unknown
caterpillar $G$.  If $\cD_m$ determines the maximal $j$-batons having at least
$m$ vertices (with multiplicity), then it determines the multiset of $r-j$
unordered degree pairs $\{d(v_i), d(v_{i+j})\}$ for $1\le i\le r-j$, without
specifying $i$ in any such pair.
Similarly, if $\cD_m$ determines the maximal $j,j'$-tritons having at
least $m$ vertices (with multiplicity), then it determines the multiset of
degree triples consisting of
$(d(v_i),d(v_{i+j}),d(v_{i+j+j'}))$ for $1\le i\le r-j-j'$ and
$(d(v_i),d(v_{i-j}),d(v_{i-j-j'}))$ for $j+j'+1\le i\le r$,
without knowing any $i$ or whether such a triple is increasing or decreasing
in subscripts.
\end{lemma}
\begin{proof}
In a caterpillar, every $j$-baton occurs in a unique maximal $j$-baton.
Thus we can apply Lemma~\ref{counting} when $\cF$ is the family of $j$-batons.
The graph induced by $\VEC vi{i+j}$ and the other neighbors of $v_i$ and
$v_{i+j}$ in $G$ is a maximal $j$-baton in $G$, and these are all
the maximal $j$-batons.  Given that we know $\hat s(F,G)$ for every $j$-baton
$F\in\cF$ with at least $m$ vertices, Lemma~\ref{counting} implies that
$\cD_m$ determines the multiset of maximal $j$-batons.  
Each maximal $j$-baton gives one of the desired degree pairs.

The same argument holds for $j,j'$-tritons, applying Lemma~\ref{counting}
if we know the maximal ones having at least $m$ vertices.  Note that when
$j\ne j'$ there are $2(r-j-j')$ maximal $j,j'$-tritons, but when $j=j'$ there
are only $r-j-j'$ corresponding subgraphs and we listed them twice, each triple
and its reverse.
\end{proof}

The idea of the next two lemmas, which are quite similar, is due to
Hunter~\cite{Hun}.

\begin{lemma}\label{levelpair}
Fix a value $q$ with $q\le(r-1)/2$.  If for each $j$ with $j\le q$ the multiset
of degree pairs $\{d(v_i),d(v_{i'})\}$ with $|i'-i|=j$ is known,
then for $1\le k\le q$ the level pair $\{d(v_k),d(v_\kb)\}$ is known.
If $q=\FL{(r-1)/2}$, then the remaining degrees provide the remaining level.
\end{lemma}
\begin{proof}
We apply the same argument iteratively for $k$ from $1$ to $q$.
Suppose that the pairs $\{d(v_i),d(v_\ib)\}$ are known for $1\le i<k$; the
assumption is vacuous when $k=1$.

When $k<i<\ov k$, the nonleaf vertex $v_i$ is a key vertex in two maximal
$k$-batons, but for $i$ outside that range $v_i$ is a key vertex in only
one maximal $k$-baton.  (When $i=k$ and $d(v_1)=2$, the maximal baton
containing $\la \VEC v0i\ra$ and all neighbors of $v_i$ and $v_0$ is a
$(k-1)$-baton but not a $k$-baton.)

Consider the multiset of nonleaf vertex degrees, with the degree of each
nonleaf vertex written twice.  For each maximal $k$-baton $B$, delete from this
multiset two entries that are the degrees of the key vertices in $B$.  Over the
entire process we delete $d(v_i)$ twice if $k<i<\ov k$ and once if $i$ is
outside that range.  From the remaining multiset, delete one copy of each entry
in $\{d(v_i),d(v_\ib)\}$ for $1\le i<k$, determined earlier.  This leaves only
$d(v_k)$ and $d(v_\kb)$, each listed twice and deleted once.

When $q=\FL{(r-1)/2}$ and we have done this for $k\le q$, only one or two
values (depending on the parity of $r$) remain in the multiset.  They are
$d(v_{(r+1)/2})$ or $\{d(v_{r/2}),d(v_{r/2+1})\}$.
\end{proof}

\begin{lemma}\label{triples}
Fix values $s$ and $q$ with $s<q\le(r-1)/2$.  If for each $j$ with $j\le s$ the
multiset of maximal $j$-batons is known, and for each pair $(j,j')$ with
$j+j'\le q$ the multiset of maximal $j,j'$-tritons is known, then
for all $k$ and $s$ with $2\le k+s\le q$ the following set of two ordered
degree pairs is known:
$$\{(d(v_k),d(v_{k+s})),(d(v_\kb),d(v_{\ov{k+s}}))\}.$$
\end{lemma}
\begin{proof}
The argument is similar to that of Lemma~\ref{levelpair}.  Fixing $s$, we
obtain the pair for $k$ iteratively, beginning with $k=1$.  Suppose that the
pairs $\{(d(v_i),d(v_{i+s})),(d(v_\ib),d(v_{\ov{i+s}}))\}$
are known for $1\le i<k$; the assumption is vacuous when $k=1$.

For each maximal $s$-baton $B$, add to a multiset $L$ two ordered pairs,
consisting of the degrees of the key vertices in $B$ written in both orders
(the two pairs may be the same).  These pairs have the form
$(d(v_i),d(v_{i+s}))$ or $(d(v_i),d(v_{i-s}))$.  We aim to cancel most of $L$,
leaving only the ordered pairs $(d(v_k),d(v_{k+s}))$ and
$(d(v_\kb),d(v_{\ov{k+s}}))$.  When $k<i<\kb-s$, the maximal $s$-baton
containing $\{v_i,v_{i+s}\}$ occurs in two maximal $k,s$-tritons,
with key vertices $(v_{i-k},v_i,v_{i+s})$ and $(v_{i+s+k},v_{i+s},v_i)$.
For each maximal $k,s$-triton $B$, delete from $L$ the two ordered degree pairs
at distance $s$ that arise from $B$.
(When $s=k$, we get four ordered pairs from $B$ to delete.)

The deletions leave in $L$ one copy each of the ordered pairs $(v_i,v_{i+s})$
with $i\le k$ and $(v_{i+s},v_i)$ with $i+s\ge\kb$.  We have previously
determined those with $i<k$ and $i+s>\kb$ (not knowing which come from which
end).  Deleting them leaves us with the two ordered pairs $(d(v_k),d(v_{k+s}))$
and $(d(v_\kb),d(v_{\ov{k+s}}))$, not knowing which is which.
\end{proof}

Say that the \emph{length} of $B_{j:a,b}$ is $j$,
and the \emph{length} of $B_{j,j':a,b,c}$ is $j+j'$.

\begin{theorem}\label{suffic}
Fix $q$ with $1\le q\le (r-1)/2$.  If the maximal batons and tritons in $G$
that have length at most $q$ and at least $\FL{(n+2)/2}$ vertices are known,
then the lists $(d(v_1),\ldots,d(v_q))$ and
$(d(v_{\ov1}),\ldots,d(v_{\qb}))$ (with specified subscripts) are known.
\end{theorem}
\begin{proof}
By Lemmas~\ref{jbaton}, \ref{levelpair}, and~\ref{triples},
we know the degree pairs $\{d(v_k),d(v_\kb)\}$ for $1\le k\le q$ and the
sets of two ordered pairs $\{(d(v_k),d(v_{k+s})),(d(v_\kb),d(v_{\ov{k+s}}))\}$
for $k+s\le q$.

If $d(v_k)=d(v_\kb)$ for all $k$ at most $q$, then there is nothing to prove.
Otherwise, from the level pairs $\{d(v_k),d(v_\kb)\}$ we know
the least index $h$ such that $d(v_h)\ne d(v_\hb)$, and it satisfies $h\le q$.
We may assume $d(v_h)<d(v_\hb)$ by symmetry.
Since we know the unordered pairs of \emph{ordered} pairs of the form
$(d(v_h),d(v_{h+s}))$ and $(d(v_\hb),d(v_{\ov{h+s}}))$ for $h+s\le q$ and
$d(v_h)\ne d(v_{\hb})$, we also know $d(v_{h+s})$ and $d(v_{\ov{h+s}})$.
For the levels below $h$, the two choices are the same.
\end{proof}

\vspace{-.5pc}
The difficulty in applying Theorem~\ref{suffic} is finding the maximal batons
and tritons having at least $\FL{(n+2)/2}$ vertices and length up to $(r-1)/2$.
If the hypotheses of Theorem~\ref{suffic} are known with $q=\CL{r/2}$, then
the caterpillar is known.  If we only know this for some smaller $q$, then we
can still determine the list of degrees if we know the middle part of the list.

\begin{lemma}\label{stitch}
Fix an integer $q$ with $q\le(r+1)/2$.  Suppose that the maximal batons with
length less than $q$ and the lists of degrees $(d(v_1),\ldots,d(v_{q-1}))$ and
$(d(v_{\ov1}),\ldots,d(v_{\ov{q-1}}))$ (with known subscripts) are determined
by the deck.  If also the ``middle portion'' $(d(v_{q}),\ldots,d(v_\qb))$ is
known up to reversal, then the unknown caterpillar is determined.
\end{lemma}
\begin{proof}
Let $W=\{\VEC v{q}{\ov{q}}\}$, and let $U$ be the set of vertices outside $W$.
We may assume $d(v_{q})\ne d(v_\qb)$, since otherwise we can move level $q$
into $U$ and apply the argument below with $q-1$ instead of $q$.  Let
$a=\max\{d(v_q),d(v_{\qb})\}$.

If $d(v_i)=d(v_{\ib})$ for all $v_i\in U$, then both choices for ordering $W$
yield the same caterpillar.  Hence we may let $h$ be the maximum index less
than $q$ such that $d(v_h)\ne d(v_\hb)$.  Let $j=q-h$, and let
$c=\max\{d(v_h),d(v_\hb)\}$.  By symmetry, we may label the indices so that
$c=d(v_h)$.
 
Let $L$ be the multiset of unordered degree pairs of the form
$\{d(v_i),d(v_{i+j})\}$, known from the maximal batons (here $j$ is fixed).
Delete from $L$ all members that arise when both vertices are in $W$ or both in
$U$; these pairs are known.  The pairs that remain in $L$ arise from a vertex
in $W$ and a vertex in $U$ in some level $i$ with $h\le i<q$.

Let $t$ be the number of copies of $\{a,c\}$ that remain in $L$ (note that $a$
may equal $c$).  If $t=0$, then $a=d(v_\qb)$, which completes the
reconstruction.

If $t\ge1$, then we eliminate from $L$ instances of $\{a,c\}$ at distance $j$ 
that do not arise as $\{d(v_h),d(v_q)\}$.  Such an instance involves a vertex
with level between $h$ and $q$.  For $h<i<q$, we have $d(v_i)=d(v_{\ib})$, so we
get the same number of such pairs no matter which of the two orderings of
$W$ is used.  Hence we know the number of copies of $\{a,c\}$ to eliminate.

\nobreak
We then check again whether $\{a,c\}$ remains in $L$.  If it does, then
$a=d(v_q)$.  Otherwise, $a=d(v_\qb)$.  This determines the order
on $W$ and completes the reconstruction.
\end{proof}

The first application of Lemma~\ref{stitch} is a slight reduction in the length
of the maximal tritons we need to know.

\begin{corollary}\label{length}
If the maximal batons with length at most $(r-1)/2$ and the tritons with length
at most $(r-2)/2$ are determined by the deck, then the unknown caterpillar is
reconstructible from the deck.
\end{corollary}
\begin{proof}
Let $q=\FL{r/2}$.  The middle portion of the spine between $v_q$ and
$v_{\qb}$ has only two vertices (when $r$ is even) or three vertices (when $r$
is odd).  By Lemma~\ref{levelpair} determining the level pairs, we thus know
$(\VEC vq{\qb})$ up to reversal.

By Theorem~\ref{suffic}, to determine the lists $(d(v_1),\ldots,d(v_{q-1}))$
and $(d(v_{\ov1}),\ldots,d(v_{\ov{q-1}}))$ we only need the maximal batons
and tritons with length at most $q-1$ that have at least $n-\ell$ vertices.
Lemma~\ref{stitch} will then complete the reconstruction.
\end{proof}

Next we further reduce the length of tritons needed to complete the
reconstruction.  Let $R$ denote the set $\{\VEC v1r\}$ of nonleaf vertices.

\begin{lem}\label{lastfew}\label{3-11}
For $n\ge48$, if the maximal batons with length at most $(r-1)/2$ are known and
the lists $(d(v_1),\ldots,d(v_{q-1}))$ and $(d(v_{r}),\ldots,d(v_{\ov{q-1}}))$
(with known subscripts) are known, where $q\ge(r-3)/2$, then the unknown
caterpillar is determined by the deck.
\end{lem}
\begin{proof}
We may assume $q=\FL{(r-2)/2}$.  Unknown between the lists
$(d(v_1),\ldots,d(v_{q-1})$ and $(d(v_{\ov1}),\ldots,d(v_{\ov{q-1}}))$
is the segment $(d(v_q),\ldots,d(v_{\qb}))$ in the middle consisting of four
vertex degrees (when $r$ is even) or five vertex degrees (when $r$ is odd).
Let $W=\{\VEC vq{\qb}\}$.

When we know all the short maximal batons, by Lemma~\ref{levelpair} we know the
level pairs.  Here each level pair $\{d(v_k),d(v_\kb)\}$ for $k<q$ has already
been allocated to $v_k$ and $v_{\kb}$ (and $d(v_{(r+1)/2}$ is known when $r$ is
odd).  We must allocate the remaining two level pairs.

Let $A=\{v_q,v_\qb\}$ and $B=\{v_{q+1},v_{\ov{q+1}}\}$.
Let $U$ denote the set of spine vertices outside $A\cup B$; the degrees of all
vertices in $U$ are known by hypothesis.  If equality holds for the two degrees
in the level pair from $A$ or in the level pair from $B$, then we know the
degrees of the vertices in $W$ in order, up to reversal.  In this case
Lemma~\ref{stitch} applies to complete the reconstruction.
Letting $\{a,a'\}=\{d(v)\st v\in A\}$ and $\{b,b'\}=\{d(v)\st v\in B\}$, we may
therefore assume $a\ne a'$ and $b\ne b'$.  Hence we may also assume (after
switching the names of $a$ and $a'$, if needed)
\begin{equation}\label{neq}
a\neq b'\neq b\neq a'\neq a.
\end{equation}
With this notation, the ordered pairs $(a,b)$ and $(a',b')$ are symmetric
in the argument.

\medskip
With $q=\FL{(r-2)/2}$, we have
$q=(r-3)/2$ when $r$ is odd, and $q=(r-2)/2$ when $r$ is even.
By symmetry, we may assume that the list of degrees of the middle
vertices from $v_q$ to $v_\qb$ on the spine is $a,b,c,b',a'$ (Type 1) or
$a',b,c,b',a$ (Type 2).  Here $c$ is an actual value when $r$ is odd,
but $c$ is nonexistent (and $b$ is next to $b'$) when $r$ is even.

For every short baton $H$, we know the number of copies of $H$ with both key
vertices in $U$, with both key vertices in $A$, with both key vertices in
$B$, and (when $r$ is odd) with $v_{(r+1)/2}$ as a key vertex.
We know the last category because (1) we know the level pairs and (2) the
distances from $v_{(r+1)/2}$ to the two vertices involved in a level pair are
the same.

These known copies of $H$ we call ``old''; the remaining copies are ``new''.
Under Type 1 or Type 2, the number of old copies of any short baton $H$ is
the same.  If for some short baton $H$ the number of new copies of $H$ is
different under Type 1 and Type 2, then we know the list of middle degrees
up to reversal, and Lemma~\ref{stitch} applies to complete the reconstruction.
Hence we may assume that for each short baton $H$ the number of new copies of
$H$ is the same for both types.

First consider $B_{1:a,b}$.  Type 1 has a new copy of $B_{1:a,b}$ with key
vertices $v_q$ and $v_{q+1}$.  Type 2 does not have such a baton intersecting
$B$.  Since \eqref{neq} guarantees $a',b'\notin \{a,b\}$, having a new copy of
$B_{1:a,b}$ in Type 2 requires $d(v_{\qb+1})=b$.
By the symmetry noted after~\eqref{neq}, considering new copies of
$B_{1:a',b'}$ yields $d(v_{q-1})=b'$.  Since $b'\neq b$, this fixes the
arrangement of $B$ with respect to $U$.  For the remainder of the argument, we
move the batons with both key vertices in $B\cup U$ into the ``old'' category
and restrict new batons to be precisely those having one key vertex in $A$
and the other in $B\cup U$.  At this point, up to reversal, the
central portion of degrees from $v_{q-1}$ to $v_{\qb+1}$ (with $c$ nonexistent
for even $r$) is
\begin{equation}\label{7tuple}
{b',a,b,c,b',a',b}\; \mbox{(Type 1)\quad or\quad}
{b',a',b,c,b',a,b}\; \mbox{(Type 2).}
\end{equation} 

\medskip
Let $p=\qb-q$, so $p=4$ when $r$ is odd and $p=3$ when $r$ is even.
The point is that the distance between the vertices in $A$ is $p$.
By induction on $i$ we now prove for $-1\le i\le \FL{(r-1)/2}$ the
following {\bf Claim:}
\begin{equation}\label{ieven}
d(v_{q-i})=b\; \mbox{and}\; d(v_{\qb+i})=b'\; \mbox{when $i\equiv-1\!\!\mod p$,}
\end{equation}
and
\begin{equation}\label{iodd}
d(v_{q-i})=b'\; \mbox{and}\; d(v_{\qb+i})=b\; \mbox{when $i\equiv1\!\!\mod p$.}
\end{equation}
The base cases $i=-1$ and $i=1$ hold in both cases described in~\eqref{7tuple}.
For the induction step, note that since
$i\leq (r-1)/2$, we have $q-(i-p)\ge1$, and symmetrically $\qb+i-p\le r$.
Hence the vertices $v_{q-i+p}$ and $v_{\qb+i-p}$ exist on the spine,
and the induction hypothesis applies to them.
For the details of the induction step, we consider the two congruence
classes for $i$ separately; it may help to refer to~\eqref{7tuple}.

First suppose $i\equiv-1\!\mod p$.  By~\eqref{ieven}, $d(v_{q-(i-p)})=b$.
Type 2 has a new copy of $B_{i:a,b}$ with key vertices $v_{\qb}$ and
$v_{q-(i-p)}$.  Since $a'\notin \{a,b\}$, one key vertex of
a new copy of $B_{i:a,b}$ in Type 1 must be $v_{q}$, so
its other key vertex must have degree $b$.  Among the two vertices in $R$ at
distance $i$ from $v_{q}$, the vertex $v_{\qb+i-p}$ has degree $b'$,
by~\eqref{ieven}.  Therefore, $d(v_{q-i})=b$.

Symmetrically, $d(v_{\qb+i-p})=b'$, by~\eqref{ieven}.  Thus Type 2 has a new
copy of $B_{i:a',b'}$ with key vertices $v_{q}$ and $v_{\qb+i-p}$.  One
key vertex of a new copy of $B_{i:a',b'}$ in Type 1 must be $v_{\qb}$.  Among
the two vertices in $R$ at distance $i$ from $v_{\qb}$, the vertex
$v_{q-(i-p)}$ has degree $b$, by~\eqref{ieven}.  Therefore, $d(v_{\qb+i})=b'$.
This finishes the induction step when $i\equiv-1\!\mod4$.

Now suppose $i\equiv1\!\mod p$.  We have $d(v_{q-(i-p)})=b'$, by~\eqref{iodd}.
Now Type 1 has a new copy of $B_{i:a',b'}$ with key vertices $v_{qb}$ and
$v_{q-(i-p)}$.  Since $a\notin \{a',b'\}$, one key vertex of a new copy of
$B_{i:a',b'}$ in Type 2 must be $v_{q}$.  The other key vertex has degree $b'$.
Among the two vertices in $R$ at distance $i$ from $v_{q}$, the vertex
$v_{\qb+i-p}$ has degree $b$, by~\eqref{ieven}.  Therefore, $d(v_{q-i})=b'$.

Symmetrically, $d(v_{\qb+i-p})=b$, by~\eqref{iodd}.  Thus Type 1 has a new
copy of $B_{i:a,b}$ with key vertices $v_{q}$ and $v_{\qb+i-p}$.  One key
vertex of a new copy of $B_{i:a,b}$ in Type 2 must be $v_{\qb}$.  Among the two
vertices in $R$ at distance $i$ from $v_{\qb}$, the vertex $v_{q-(i-p)}$ has
degree $b'$, by~\eqref{iodd}.  Therefore, $d(v_{\qb+i})=b$.  This completes
the proof of the claim.

\medskip
Under the assumption that in both structures Type 1 and Type 2 the number of
new copies of each short baton is the same, we have obtained for
$i$ equal to $\FL{(r-1)/2}$ or $\FL{(r-1)/2}-1$ (whichever is odd when $r$
is odd or whichever is not divisible by $3$ when $r$ is even) a vertex
$v_{q-i}$ with degree in $\{b,b'\}$.  When $r$ is odd and $q=(r-3)/2$, this
yields $q-i\le 0$, so there is no such vertex.  When $r$ is even and
$(r-2)/2$ is not divisible by $3$, again $q-i-0$ and there is no such vertex.
In these cases, the contradiction implies that somewhere along the way we
discover that only one of Type 1 and Type 2 for the central vertices can
generate the deck.

This leaves the case $3\mid (r-2)/2$ (that is, $r\equiv2\!\mod 6$),
where we must work much harder.  Because we are only guaranteed knowing the
maximal batons with length at most $(r-1)/2$, the last step we can perform in
the iteration is $i=(r-4)/2\equiv-1\!\mod3$, which yields
$(d(v_1),d(v_r))=(b,b')$.  The previous step $i=(r-6)/2\equiv1\!\mod3$
told us $(d(v_2),d(v_{r-1}))=(b',b)$.

We know all old batons of all lengths, since we know the entire caterpillar
except for the assignment of $\{a,a'\}$ to $\{v_q,v_\qb\}$.  After discarding
old batons, the only new batons with length $r/2$ are those with key vertices
$\{v_q,v_{r-1}\}$ and $\{v_\qb,v_2\}$.  In Type 1 the new maximal batons of
length $r/2$ are $B_{r/2:a,b}$ and $B_{r/2:a',b'}$.  In Type 2 we instead have
$B_{r/2:a',b}$ and $B_{r/2:a,b'}$.  It suffices to determine which of these
pairs of batons occurs.

Let $b_0=\min\{b,b'\}$ and $b_1=\max\{b,b'\}$.
Let $a_0=\min\{b,b'\}$ and $a_1=\max\{b,b'\}$.
We look for $B_{r/2:a_0+1,b_0+1}$ as a new baton.  If it occurs, then
the correct choice of the unknown caterpillar is the Type containing
$B_{r/2:a_1,b_1}$ as a new baton.  Otherwise, the caterpillar is the Type
not having $B_{r/2:a_1,b_1}$ as a new baton.

It suffices to show that $B_{r/2:a_0+1,b_0+1}$ is small enough to fit in a card.
The number of vertices in $B_{r/2:a_0+1,b_0+1}$ is $r/2+a_0+b_0+1$.
Thus we want to prove $a_0+b_0\le (n-r+1)/2$.  We have at least $(r+1)/3$
vertices each with degrees $b$ and $b'$, plus at least one vertex each with
degrees $a$ and $a'$.  From the equation
$\SE i1r(d(v_i)-2)=n-r-2$ we obtain
\begin{equation}\label{a0b0}
2(a_0-2)+2\FR{r+1}3(b_0-2)\le n-r-3-\FR{r+1}3.
\end{equation}
When either $b_0$ or $a_0$ reaches its maximum under~\eqref{a0b0}, the other
is restricted to value $2$.  Since we want to bound $a_0+b_0$, the more
dangerous extreme is when $b_0=2$ and $a_0$ is maximized.  In this
case~\eqref{a0b0} requires $a_0\le (n-r+1)/2-(r+1)/6$.  With $b_0=2$,
we now have $b_0+a_0\le(n-r+1)/2$ as long as $2-(r+1)/6\le0$, which holds 
when $r\ge11$.  This is guaranteed by $n\ge27$, since $r\ge(n-5)/2$.

Now consider the non-extreme choices for $a_0$ and $b_0$.  By~\eqref{a0b0},
every increase of $1$ in $b_0$ (starting from $b_0=2$) requires $a_0$ to
decrease by $(r+1)/3$.  That amount is an integer and is at least $1$.  Hence
all other choices of $(a_0,b_0)$ also result in $B_{r/2:a_0+1,b_0+1}$ fitting
in a card.

Since $r\equiv2\mod6$, the only case not covered is $r=8$ and $b_0=2$ (the gain
discussed above will be good enough when $b_0\ge3$).
The degree list in order along the spine in this remaining case is
$(b,b',a,b,b',a',b,b')$ (Type 1) or $(b,b',a',b,b',a,b,b')$ (Type 2).  We fail
only if $a_0=(n-r-2)/2$ (so $n$ is even), with $\{b,b'\}=\{2,3\}$ and
$\{a,a'\}=\{(n-10)/2,(n-8)/2\}$.  By symmetry, we may assume $b=2$.

The baton $B_{r/2:a_0+1,b_0+1}$ that we were seeking has $(n+4)/2$ vertices.
This is truly too big only when $(n+4)/2>n-\ell$, which is equivalent to
$n<2\ell+4$.  Since $n$ here is even, we have $n=2\ell+2$.
Also, $(n-10)/2=a_0\ge2$, so $n\ge14$.
For the remaining cases, let $(n,\ell)=(14+2t,6+t)$, where $t\ge0$.
We seek a subgraph with $8+t$ vertices whose multiplicity differs in the
decks of the Type 1 and Type 2 caterpillars.

The remaining cases are shown in Figure~\ref{r2mod6} (recall $q=3$ and $\qb=6$).
We count copies of the triton $B_{2,2:2,3,2+t}$, which has $8+t$ vertices (when
$t=0$, it consists of a $7$-vertex path plus one pendant edge at the center).
In the display below, we list the number of copies according to which vertex 
plays the role of the central key vertex.  Although $\#B_{2,2:2,3,2+t}$ depends
on $t$, in each case the Type 2 caterpillar has more copies of
$B_{2,2:2,3,2+t}$ than the Type 1 caterpillar does.
Hence the deck distinguishes the two Types.
\end{proof}
\begin{center}
\begin{tabular}{c c c c c c c c c}
 & &Type 1& & & & & Type 2 \\
 &$v_3$&$v_5$&$v_6$&  &  &$v_3$&$v_5$&$v_6$\\
$t=0$&0&1&2&~\qquad~&$t=0$&2&2&0\\
$t=1$&1&1&2& &$t=1$&2&2&1\\     
$t>1$&0&1&0& &$t>1$&0&t+1&0
\end{tabular}
\end{center}

\begin{figure}
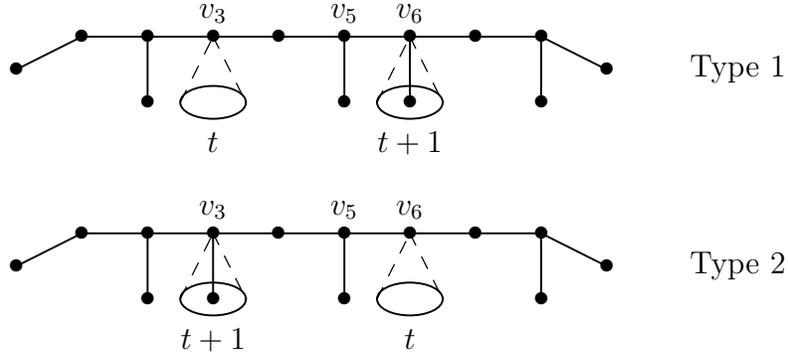

\begin{center}
\gpic{
\expandafter\ifx\csname graph\endcsname\relax
   \csname newbox\expandafter\endcsname\csname graph\endcsname
\fi
\ifx\graphtemp\undefined
  \csname newdimen\endcsname\graphtemp
\fi
\expandafter\setbox\csname graph\endcsname
 =\vtop{\vskip 0pt\hbox{%
    \graphtemp=.5ex
    \advance\graphtemp by 0.309in
    \rlap{\kern 0.086in\lower\graphtemp\hbox to 0pt{\hss $\bu$\hss}}%
    \graphtemp=.5ex
    \advance\graphtemp by 0.137in
    \rlap{\kern 0.429in\lower\graphtemp\hbox to 0pt{\hss $\bu$\hss}}%
    \graphtemp=.5ex
    \advance\graphtemp by 0.137in
    \rlap{\kern 0.773in\lower\graphtemp\hbox to 0pt{\hss $\bu$\hss}}%
    \graphtemp=.5ex
    \advance\graphtemp by 0.481in
    \rlap{\kern 0.773in\lower\graphtemp\hbox to 0pt{\hss $\bu$\hss}}%
    \graphtemp=.5ex
    \advance\graphtemp by 0.137in
    \rlap{\kern 1.116in\lower\graphtemp\hbox to 0pt{\hss $\bu$\hss}}%
    \graphtemp=.5ex
    \advance\graphtemp by 0.137in
    \rlap{\kern 1.459in\lower\graphtemp\hbox to 0pt{\hss $\bu$\hss}}%
    \special{pn 11}%
    \special{pa 86 309}%
    \special{pa 429 137}%
    \special{fp}%
    \special{pa 429 137}%
    \special{pa 1803 137}%
    \special{fp}%
    \special{pa 773 137}%
    \special{pa 773 481}%
    \special{fp}%
    \graphtemp=.5ex
    \advance\graphtemp by 0.000in
    \rlap{\kern 1.116in\lower\graphtemp\hbox to 0pt{\hss $v_3$\hss}}%
    \graphtemp=.5ex
    \advance\graphtemp by 0.000in
    \rlap{\kern 2.146in\lower\graphtemp\hbox to 0pt{\hss $v_6$\hss}}%
    \graphtemp=.5ex
    \advance\graphtemp by 0.309in
    \rlap{\kern 3.863in\lower\graphtemp\hbox to 0pt{\hss Type 1\hss}}%
    \graphtemp=.5ex
    \advance\graphtemp by 0.481in
    \rlap{\kern 1.803in\lower\graphtemp\hbox to 0pt{\hss $\bu$\hss}}%
    \graphtemp=.5ex
    \advance\graphtemp by 0.137in
    \rlap{\kern 1.803in\lower\graphtemp\hbox to 0pt{\hss $\bu$\hss}}%
    \graphtemp=.5ex
    \advance\graphtemp by 0.481in
    \rlap{\kern 2.146in\lower\graphtemp\hbox to 0pt{\hss $\bu$\hss}}%
    \graphtemp=.5ex
    \advance\graphtemp by 0.137in
    \rlap{\kern 2.146in\lower\graphtemp\hbox to 0pt{\hss $\bu$\hss}}%
    \graphtemp=.5ex
    \advance\graphtemp by 0.137in
    \rlap{\kern 2.489in\lower\graphtemp\hbox to 0pt{\hss $\bu$\hss}}%
    \graphtemp=.5ex
    \advance\graphtemp by 0.137in
    \rlap{\kern 2.833in\lower\graphtemp\hbox to 0pt{\hss $\bu$\hss}}%
    \graphtemp=.5ex
    \advance\graphtemp by 0.481in
    \rlap{\kern 2.833in\lower\graphtemp\hbox to 0pt{\hss $\bu$\hss}}%
    \graphtemp=.5ex
    \advance\graphtemp by 0.309in
    \rlap{\kern 3.176in\lower\graphtemp\hbox to 0pt{\hss $\bu$\hss}}%
    \special{pa 1803 481}%
    \special{pa 1803 137}%
    \special{fp}%
    \special{pa 1803 137}%
    \special{pa 2833 137}%
    \special{fp}%
    \special{pa 2833 137}%
    \special{pa 3176 309}%
    \special{fp}%
    \special{pa 2146 481}%
    \special{pa 2146 137}%
    \special{fp}%
    \special{pa 2833 137}%
    \special{pa 2833 481}%
    \special{fp}%
    \special{ar 1116 481 172 86 0 6.28319}%
    \special{ar 2146 481 172 86 0 6.28319}%
    \special{pn 8}%
    \special{pa 944 481}%
    \special{pa 1116 137}%
    \special{pa 1288 481}%
    \special{da 0.069}%
    \special{pa 1974 481}%
    \special{pa 2146 137}%
    \special{pa 2318 481}%
    \special{da 0.069}%
    \graphtemp=.5ex
    \advance\graphtemp by 0.000in
    \rlap{\kern 1.803in\lower\graphtemp\hbox to 0pt{\hss $v_5$\hss}}%
    \graphtemp=.5ex
    \advance\graphtemp by 0.704in
    \rlap{\kern 1.116in\lower\graphtemp\hbox to 0pt{\hss $t$\hss}}%
    \graphtemp=.5ex
    \advance\graphtemp by 0.704in
    \rlap{\kern 2.146in\lower\graphtemp\hbox to 0pt{\hss $t+1$\hss}}%
    \graphtemp=.5ex
    \advance\graphtemp by 1.339in
    \rlap{\kern 0.086in\lower\graphtemp\hbox to 0pt{\hss $\bu$\hss}}%
    \graphtemp=.5ex
    \advance\graphtemp by 1.167in
    \rlap{\kern 0.429in\lower\graphtemp\hbox to 0pt{\hss $\bu$\hss}}%
    \graphtemp=.5ex
    \advance\graphtemp by 1.167in
    \rlap{\kern 0.773in\lower\graphtemp\hbox to 0pt{\hss $\bu$\hss}}%
    \graphtemp=.5ex
    \advance\graphtemp by 1.511in
    \rlap{\kern 0.773in\lower\graphtemp\hbox to 0pt{\hss $\bu$\hss}}%
    \graphtemp=.5ex
    \advance\graphtemp by 1.167in
    \rlap{\kern 1.116in\lower\graphtemp\hbox to 0pt{\hss $\bu$\hss}}%
    \graphtemp=.5ex
    \advance\graphtemp by 1.167in
    \rlap{\kern 1.459in\lower\graphtemp\hbox to 0pt{\hss $\bu$\hss}}%
    \graphtemp=.5ex
    \advance\graphtemp by 1.511in
    \rlap{\kern 1.116in\lower\graphtemp\hbox to 0pt{\hss $\bu$\hss}}%
    \special{pn 11}%
    \special{pa 86 1339}%
    \special{pa 429 1167}%
    \special{fp}%
    \special{pa 429 1167}%
    \special{pa 1803 1167}%
    \special{fp}%
    \special{pa 773 1167}%
    \special{pa 773 1511}%
    \special{fp}%
    \special{pa 1116 1167}%
    \special{pa 1116 1511}%
    \special{fp}%
    \graphtemp=.5ex
    \advance\graphtemp by 1.030in
    \rlap{\kern 1.116in\lower\graphtemp\hbox to 0pt{\hss $v_3$\hss}}%
    \graphtemp=.5ex
    \advance\graphtemp by 1.030in
    \rlap{\kern 2.146in\lower\graphtemp\hbox to 0pt{\hss $v_6$\hss}}%
    \graphtemp=.5ex
    \advance\graphtemp by 1.339in
    \rlap{\kern 3.863in\lower\graphtemp\hbox to 0pt{\hss Type 2\hss}}%
    \graphtemp=.5ex
    \advance\graphtemp by 1.511in
    \rlap{\kern 1.803in\lower\graphtemp\hbox to 0pt{\hss $\bu$\hss}}%
    \graphtemp=.5ex
    \advance\graphtemp by 1.167in
    \rlap{\kern 1.803in\lower\graphtemp\hbox to 0pt{\hss $\bu$\hss}}%
    \graphtemp=.5ex
    \advance\graphtemp by 1.167in
    \rlap{\kern 2.146in\lower\graphtemp\hbox to 0pt{\hss $\bu$\hss}}%
    \graphtemp=.5ex
    \advance\graphtemp by 1.167in
    \rlap{\kern 2.489in\lower\graphtemp\hbox to 0pt{\hss $\bu$\hss}}%
    \graphtemp=.5ex
    \advance\graphtemp by 1.167in
    \rlap{\kern 2.833in\lower\graphtemp\hbox to 0pt{\hss $\bu$\hss}}%
    \graphtemp=.5ex
    \advance\graphtemp by 1.511in
    \rlap{\kern 2.833in\lower\graphtemp\hbox to 0pt{\hss $\bu$\hss}}%
    \graphtemp=.5ex
    \advance\graphtemp by 1.339in
    \rlap{\kern 3.176in\lower\graphtemp\hbox to 0pt{\hss $\bu$\hss}}%
    \special{pa 1803 1511}%
    \special{pa 1803 1167}%
    \special{fp}%
    \special{pa 1803 1167}%
    \special{pa 2833 1167}%
    \special{fp}%
    \special{pa 2833 1167}%
    \special{pa 3176 1339}%
    \special{fp}%
    \special{pa 2833 1167}%
    \special{pa 2833 1511}%
    \special{fp}%
    \special{ar 1116 1511 172 86 0 6.28319}%
    \special{ar 2146 1511 172 86 0 6.28319}%
    \special{pn 8}%
    \special{pa 944 1511}%
    \special{pa 1116 1167}%
    \special{pa 1288 1511}%
    \special{da 0.069}%
    \special{pa 1974 1511}%
    \special{pa 2146 1167}%
    \special{pa 2318 1511}%
    \special{da 0.069}%
    \graphtemp=.5ex
    \advance\graphtemp by 1.030in
    \rlap{\kern 1.803in\lower\graphtemp\hbox to 0pt{\hss $v_5$\hss}}%
    \graphtemp=.5ex
    \advance\graphtemp by 1.734in
    \rlap{\kern 1.116in\lower\graphtemp\hbox to 0pt{\hss $t+1$\hss}}%
    \graphtemp=.5ex
    \advance\graphtemp by 1.734in
    \rlap{\kern 2.146in\lower\graphtemp\hbox to 0pt{\hss $t$\hss}}%
    \hbox{\vrule depth1.734in width0pt height 0pt}%
    \kern 4.000in
  }%
}%
}
\vspace{-.5pc}
\caption{Caterpillars with $r=8$ \label{r2mod6}}
\end{center}
\end{figure}

\begin{remark}\label{btplan}
At this point we summarize what the results of this section imply about the
goal of proving that $n$-vertex caterpillars are reconstructible from the
multiset of subgraphs with at most $(n+2)/2$ vertices.  We will consider cases
for the degree list; we know when these cases occur because we know the degree
list.

\smallskip
\begin{narrower}

\noindent
If in a particular class of caterpillars the deck determines all maximal
batons with length at most $(r-1)/2$ and all maximal tritons with length at
most $(r-4)/2$, then caterpillars in that class are $\ell$-reconstructible.

\end{narrower}
\smallskip

\noindent
To summarize how this works, the maximal batons give us the degree pairs for
each fixed distance between the vertices on the spine.  The distance pairs then
give us the level pairs as in Lemma~\ref{levelpair}.  Incorporating the
knowledge of the tritons then gives us the lists $(d(v_1),\ldots,d(v_q))$ and
$(d(v_{\ov1}),\ldots,d(v_\qb))$ for $q=\FL{(r-4)/2}$, as in
Theorem~\ref{suffic}.  Finally, we apply Lemma~\ref{lastfew} to stitch the
lists together.

To apply this method, we need to determine the short maximal batons and
shorter maximal tritons, where we define a baton to be {\it short} if it
has length at most $(r-1)/2$ and define a triton to be {\it shorter} if it
has length at most $(r-4)/2$.

When there is a unique vertex of maximum degree, the technique is different.
We will be able to determine exactly where on the spine that vertex is and
then use it to find the other vertex degrees.  See Section~\ref{1max}.

In these arguments it will also be helpful to have at least four branch
vertices, in order to obtain better upper bounds on the high degrees.
For the case where $G$ has at most three branch vertices, see
Section~\ref{sec:3branch}.
\end{remark}

\section{Short Batons when $d_1=d_2\ge d_4\ge3$}\label{2maxbaton}

The first step in applying the plan in Remark~\ref{btplan} is to determine the
short maximal batons.  In order to obtain the desired threshold, we must
access only subgraphs with at most $(n+2)/2$ vertices (these will automatically
have at most $(n+1)/2$ vertices when $n$ is odd, since the number of vertices
is an integer.  In this section we find the short maximal batons when there are
at least two vertices of maximum degree and at least four branch vertices,
so throughout this section we assume $d_1=d_2\ge d_4\ge3$.


We first summarize what we know and introduce more notation.

\begin{remark}\label{batalph}
We have the deck $\DD$ of an $n$-vertex tree $G$, the multiset of all induced
subgraphs with at most $(n+2)/2$ vertices.  With $n\ge12$, we know that $G$ is
a caterpillar.  We know the nonincreasing degree list $\VEC d1n$ with $r$
nonleaves and may assume the case $r\ge (n-5)/2$.  Let $R$ be the set of
nonleaf vertices; index $R$ as $\VEC v1r$ in order along the spine and also as
$\VEC w1r$ with $d(w_i)=d_i$.

The baton $B_{j:a,b}$ is {\it short} if $j\le (r-1)/2$.
Throughout this section, let $\alpha=\FL{(n-r+1)/4}$.
The number of vertices in a short baton whose key vertices have degree at most
$\alpha+1$ is at most $(r-1)/2+2\alpha+1$, which is at most $(n+2)/2$.
Thus {\it all short batons with maximum degree at most $\alpha+1$
fit in cards, so we can count them.}

By Lemma~\ref{counting}, to determine all the short maximal batons it suffices
to determine all the short maximal batons with at least $\FL{(n+2)/2}$ vertices.
\end{remark}

\begin{lemma}\label{d3alpha}
If $d_1=d_2$ and $d_4>2$ and $d_3\leq \alpha$, then the deck determines
all short batons.
\end{lemma}
\begin{proof}
If $d_2\le \alpha$, then by Remark~\ref{batalph} every short baton fits in a
card, which suffices.  Hence we may assume $d_1=d_2>\alpha\geq d_3$.
By Remark~\ref{batalph}, we see all copies of $B_{j:\alpha+1,\alpha+1}$
Every copy of $B_{j:\alpha+1,\alpha+1}$ in the deck (if any exist) is contained
in the unique copy of $B_{j:d_1,d_1}$ in $G$, and $B_{j:d_1,d_1}$ contains
${d_1-1\choose \alpha}^2$ copies of $B_{j:\alpha+1,\alpha+1}$.  Thus we know
the number of copies of $B_{j:d_1,d_1}$ when $j\leq (r-1)/2$ (it is always $0$
or $1$).

If $\#B_{j:d_1,d_1}=1$ with $j\le(r-1)/2$, then the unique copy of
$B_{j:d_1,d_1}$ in $G$ contains ${d_1-1\choose a-1}{d_1-1\choose b-1}$ copies
of $B_{j:a,b}$, for each $a$ and $b$ with $2\leq a,b\leq \alpha+1$.  Delete
this many copies of $B_{j:a,b}$ from the deck; they do not contribute to the
count of maximal batons.

Each remaining copy of $B_{j:\alpha+1,d_3}$ occurs in some copy of
$B_{j:d_1,d_3}$ in $G$, and each copy of $B_{j:d_1,d_3}$ contains
${d_1-1\choose \alpha}$ copies of $B_{j:\alpha+1,d_3}$.  This tells us the
number of maximal batons in $G$ that are copies of $B_{j:d_1,d_3}$,
for $j\leq (r-1)/2$.  As in the previous paragraph, when $G$ has $q_j$ copies
of $B_{j:d_1,d_3}$ as maximal batons, we delete from the remaining deck
$q_j {d_1-1\choose a-1}{d_3-1\choose b-1}$ copies of $B_{j:a,b}$, for
$2\leq a,b\leq \alpha+1$.

Every remaining short baton
lies in a maximal baton (of the same length) having a key
vertex with degree at most $\min\{d_4,d_3-1\}$, and its maximum degree is at
most $d_1$.  Letting $N$ be its number of vertices, we have
\begin{equation}\label{d2}
N\le\frac{r-1}{2}+d_1+\min\{d_4,d_3-1\}-1\leq\frac{r+1}{2}+(d_2+d_4-2).
\end{equation}
Since $d_2=d_1$ and $d_4\le d_3$, \eqref{d'} yields
$d_2-2+d_4-2\le (n-r-2)/2$.
If $d_3>d_4$, then this improves to $d_2-2+d_4-2\le(n-r-3)/2$,
and then $N\leq(r+1)/2+(n-r+1)/2=(n+2)/2$. Otherwise, $d_3-1=d_4-1$,
which improves the upper bound in~\eqref{d2} by $1$, yielding
$$
N\leq \frac{r+1}{2}+(d_2+d_4-3)\leq\frac{r+1}{2}+\frac{n-r}{2}=\frac{n+1}{2}. $$
In both cases, all remaining short maximal batons have at most
$(n+2)/2$ vertices, and Lemma~\ref{counting} applies.
\end{proof}

Let $W_k=\{w_1,\ldots,w_k\}$ ($k$ specified vertices of highest degrees).
\smallskip

\begin{lemma}\label{d4alpha}
If $d_1=d_2$ and $2<d_4\le\alpha\le d_3-1$, then the deck determines
all short batons.
\end{lemma}
\begin{proof}
Our main task is to obtain the short maximal batons whose key vertices
are both in $W_3$, after which an exclusion argument determines the other
short maximal batons.

If $d_2=d_3$, then any copy of $B_{j:\alpha+1,\alpha+1}$ in the deck lies in a
copy of $B_{j:d_1,d_1}$ in $G$, of which there are at most two.
Each copy of $B_{j:d_1,d_1}$ contains ${d_1-1\choose \alpha}^2$ 
copies of $B_{j:\alpha+1,\alpha+1}$.  Thus the number of copies of 
$B_{j:\alpha+1,\alpha+1}$ in the deck tells us the number of 
copies of $B_{j:d_1,d_1}$ in $G$.  Thus we know the numbers of all
short maximal batons having both key vertices in $W_3$.

Now suppose $d_2>d_3$.  Since $d_3>d_4$, in this case only one vertex has
degree $d_3$.  Each copy of $B_{j:\alpha+1,\alpha+1}$ in the deck (if exists)
is contained either in a copy of $B_{j:d_1,d_1}$ (there is at most one) or in a
copy of $B_{j:d_1,d_3}$ (there are at most two).  A copy of $B_{j:d_1,d_1}$ (if
it exists) contains ${d_1-1\choose \alpha}^2$ copies of
$B_{j:\alpha+1,\alpha+1}$, and each copy of $B_{j:d_1,d_3}$ contains
${d_1-1\choose \alpha}{d_3-1\choose \alpha}$ copies of
$B_{j:\alpha+1,\alpha+1}$.  So, if over all $j$ with $j\le(r-1)/2$ the total
number of short batons in $\DD$ with key vertices of degree $\alpha+1$
is ${d_1-1\choose \alpha}^2+2{d_1-1\choose \alpha}{d_3-1\choose \alpha}$, then
we know all short batons whose key vertices are both in $W_3$.

We also know these batons if the total is
${d_1-1\choose \alpha}^2+{d_1-1\choose \alpha}{d_3-1\choose \alpha}$
or ${d_1-1\choose \alpha}{d_3-1\choose \alpha}$.  The only confusion is
when the total equals both ${d_1-1\choose \alpha}^2$ and
$2{d_1-1\choose \alpha}{d_3-1\choose \alpha}$, which requires
${d_1-1\choose \alpha}=2{d_3-1\choose \alpha}$.  If this happens and
$d_4<\alpha$, then we consider copies of $B_{j:\alpha,\alpha}$, which now must
be contained in a copy of $B_{j:d_1,d_1}$-baton or $B_{j:d_1,d_3}$ in $G$.
Since ${d_1-1\choose \alpha}\big/ {d_1-1\choose \alpha-1}\neq 
{d_3-1\choose \alpha}\big/ {d_3-1\choose \alpha-1}$, this distinguishes the
two possibilities.

On the other hand, if $d_4=\alpha$, then $d_1>d_3>\alpha$.
Now, since $d_1\ge\alpha+2$,~\eqref{d'} yields
$$
n-r-2\geq \sum_{i=1}^4(d_i-2)\geq 2d_1+2\alpha+1-8\ge 4\alpha-3\ge n-r-5.
$$
Keeping $\sum_{i=1}^4(d_i-2)\le n-r-2$ thus requires 
$d_1\leq \alpha+3$.  If ambiguity occurs, then
$$
2={d_1-1\choose \alpha}\Big/ {d_3-1\choose \alpha}\geq {d_1-1\choose \alpha}\Big/ {d_1-2\choose \alpha}
=\frac{d_1-1}{d_1-1-\alpha}\geq\frac{\alpha+2}{2}.
$$
Since $\alpha=d_4\geq 3$, we have $(\alpha+2)/2>2$, so the ambiguity does not
occur.

Thus in all cases we know all short maximal batons whose key vertices
lie in $W_3$.
More simply than in Lemma~\ref{d3alpha} (since here $d_4<d_3$), we derive
that every short baton in $G$ with at most one key vertex in $W_3$ has at most
$(n+2)/2$ vertices.  Again Lemma~\ref{counting} applies.
\end{proof}

\begin{lemma}\label{d4alpha+}
If $d_1=d_2$ and $d_4\ge\max\{3,\alpha+1\}$ with $2\alpha+1\le (n-r+1)/2$, then
the deck determines all short batons.
\end{lemma}
\begin{proof}
When $j\le(r-1)/2$, the number of vertices in a copy of
$B_{j:\alpha+2,\alpha+1}$ is at most $(r-1)/2+2\alpha+2$, which in the
specified case is at most $(n+2)/2$.  This is small enough to fit in cards,
so we see all such batons.

If $\sum_{i=1}^4 d_i\geq 4\alpha+9$, then by~\eqref{d'} we have
$$
n-r-2\ge\sum_{i=1}^4(d_i-2)\ge4\alpha+1=4\FL{\frac{n-r+1}{4}}+1\geq n-r-2+1,
$$
a contradiction.  Thus $\sum_{i=1}^4 d_i\leq 4\alpha+8$.  Since
$d_4\ge\alpha+1$, either $\alpha+2\ge d_1\ge d_4\ge\alpha+1$ or
$\alpha+3=d_1=d_2=d_3+2=d_4+2$.  Note also that when $d_4=\alpha+1$ we have
$\alpha\ge2$, since we have assumed $d_4>2$.
  
Suppose first that  $\alpha+3=d_1=d_2=d_3+2=d_4+2$.   Each copy of
$B_{j:\alpha+2,\alpha+1}$ lies either in the unique copy of $B_{j:d_1,d_1}$ in
$G$ or in some copy of $B_{j:d_1,\alpha+1}$ that is a maximal baton in $G$.
The copy of $B_{j:d_1,d_1}$ (if it exists) contains 
$2{d_1-1\choose \alpha+1}{d_1-1\choose \alpha}$ copies of
$B_{j:\alpha+2,\alpha+1}$, and every copy of $B_{j:d_1,d_3}$ that is a maximal
baton in $G$  contains ${d_1-1\choose \alpha+1}$ copies of
$B_{j:\alpha+2,\alpha+1}$.  Since for a given
$j$, each of $w_1$ and $w_2$ is in at most two copies of
$B_{j:d_1,d_3}$, the total number of such batons in $G$ is at most
$4$.  On the other hand, $2{d_1-1\choose \alpha}\geq 2(\alpha+1)\geq 6$,
since $\alpha\ge2$.  It follows
that the number of copies of $B_{j:\alpha+2,\alpha+1}$ in $\DD$ uniquely
determines the numbers of maximal batons in $G$ that are copies of
$B_{j:d_1,d_2}$ and $B_{j:d_1,d_3}$.

Thus, we know all short batons in $G$ with both key vertices having degree
at least $d_4$.  The number of vertices in a short baton in $G$ having at least
one key vertex of degree less than $d_4$ is at most
$d_1+d_4-2+(r-1)/2$, and we compute
$$
d_1+d_4-2+\frac{r-1}{2}\leq 2\alpha+2+\FR{r-1}2\le
\FR{n-r+1}2+1+\FR{r-1}2=\FR{n+2}2.
$$ 
Thus such a baton has at most $\FL{(n+2)/2}$ vertices, which 
is small enough to fit in a card.

We now know all short maximal batons with at least $\FL{(n+2)/2}$ vertices,
which by Lemma~\ref{counting} is enough to determine all short maximal batons
in $G$.

\medskip
Now suppose $\alpha+2\geq d_1\geq d_4\geq\alpha+1$.  If $d_1=d_4=\alpha+2$,
then $d_5=2$, and no copy of $B_{j:\alpha+2,\alpha+1}$ is a maximal baton in
$G$.  Thus $\#B_{j:\alpha+2,\alpha+1}$ tells us all maximal batons with two
branch vertices, which suffices.

Hence we may assume $\alpha+1=d_4\ge3$, so $2(\alpha+1)\ge6$.  Each copy of
$B_{j:\alpha+2,\alpha+1}$ in $\DD$ lies in a copy of $B_{j:d_1,d_1}$ in $G$ or
is itself a copy of $B_{j:d_1,d_4}$ that is a maximal baton in $G$.
Now every copy of $B_{j:d_1,d_1}$ in $G$ contains $2(\alpha+1)$
copies of $B_{j:\alpha+2,\alpha+1}$.

If $d_1=d_3=\alpha+2>\alpha+1=d_4$, then $d_6=2$ and $d_5\leq 3$.
In this case only $w_4$ and $w_5$ can have degree $d_4$, so at most four
copies of $B_{j:d_1,d_4}$ can be maximal batons in $G$, while every copy of
$B_{j:d_1,d_1}$ in $G$ contains $2(\alpha+1)$ copies of $B_{j:d_1,d_4}$ as
cards in $\DD$.  Since $2(\alpha+1)\geq 6$, the number of copies of
$B_{j,\alpha+2,\alpha+1}$ in $\DD$ tells us how many copies of $B_{j:d_1,d_1}$
and $B_{j:d_1,d_4}$ occur as maximal batons in $G$.  All other short batons fit
in cards.

If $d_1=d_2=\alpha+2>\alpha+1=d_3=d_4$, then only $w_1$ and $w_2$ have degree
$\alpha+2$.  They are the only candidates for key vertices of degree $d_1$,
so again $G$ can only have four maximal batons that are copies of
$B_{j:d_1,d_4}$.  Still any copy of $B_{j:d_1,d_1}$ in $G$ (there is at most
one) contains $2(\alpha+1)$ copies of $B_{j:d_1,d_4}$.  Since
$2(\alpha+1)\geq 6$, again we know how many maximal batons in $G$ are copies of
$B_{j:d_1,d_1}$ or $B_{j:d_1,d_4}$.  Again, all other short batons fit in
cards.

Finally, when $d_1=d_2=d_3=d_4=\alpha+1$, all short batons fit in cards.
\end{proof}

\begin{lemma}\label{d4alpha++}
If $d_1=d_2$ and $d_4\ge\alpha+1$ with $2\alpha+1\ge (n-r+2)/2$, then the deck
determines all short batons.
\end{lemma}
\begin{proof}
The hypothesis yields $\alpha\ge(n-r)/4$ and
$d_4\ge (n-r+4)/4$.  As before, by Remark~\ref{batalph}, $d_1=\alpha+1$ implies
that all short batons fit in cards, so we may assume $d_1\geq\alpha+2$.

By~\eqref{d'} and the lower bounds on $d_1$ and $d_4$,
\begin{align*}
n-r-2&= \sum_{i=1}^r(d_i-2)\ge2(d_1-2)+2(d_4-2)+\sum_{i=5}^r(d_i-2)\\
&\geq 2+4(\alpha-1)+\sum_{i=5}^r(d_i-2)
\geq (n-r-2)+\sum_{i=5}^r(d_i-2).
\end{align*}
This requires $d_5=2$, $\alpha=(n-r)/4$ and $d_1=d_3+1=d_4+1=\alpha+2$.
Since $d_4=\alpha+1$, again $\alpha\ge2$.  Let $\beta=\alpha+1$.

By Remark~\ref{batalph}, each copy of $B_{j:\beta,\beta}$ with $j\le(r-1)/2$
fits in a card.  With $d_5=2$, both branch vertices in each such baton are in
$W_4$.  The unique copy of $B_{j:d_1,d_1}$ in $G$ (if it exists) contains
${d_1-1\choose \alpha}^2$ (that is, $\beta^2$) copies of $B_{j:\beta,\beta}$.
The unique copy of $B_{j:d_3,d_3}$ that is a maximal baton in $G$ (if it
exists) is a copy of $B_{j:\beta,\beta}$.  Also, each copy of $B_{j:d_1,d_4}$
that is a maximal baton in $G$ contains ${d_1-1\choose \alpha}$ (that is,
$\beta$) copies of $B_{j:\beta,\beta}$.

Say that $W_4$ is {\it equally spaced} if there is some $i$ such that
$W_4=\{v_i,v_{i+j},v_{i+2j},v_{i+3j}\}$.
If $B_{j:d_1,d_1}\esub T$, then there may be two, one, or no copies of
$B_{j:d_1,d_4}$ that are maximal batons in $G$ (when there are two, $W_4$
must be equally spaced).  In these cases the numbers of copies of
$B_{j:\beta,\beta}$ in $\cD$ in addition to the $\beta^2$ contained in
$B_{j:d_1,d_1}$ are $2\beta$, ($\beta+1$ or $\beta$), or ($1$ or $0$),
respectively.  Since $\beta\ge3$, these values are all different and
distinguish the cases.

If $B_{j:d_1,d_1}\nosub T$, then there may be up to three copies of
$B_{j:d_1,d_4}$ that are maximal batons in $G$.  Having three such copies
requires $W_4$ to be equally spaced and prohibits $B_{j:d_4,d_4}$ as a 
maximal baton in $G$, yielding $3\beta$ copies of $B_{j:\beta,\beta}$ in $\cD$.
With two such maximal batons in $G$, the number of copies of
$B_{j:\beta,\beta}$ in $\cD$ is $2\beta+1$ or $2\beta$, depending on whether
$W_4$ is equally spaced.  With at most one copy of $B_{j:d_1,d_4}$ as a
maximal baton in $G$, the numbers of copies of $B_{j:\beta,\beta}$ can be
$\beta+1$, $\beta$, $1$, or $0$, depending on whether $B_{j:d_1,d_4}$ and
$B_{j:d_4,d_4}$ occur as maximal batons in $G$.  These values are all 
different and distinguish the cases.

The only ambiguity in using $\#B_{j:\beta,\beta}$ to determine the maximal
batons in $G$ having at least $n-\ell$ vertices is that when $\beta=3$ the case
with $B_{j:d_1,d_1}\esub T$ generating all $\beta^2$ copies of 
$B_{j:\beta,\beta}$ may be confused with the case with $B_{j:d_1,d_1}\nosub T$
in which $W_4$ is equally spaced and $\#B_{j:\beta,\beta}=3\beta$.
To have such confusion, we must have $j\le (r-1)/3$.  Now the number of
vertices in $B_{j:d_1,d_4}$ is bounded by $(r-1)/3+(n-r)/2+2$, since
$d_1=\alpha+2$ and $\alpha=(n-r)/2$.  This bound is at most $(n+2)/2$ when
$r\ge4$.  If $r\le3$, then we cannot have four branch vertices.  Hence
$B_{j:d_1,d_4}$ fits in a card, we can use $\#B_{j:d_1,d_4}$ to distinguish
the two cases (there are $2\beta+2$ copies in $B_{j:d_1,d_1}$ and only
three copies in the case where $W_4$ is equally spaced.

Thus we find all short maximal batons in $G$ whose key vertices are in $W_4$,
and all other short batons fit in cards.
\end{proof}

\section{Four or More Vertices of Maximum Degree}

In Section~\ref{highr}, we proved that for reconstruction of the unknown
$n$-vertex caterpillar $G$ it suffices to know the short maximal batons (length
at most $(r-1)/2$) and the shorter maximal tritons (length at most $(r-4)/2$).
In Section~\ref{2maxbaton}, we proved that the short maximal batons are
known when $G$ has at least two vertices of maximum degree.  Here we begin by
observing that the shorter tritons are also determined when the maximum degree
is small.  Keep in mind that the deck $\cD$ is the multiset of induced
subgraphs with at most $(n+2)/2$ vertices.

\begin{lemma}\label{maxdeg}
If $d_1\le (n-r+12)/6$, then $\cD$ determines all tritons of length at most
$(r-4)/2$.  Thus $n$-vertex caterpillars with $d_1=d_6$ are reconstructible
from $\cD$.
\end{lemma}
\begin{proof}
Tritons with length at most $(r-4)/2$ have at most $(r-4)/2+3d_1-3$ vertices.
This is at most $(n+2)/2$ when $d_1\le (n-r+12)/6$.  Hence all shorter
tritons are then known.

If $d_1=d_6$, then the short maximal batons are known by
Section~\ref{2maxbaton}.  Also, by~\eqref{d'}, $n-r-2\ge6(d_1-2)$, so
$d_1\le(n-r+10)/6$.
\end{proof}

We may henceforth assume $d_1>d_6$.
The next lemmas provide sufficient conditions for tritons to fit in cards.
Recall that $B_{j:a,b}$ has $j+a+b-1$ vertices and $B_{j,j':a,b,c}$ has
$j+j'+a+b+c-3$ vertices.

\begin{lemma}\label{short4}
Suppose $d_1=d_4>d_6$.  If $j+j'\le(r-4)/2$, then the triton 
$B_{j,j':d_6+1,d_6+1,d_6+1}$ fits in a card.
\end{lemma}

\smallskip
\noindent
{\it Proof.}
The triton has at most $(r-4)/2+3d_6$ vertices.  Since $d_4\ge d_6+1$, we
obtain $6d_6\le -4+\SE i16 d_i$.  Now for the number $N$ of vertices
in the triton we use~\eqref{d'} to obtain
$$
\qquad\qquad N\le \FR{r-4}2+3d_6\le\FR{r-4}2-2+\FR12\SE i16 d_i
\le \FR{r-8}2+\FR{n-r+10}2=\FR{n+2}2.\qquad\qquad\Box
$$

\bigskip

Lemma~\ref{lastfew} enables us to handle the case $d_1=d_5$.
\smallskip

\begin{corollary}\label{5max}
If the unknown caterpillar has exactly five vertices of maximum degree
(that is, $d_1=d_5>d_6$), then it is determined by the deck.
\end{corollary}
\begin{proof}
It suffices to discover the shorter tritons.
A vertex of degree at least $d_6+1$ in a subgraph of
$G$ has degree $d_1$ in $G$.  On the other hand, a $j,j'$-triton with three
branch vertices of degree $d_1$ is a maximal triton and contains exactly $M$
copies of $B_{j,j':d_6+1,d_6+1,d_6+1}$, where
$M={d_1-1\choose d_6}^2{d_1-2\choose d_6-1}$.
Thus $\#B_{j,j':d_1,d_1,d_1}=\#B_{j,j':D,D,D}/M$, where $D=d_6+1$.
By Lemma~\ref{short4}, $B_{j,j':D,D,D}$ fits in a card.

We must also find the other shorter maximal tritons with at least
$\FL{(n+2)/2}$ vertices.  These may exist with one or two vertices of maximum
degree.  Since we know the number of these with three vertices of maximum
degree, and those with at most two vertices are contained in
$B_{j,j':D,D,d_6}$, we can find the remaining ones via an exclusion argument.
\end{proof}

Now suppose $d_1=d_4>d_5$.  As before, it suffices to determine the shorter
tritons.  We split the task into two cases, depending on whether
$d_5\leq 1+(n-r)/{6}$.

\begin{lemma}\label{d14low5}
Suppose $d_1=d_4>d_5$.  If $d_5\le (n-r+6)/6$, then the deck determines all
shorter maximal tritons (length at most $(r-4)/2$) having at least
$\FL{(n+2)/2}$ vertices.
\end{lemma}
\begin{proof}
Let $L=1+d_5$, so any triton containing a vertex of degree $L$ is a proper
subgraph of a triton in which that vertex has degree $d_1$.  Also, since
$B_{j,j':a,b,c}$ has $j+j'+a+b+c-3$ vertices, a shorter triton with maximum
degree $L$ has at most $(r-4)/2+3L-3$ vertices.  We compute
$$\FR{r-4}2+3L-3\le \FR{r-4}2+3d_5\le \FR{r-4}2+3\FR{n-r+6}6=\FR{n+2}2.$$
Thus every shorter triton with maximum degree $L$ fits in a card, and we can
count the copies of each.  We use these counts to determine the shorter maximal
tritons containing at least one vertex of maximum degree.

Every copy of $B_{j,j':L,L,L}$ lies in a copy of $B_{j,j':d_1,d_1,d_1}$,
which is the unique maximal triton containing it, and every copy of
$B_{j,j':d_1,d_1,d_1}$ contains $M$ copies of $B_{j,j':L,L,L}$, where
$M=\CH{d_1-1}{L-1}^2\CH{d_1-2}{L-2}$.  Hence
$\#B_{j,j':d_1,d_1,d_1}=\#B_{j,j':L,L,L}/M$.

Now consider the copies of $B_{j,j':a,L,L}$, where $a\le d_5$.  We know
$\#B_{j,j':a,L,L}$, since such subgraphs fit in cards.  Each copy of 
$B_{j,j':a,L,L}$ occurs in a unique copy of $B_{j,j':a,d_1,d_1}$, but that
copy of $B_{j,j':a,d_1,d_1}$ need not be a maximal triton.

We start with $B_{j,j':d_5,L,L}$.  Each copy of $B_{j,j':d_1,d_1,d_1}$
contains $2\CH{d_1-1}{d_5-1}\CH{d_1-2}{L-2}\CH{d_1-1}{L-1}$ copies of
$B_{j,j':d_5,L,L}$.  When we delete these from the count $\#B_{j,j':d_5,L,L}$,
the remaining copies lie in distinct copies of $B_{j,j':d_5,d_1,d_1}$, and
those copies of $B_{j,j:d_5,d_1,d_1}$ are maximal tritons.

We proceed inductively with this exclusion argument as $a$ decreases.
From $\#B_{j,j':a,L,L}$ we must delete all copies that arise as subgraphs
of maximal tritons that are copies of $B_{j,j':a',d_1,d_1}$ for $a<a'\le d_5$
or $B_{j,j':d_1,d_1,d_1}$, all of which have already been counted.
The remaining count is the number of maximal tritons that are copies of
$B_{j,j':a,d_1,d_1}$.

The same technique enables us to count the maximal tritons that are copies
of $B_{j,j':d_1,b,d_1}$, for $2\le b\le d_5$.  Having determined the
maximal $j,j'$-tritons having two vertices of maximum degree, we can then use
the same exclusion approach to determine the maximal $j,j'$-tritons having one
vertex of maximum degree, starting with copies of $B_{j,j':d_5,d_5,d_1}$
and $B_{j,j':d_5,d_1,d_5}$ and working down to smaller $j,j'$-tritons.

Since all shorter tritons having no vertices of degree $d_1$ fit in cards,
the exclusion argument now permits finding all the maximal tritons among them.
\end{proof}


\begin{lemma}\label{d14high5}
Suppose $d_1=d_4>d_5$.  If $d_5\ge(n-r+7)/6$, then the deck determines all
shorter maximal tritons having at least $\FL{(n+2)/2}$ vertices.
\end{lemma}
\begin{proof}
We first prove $d_5\ge3$.  Otherwise, $2\ge(n-r+7)/6$, so $n-r\le5$.
Now~\eqref{d'} implies
$$
3\ge n-r-2\ge\SE i14(d_i-2)\ge4.
$$

The inequality $d_6<d_5$ is now immediate if $d_6\le 2$.
When $d_6\ge 3$, we prove a stronger inequality.
Let $\alpha=d_5-(n-r+6)/6$, so $\alpha\ge1/6$.  Since $d_1\ge d_5+1$,
using~\eqref{d'} yields
\begin{align}\label{d6}
d_6-2&\leq n-r-2-5(d_5-2)-4\le (n-r-1)-5(d_5-1)\notag\\
&\le n-r-1-5\left(\FR{n-r}6+\alpha\right)= \frac{n-r+6}{6}-2-5\alpha.
\end{align}
In particular, $d_6\le (n-r+6)/6-5\alpha\le d_5-1$.

Again let $W_k=\{w_1,\ldots,w_k\}$.
Given $j$ and $j'$ with $j+j'\le (r-4)/{2}$, vertex $w_5$ can be in at most
two maximal tritons that are copies of $B_{j,j':d_1,d_5,d_1}$ in $G$.
Moreover, if $w_5$ is in two such maximal tritons, then $j'\neq j$ and $w_5$
has two vertices in $W_4$ at distance $j$ in $G$ and two vertices in $W_4$ at
distance $j'$.  Since we know all short batons, we can determine whether this
occurs.  If it does, then we know all shorter tritons whose key vertices are
all in $W_5$.  Otherwise, we know that at most one maximal triton is a copy of
$B_{j,j':d_1,d_5,d_1}$.

A similar statement holds for copies of $B_{j,j':d_5,d_1,d_1}$, with the only
difference being that then two such maximal tritons may also occur when $j=j'$.
We still see whether this happens from the deck.  The case of
$B_{j,j':d_1,d_1,d_5}$ is symmetric.

Having restricted to one of these $j,j'$-tritons occurring at most once
as a maximal triton, we want to determine whether it in fact appears.
Let $\beta$ be the number of vertices in $B_{j,j':d_5+1,d_6+1,d_5+1}$.
Using~\eqref{d6}, we compute that it fits in a card.  That is,
\begin{equation}\label{fitt}
\beta\le\frac{r-4}{2}+2d_5+d_6\leq
\frac{r-4}{2}+2\alpha+3\frac{n-r+6}{6} -5\alpha<\frac{n+2}{2}.
\end{equation}
Every such triton is contained in a maximal triton in $G$ that is a copy of
(i)~$B_{j,j':d_1,d_1,d_1}$ or (ii) $B_{j,j':d_1,d_5,d_1}$.  Each triton of
type (i) in $G$ contains ${d_1-1\choose d_5}^2{d_1-2\choose d_6-1}$ copies of
$B_{j,j':d_5+1,d_6+1,d_5+1}$, while each triton of type (ii)  contains
fewer copies, only ${d_1-1\choose d_5}^2{d_5-2\choose d_6-1}$.
Since we have restricted to the case where there is at most one triton of type
(ii), we decide whether there is such a triton by checking whether the number
of copies of $B_{j,j':d_{5}+1,d_6+1,d_5+1}$ is divisible by
${d_1-1\choose d_5}^2{d_1-2\choose d_6-1}$; if so, then all the copies of 
$B_{j,j':d_5+1,d_6+1,d_5+1}$ are contained in copies of $B_{j,j':d_1,d_1,d_1}$.
We thus find the number of maximal $j,j'$-tritons of types (i) and (ii)
in $G$.

Now, consider shorter tritons of the form $B_{j,j':d_6+1,d_5+1,d_5+1}$ that are
not subgraphs of the tritons of types (i) and (ii) in $G$.  As in~\eqref{fitt},
such shorter tritons fit in cards, so we can count them.  The copies we have
not excluded are contained in copies of $B_{j,j':d_5,d_1,d_1}$.  We know
how many copies of $B_{j,j':d_6+1,d_5+1,d_5+1}$ lie in each copy of
$B_{j,j':d_5,d_1,d_1}$, so we now know how the number of copies of
$B_{j,j':d_5,d_1,d_1}$ that are maximal tritons in $G$.  Thus, we know the
all the shorter maximal tritons in $G$ whose key vertices all lie in $W_5$. 

The tritons $B_{j,j':d_5+1,d_5+1,d_6}$, $B_{j,j':d_5+1,d_6,d_5+1}$,
and $B_{j,j':d_6,d_5+1,d_5+1}$ fit in cards, since they are smaller than
$B_{j,j':d_6+1,d_5+1,d_5+1}$.  They tell us how many maximal $j,j'$-tritons in
$G$ have exactly two key vertices in $W_4$ and one key vertex of degree $d_6$.
Now, using exclusion, we obtain all the shorter maximal tritons. 
\end{proof}

\begin{theorem}
For $n\ge48$, every $n$-vertex caterpillar with at least four vertices
of maximum degree is reconstructible from the deck.
\end{theorem}
\begin{proof}
We have shown that the short maximal batons and the shorter maximal tritons
are determined by the deck.  Theorem~\ref{suffic} and Lemma~\ref{lastfew}
complete the reconstruction.
\end{proof}

\section{Three Maximum-degree Vertices}\label{3max}

In this section we reconstruct high-diameter caterpillars with exactly
three vertices of maximum degree and at least four branch vertices:
$d_1=d_3>d_4\ge3$.  We recognize being in this case from the degree list.

Our approach has three main steps.
First we find the positions of the three vertices of maximum degree.
Next we consider the position of a branch vertex whose degree appears only once
in the degree list.  Finally, we show that the deck determines the shorter
tritons.  As we have observed, that suffices for the reconstruction.

By Section~\ref{2maxbaton} we know all short maximal batons.
Hence by Lemma~\ref{levelpair} we know all level pairs.
Throughout this section, we use the following terminology and notation.

\begin{definition}\label{W3def}
Recall that $\VEC w1r$ indexes vertices on the spine $P$ in order of degree,
so that $d(w_i)=d_i$, and $W_t=\{w_1,\ldots,w_t\}$.  In particular, with
$d_1=d_3$, we index $\{w_1,w_2,w_3\}$ in nondecreasing order from the center of
$P$.  A {\it $W_3$-baton} is a baton whose key vertices are both in $W_3$.
The {\it depth} $f(v)$ of a vertex $v$ on $P$ is the distance of
$v$ from the ``middle'' or ``top'' of $P$.  That is, for the level pair
$\{v_k,v_\kb\}$ we have $f(v_k)=f(v_\kb)=\C{s-k}$, where $s=(r+1)/2$.  The
depth of a vertex is a half-integer when $r$ is even.  To describe the
depths of the vertices in $W_3$, we use $a,b,c$ as follows:
\begin{equation}\label{leveldist}
f(w_1)=a\le f(w_2)=a+b\le f(w_3)=a+b+c\le\FR{r-1}2.
\end{equation}
Also let $h=s-a-b-c$, so $w_3\in \{v_h,v_\hb\}$.
The notation appears in Figure~\ref{3maxfig}, showing a difficult case in the
next lemma, where we aim to decide which square is the location of $w_2$.
\end{definition}

\begin{figure}[h!]
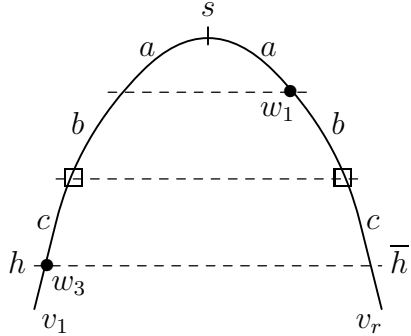

\begin{center}
\gpic{
\expandafter\ifx\csname graph\endcsname\relax
   \csname newbox\expandafter\endcsname\csname graph\endcsname
\fi
\ifx\graphtemp\undefined
  \csname newdimen\endcsname\graphtemp
\fi
\expandafter\setbox\csname graph\endcsname
 =\vtop{\vskip 0pt\hbox{%
    \special{pn 11}%
    \special{pa 91 1568}%
    \special{pa 318 659}%
    \special{pa 1000 -22}%
    \special{pa 1682 659}%
    \special{pa 1909 1568}%
    \special{sp}%
    \special{pn 8}%
    \special{pa 91 1341}%
    \special{pa 1909 1341}%
    \special{da 0.045}%
    \special{pa 205 886}%
    \special{pa 1795 886}%
    \special{da 0.045}%
    \special{pa 477 432}%
    \special{pa 1523 432}%
    \special{da 0.045}%
    \graphtemp=.5ex
    \advance\graphtemp by 0.205in
    \rlap{\kern 0.682in\lower\graphtemp\hbox to 0pt{\hss $a$\hss}}%
    \graphtemp=.5ex
    \advance\graphtemp by 0.614in
    \rlap{\kern 0.318in\lower\graphtemp\hbox to 0pt{\hss $b$\hss}}%
    \graphtemp=.5ex
    \advance\graphtemp by 1.114in
    \rlap{\kern 0.136in\lower\graphtemp\hbox to 0pt{\hss $c$\hss}}%
    \graphtemp=.5ex
    \advance\graphtemp by 0.205in
    \rlap{\kern 1.318in\lower\graphtemp\hbox to 0pt{\hss $a$\hss}}%
    \graphtemp=.5ex
    \advance\graphtemp by 0.614in
    \rlap{\kern 1.682in\lower\graphtemp\hbox to 0pt{\hss $b$\hss}}%
    \graphtemp=.5ex
    \advance\graphtemp by 1.114in
    \rlap{\kern 1.864in\lower\graphtemp\hbox to 0pt{\hss $c$\hss}}%
    \graphtemp=.5ex
    \advance\graphtemp by 0.432in
    \rlap{\kern 1.432in\lower\graphtemp\hbox to 0pt{\hss $\bu$\hss}}%
    \graphtemp=.5ex
    \advance\graphtemp by 1.341in
    \rlap{\kern 0.159in\lower\graphtemp\hbox to 0pt{\hss $\bu$\hss}}%
    \graphtemp=.5ex
    \advance\graphtemp by 0.523in
    \rlap{\kern 1.364in\lower\graphtemp\hbox to 0pt{\hss $w_1$\hss}}%
    \graphtemp=.5ex
    \advance\graphtemp by 1.432in
    \rlap{\kern 0.273in\lower\graphtemp\hbox to 0pt{\hss $w_3$\hss}}%
    \graphtemp=.5ex
    \advance\graphtemp by 1.632in
    \rlap{\kern 0.201in\lower\graphtemp\hbox to 0pt{\hss $v_1$\hss}}%
    \graphtemp=.5ex
    \advance\graphtemp by 1.632in
    \rlap{\kern 1.845in\lower\graphtemp\hbox to 0pt{\hss $v_r$\hss}}%
    \graphtemp=.5ex
    \advance\graphtemp by 0.000in
    \rlap{\kern 1.000in\lower\graphtemp\hbox to 0pt{\hss $s$\hss}}%
    \graphtemp=.5ex
    \advance\graphtemp by 1.341in
    \rlap{\kern 0.000in\lower\graphtemp\hbox to 0pt{\hss $h$\hss}}%
    \graphtemp=.5ex
    \advance\graphtemp by 1.341in
    \rlap{\kern 2.000in\lower\graphtemp\hbox to 0pt{\hss $\hb$\hss}}%
    \special{pn 11}%
    \special{pa 1000 91}%
    \special{pa 1000 182}%
    \special{fp}%
    \graphtemp=.5ex
    \advance\graphtemp by 0.886in
    \rlap{\kern 0.295in\lower\graphtemp\hbox to 0pt{\hss $\marker$\hss}}%
    \graphtemp=.5ex
    \advance\graphtemp by 0.886in
    \rlap{\kern 1.705in\lower\graphtemp\hbox to 0pt{\hss $\marker$\hss}}%
    \hbox{\vrule depth1.659in width0pt height 0pt}%
    \kern 2.000in
  }%
}%
}
\vspace{-.5pc}
\caption{Notation for Section~\ref{3max}\label{3maxfig}}
\end{center}
\end{figure}
\vspace{-1pc}

Since $r\ge(n-5)/2$ in the high-diameter case, the condition $r\ge9$ is
ensured by $n\ge23$.

\begin{lemma}\label{W3}
In the high-diameter case with exactly three vertices of maximum degree
(and $r\ge9$), the deck determines the positions of the three maximum-degree
vertices.
\end{lemma}
\begin{proof}
By Sections~\ref{2maxbaton} and~\ref{highr}, we know the short maximal batons
and the level pairs.  Hence we know the differences $a,b,c$ in
Definition~\ref{W3def}.  When two vertices in $W_3$ have the same level, the
claim is trivial (up to reflection).  Hence we may assume $b,c>0$.
If $a=0$, then we see $W_3$-batons with lengths $b$ and $b+c$.
If we also see a $W_3$-baton with length $c$, then $w_2$ and $w_3$ are in the
same half of $P$; otherwise, they are in different halves.  Thus henceforth
we may also assume $a>0$.

Our goal is to determine the ``Type'' of $W_3$, where $W_3$ is {\em Type 0} if all
of $W_3$ is in the same half of $P$ and {\em Type i} for $i\in\{1,2,3\}$ if $w_i$
is in the half opposite $W_3-\{w_i\}$.  There are always three $W_3$-batons; we
consider cases by the number of short $W_3$-batons.  Since $W_3$ has at least two
vertices on the same side of $P$, there is always at least one short $W_3$-baton
with length $b$, $c$, or $b+c$.  Short $W_3$-batons may also occur with length
$b+2a$, $b+c+2a$, or $c+2a+2b$ when the key vertices are in opposite halves
of $P$.

\medskip
{\bf Case 1:} {\it All three $W_3$-batons are short.}  
We see all three $W_3$-batons and their lengths.  The triples of lengths are
as follows:
\begin{center}
\begin{tabular}{l l l l}
Type 0:&$\{b,c,b+c\}$   &Type 2:&$\{b+c,2a+b,2a+2b+c\}$ \\
Type 1:&$\{c,2a+b,2a+b+c\}$   &Type 3:& $\{b,2a+b+c,2a+2b+c\}$.
\end{tabular}
\end{center}
The triples of lengths are distinct in the four cases, so they distinguish the
cases.  In all cases these are the triples of lengths of the $W_3$-batons,
but in other cases we don't see them all.

\medskip
{\bf Case 2:} {\it There are exactly two short $W_3$-batons.}
In each triple of lengths listed above, the third length is the sum of the
other two, so to have only two short $W_3$-batons they must be the two shorter
lengths in each triple.  For the four Types, the lengths of the two shortest
$W_3$-batons are distinct pairs, so the lengths of the two short $W_3$-batons
distinguish the cases.

\medskip
{\bf Case 3:} {\it There is only one short $W_3$-baton, $B$.}
Type 0 has three short $W_3$-batons, so that is excluded.
Among the other types, the two vertices of $W_3$ in the same half of $P$
guarantee a short $W_3$-baton with length in $\{b,c,b+c\}$.
If $B$ has length $b+c$, then we see no $W_3$-batons with the shorter lengths
$b$ or $c$, and $W_3$ is Type 2.
If $B$ has length $b$ or $c$, and $b\ne c$, then $W_3$ is Type 3 or Type 1,
respectively.  Hence we may assume that $B$ has length $b$ and $b=c$, and
we need to distinguish between Type 3 and Type 1.

In both types, $w_1$ and $w_3$ are in opposite halves of $P$, and we need to
determine which half contains $w_2$ (see Figure~\ref{3maxfig} in
Section~\ref{3max}).  By reflection, we may assume $w_1=v_{s+a}$ and $w_3=v_h$.
If $2a+b\le(r-1)/2$, then $W_3$ is Type 3, since $B$ being the only short
$W_3$-baton forbids a $W_3$-baton of length $2a+b$.

Hence we may assume $2a+b\geq r/2$.  Now let
$H=B_{c,h:d_4+1,d_4+1,2}$ and $H'=B_{c,h:d_4+1,d_4+1,3}$.  Since $d_4\ge3$,
any copy of $H$ or $H'$ in $G$ has exactly two vertices in $W_3$.  Let
$M=\CH{d_1-1}{d_4}\CH{d_1-2}{d_4-1}$; there are $M$ ways of
choosing the neighbors of the high-degree vertices when forming such tritons.
If $W_3$ is Type 1, then $w_2=v_{h+c}$, and the third key vertex of copies of
$H$ or $H'$ is $v_{2h+c}$.  If $W_3$ is Type 3, then $w_2=v_{s+a+b}$ and the
third key vertex may be $v_{s+a+b+h}$ or $v_{s+a-h}$.  In the two cases we
obtain the following counts.
\begin{center}
\begin{tabular}{c | c c}
 & Type 1& Type 3 \\
\noalign{\hrule}
$\#B_{c,h:d_4+1,d_4+1,2}/M$ & $d(v_{2h+c})-1$ & $d(v_{s+a+b+h})-1+d(v_{s+a-h})-1$\\
$\#B_{c,h:d_4+1,d_4+1,3}/M$ & $\CH{d(v_{2h+c})-1}2$ &
$\CH{d(v_{s+a+b+h})-1}2+\CH{d(v_{s+a-h})-1}2$
\end{tabular}
\end{center}
The counts distinguish the two cases, because if positive
integers $x$ and $y$ sum to $z$, then $\CH x2+\CH y2<\CH z2$.

It remains only to show that $H'$ (and hence also $H$) fits in a card, so we
can count the copies.  By~\eqref{d'} and $d_3>d_4$, we have
$n-r-2\geq 4d_4-5$, so $d_4\le(n-r+3)/4$.
Also $h+c=s-a-b$, so $2a+b\ge r/2$ yields $2h+2c+b\le 2s-r/2=(r+2)/2$.
Since $b\ge1$, this implies $h+c\le r/4$.  Now
$$
|V(H')|= c+h+2d_4+2\le \FR{r}4+\FR{n-r+3}2+2=\FR{n+7}2-\FR{r}4.
$$
We need $|V(H')|<(n+3)/2$, which is guaranteed when $r\ge9$.
\end{proof}

\begin{lemma}\label{3bat1}
In the high-diameter case with $d_1=d_3>d_4>d_5$, if we know the
positions of the vertices of $W_3$ on $P$, then we know the position of $w_4$.
\end{lemma}
\begin{proof}
We continue using the notation of Definition~\ref{W3def}.  Also, let a
{\it $W_4$-baton} be a maximal baton whose key vertices have degrees $d_4$ and
$d_1$; the lengths of $W_4$-batons are the distances from $w_4$ to vertices of
$W_3$.  

Again we know the short maximal batons and the level pairs; we also know the
$W_3$-batons.  If $W_3$ has two vertices on the same level, then the pair of
distances from $w_4$ to these vertices is known.  From the level pairs,
$f(w_4)$ is known, so we have two options for the length of the third
$W_4$-baton.  At least one of the options is short, so we can tell 
whether $w_4$ is in the same half of $P$ with the third vertex of $W_3$.

Hence we may again assume that the vertices of $W_3$ have distinct depths,
so $b,c>0$.  If $f(w_4)=f(w_j)$ for some $1\leq j\leq 3$, then we know the
position of $w_4$.  Hence we may assume strict inequalities in specify the
case below for the depth of $w_4$.

\medskip
{\bf Case 1:} {\it $f(w_4)<a$ or $f(w_4)>a+b+c$}.  In this case the length of
the shortest $W_4$-baton determines which half of $P$ contains $w_4$.  If we see
no short $W_4$-baton, then $w_4$ is in the opposite half of $P$ from all of $W_3$.

\medskip
{\bf Case 2:} {\it $a+b<f(w_4)<a+b+c$}.
If $f(w_4)\neq a+b+c/2$ or $w_2$ and $w_3$ are in the same half of $P$, then 
the length of the shortest $W_4$-baton determines the position of $w_4$.
Hence we may assume $f(w_4)= a+b+c/2$ and that $w_2$ and $w_3$ are in opposite
halves of $P$.

One of the distances from $w_4$ to $\{w_2,w_3\}$ is $c/2$, and the other is
at least $2a+2b+c/2$.  We determine the position of $w_4$ by testing for
$W_4$-batons of length $b+c/2$, which is between the two distances from $w_4$
to $\{w_2,w_3\}$.  If such a $W_4$-baton exists, then $w_4$ is in the same
half of $P$ as $w_1$; otherwise, it is in the opposite half.

\medskip
{\bf Case 3:} {\it $a<f(w_4)<a+b$}.  This begins like Case 2 but becomes
harder.  If $f(w_4)\neq a+b/2$ or $w_1$ and $w_2$ are in the same half of $P$,
then again the length of the shortest $W_4$-baton determines the position of
$w_4$.  Hence we may assume $f(w_4)= a+b/2$ and that $w_1$ and $w_2$ are in
opposite halves of $P$.
 
One of the distances from $w_4$ to $\{w_1,w_2\}$ is $b/2$, and the other is
either $2a+b/2$ or $2a+3b/2$.  If no $W_4$-baton has length $b/2+c$, then $w_4$
is the opposite half of $P$ from $w_3$.  If there are two $W_4$-batons with
length $b/2+c$, or if there is one and $b/2+c\notin\{2a+b/2,2a+3b/2\}$,
then $w_4$ is in the same half of $P$ as $w_3$. 

Hence we may assume one $W_4$-baton has length $b/2+c$ and
$b/2+c\in\{2a+b/2,2a+3b/2\}$.  First suppose $b/2+c=2a+3b/2$.
This requires $2a+3b/2<r/2$, so we can see whether there is a $W_4$-baton
with length $2a+3b/2$.  If not, then $w_4$ is in the same half of $P$ as $w_1$;
otherwise, $w_4$ is in the same half of $P$ as $w_2$. 

Finally, suppose that one $W_4$-baton has length $b/2+c$ and $b/2+c=2a+b/2$,
which requires $c=2a$.  Since there is exactly one $W_4$-baton of this length,
$w_3$ must be in the same half of $P$ as $w_1$ (and $w_2$ is in the opposite
half).  If $2a+3b/2<r/2$, then we can check the existence of a $W_4$-baton with
length $2a+3b/2$; it exists if and only if $w_4$ is on the opposite side of $P$
from $w_2$.

Hence we may assume $2a+3b/2\geq r/2$.  Let $H=B_{b/2+c,h:d_4,d_{4}+1,2}$.
Note that the key vertex of degree $d_4$ in any copy of $H$ in $G$ must be
$w_4$, since no two vertices in $W_3$ are at distance $b/2+c$.  Now $w_4$ can be
on the same side of $P$ as $w_2$ at distance $b/2+c$ from $w_1$, or on the
opposite side from $w_2$ at distance $b/2+c$ from $w_3$.  However, the path
does not extend far enough beyond $w_3$ to complete a copy of $H$ in the latter
case.  Hence if $G$ contains $H$, then $w_4$ is in the same half of $P$ as
$w_2$; otherwise it is not.

It remains only to show that $H$ fits in a card in order to test whether it
appears.  Let $N$ be its number of vertices.  Again~\eqref{d'} implies
$d_4\le(n-r+3)/4$.  Using $h=s-a-b-c$, and letting $x=a+b/2$, we have
$$
N=\frac{b}{2}+c+h+2d_4=\frac{r+1}{2}-a-\frac{b}{2}+2d_4
\le\FR{r+1}2-x+\FR{n-r+3}2=\FR{n+2}2+1-x.
$$
It thus suffices to have $x\ge1$.  Because $2a=c\ge1$, we have $a\ge1/2$.
Also $b$ is a positive integer, so $a+b/2\ge1$.
\end{proof}

\begin{lemma}\label{3bat11}
In the high-diameter case with $d_1=d_3>d_4>2$, let $i$ and $j$ be indices
such that both $d_i$ and $d_j$ are at least $3$ and occur only once
in the degree list.  If the deck determines the position of $w_i$ on $P$,
then the deck also determines the position of $w_{j}$.
\end{lemma}
\begin{proof}
We know the short maximal batons, and these give us the level pairs.  Hence
along with the position of $w_i$ we also know the depth of $w_j$.
Let $h=\C{f(w_i)-f(w_j)}$.  If the deck has $B_{h:d_i,d_j}$ as a
maximal baton, then $w_{j}$ is in the same half of $P$ as $w_i$; otherwise,
they are in opposite halves.
\end{proof}

\begin{theorem}
In the high-diameter case with $d_1=d_3>d_4>2$ and $n\ge48$, 
all $n$-vertex caterpillars are reconstructible from their decks.
\end{theorem}
\begin{proof}
By Section~\ref{2maxbaton} we know the short maximal batons.
By Section~\ref{highr} we know the level pairs and that it suffices to
determine the shorter maximal tritons.

Let $t$ be the smallest index at least $4$ such that $d_t=d_{t+1}$.
By the lemmas above, we know the positions of all vertices in $W_{t-1}$, so we
also know all shorter maximal tritons whose three key vertices all lie in
$W_{t-1}$, though these tritons may not all fit in cards.  Our goal is to
determine all shorter maximal tritons with at most two vertices in $W_{t-1}$.

\medskip
{\bf Case 1:}  {\it $t=4$, so $d_1=d_3>d_4=d_5$}.
Let $d^*_6=\min\{d_6,d_5-1\}$, so $d^*_6\leq d_1-2$.  For all $j$ and $j'$ with
$j+j'\leq (r-4)/2$, consider the tritons $H$ and $H'$ defined by
$H=B_{j,j':d_4+1,d_4+1,d^*_6+1}$ and $H=B_{j,j':d_4+1,d^*_6+1,d_4+1}$. 
We first show that $H$ and $H'$ fit in cards.
Since $d_1>d_4$, $d_2>d_5$, and $d_3\ge d^*_6+2$,~\eqref{d'} yields
$$
n-r-2\geq \sum_{i=1}^{6}(d_i-2)\geq -12+2(d_4+d_5+d^*_6)+4.
$$
For the number $N$ of vertices, with $d_4=d_5$, we have
$$
N\le\frac{r-4}{2}+d_4+d_4+d^*_6\leq \frac{r-4}{2}+\frac{n-r+6}{2}=\frac{n+2}{2}.
$$
Thus $H$ and $H'$ fit in cards.

Each copy of $H$ or $H'$ has at least two key vertices in $W_3$, and the third
key vertex has degree at least $d_4$, since $d_4=d_5$.  Since $H$ and $H'$ fit
in cards, we can compute $\#H$ and $\#H'$.  After excluding the copies that
are contained in $B_{j,j':d_1,d_1,d_1}$, we obtain the number of copies of
$B_{j,j':d_1,d_1,d_4}$ and $B_{j,j':d_1,d_4,d_1}$ that are maximal tritons.
All other shorter maximal tritons are contained in $H$ or $H'$ and thus also
fit in cards.  Hence by Lemma~\ref{counting} we can find all the shorter
maximal tritons.

\medskip
{\bf Case 2:} {\it $t\geq 5$}.
If $d_t=2$, then we already know the positions of all branch vertices
and hence know the caterpillar.  Hence we may assume
$d_{t+1}\geq 3$.
For all $j,j'$ with $j+j'\leq (r-4)/2$, consider the tritons $H$ and $H'$
given by
$H=B_{j,j':d_4+1,d_4+1,d_{t+1}}$ and $H=B_{j,j':d_4+1,d_{t+1},d_4+1}$. 
Since $t+1\geq 6$ and $d_1>d_4$,~\eqref{d'} yields
$$
n-r-2\ge \sum_{i=1}^{t+1}(d_i-2)\geq
3(d_1-2)+(d_4-2)+2(d_{t+1}-2)\geq 3+4(d_4-2)+2(d_{t+1}-2),
$$
which implies $4d_4+2d_{t+1}\le n-r-7$.  Now,
$$
|V(H)|=|V(H')|\le\frac{r-4}{2}+2d_4+d_{t+1}-1
\leq \frac{r-4}{2}+\frac{n-r+7}{2}-1=\frac{n+1}{2}.
$$
Hence these tritons fit in cards.

As in Case 1, we know the shorter maximal tritons with key vertices in $W_4'$.
Again we can exclude from the count of $H$ and $H'$ the copies that arise from
such larger tritons to obtain the number of copies of $H$ and $H'$ that are
maximal tritons.
All other shorter maximal tritons are contained in $H$ or $H'$, and as in
Case 1 we can find them.
\end{proof}

\section{Two Maximum-degree Vertices}\label{2max}

In this section we reconstruct high-diameter caterpillars with exactly two
vertices of maximum degree and at least four branch vertices:
$d_1=d_2>d_3\ge d_4\ge3$.  We recognize being in this case from the degree list.
Throughout this section we consider only such caterpillars.

\begin{lemma}\label{2maxrd3}
In this section we may assume that the number $r$ of nonleaf vertices satisfies
\begin{equation}\label{eq:rbounds}
\FR{n-5}2 \le r \leq n-8,
\end{equation}
and also we may assume $d_3\le(n-r+1)/3$.
\end{lemma}
\begin{proof}
When $r\le (n-6)/2$ the small-diameter case in Section~\ref{lowr} applies,
so we may assume $r\ge (n-5)/2$.
Given $d_4\ge3$, with exactly two vertices of maximum degree we also have
$d_2\ge4$.  Thus by
\eqref{d'} we have $6\le \SE i14(d_i-2)\le n-r-2$, so $r\le n-8$.

With $d_2>d_3$ and $d_4\ge3$, we also have $3d_3+2+3-8\le n-r-2$,
which simplifies to $d_3\le(n-r+1)/3$.
\end{proof}

Let $b_1$ and $b_2$ be the indices of the two vertices of degree $d_1$;
that is, $\deg(v_{b_1}) = \deg(v_{b_2}) = d_1$.  By symmetry, we may assume
that $v_{b_1}$ is earlier along the spine and is at no higher a level than
$v_{b_2}$; that is, $b_1 \leq \min\{b_2, \bb_2\}$.  (Among two level
pairs, the one giving degrees of vertices closer to the center of the spine is
``higher'' in the sense of Figure~\ref{3maxfig} in Section~\ref{3max}.)

\begin{lemma}\label{b1b2}
All short maximal batons and the level pairs are determined by the deck,
as are the indices $b_1$ and $b_2$ of the maximum-degree vertices.
\end{lemma}
\begin{proof}
Since $d_1=d_2$, by Section~\ref{2maxbaton} we know the multiset of short
maximal batons.  This gives us the multiset of degree pairs for each distance
at most $(r-1)/2$ along the spine.  By Lemma~\ref{levelpair},
we thus also know the level pairs $\{d(v_k),d(v_{\kb})\}$ for all $k$.

The value $b_1$ is the least $k$ such that the level pair $\{d(v_k),d(v_\kb)\}$
contains $d_1$.  When this level pair consists of two copies of $d_1$, we have
$b_2=\bb_1$ and call this the {\it symmetric case}.  When $b_2\ne\bb_1$,
another level pair contains $d_1$.  That pair has the degrees of two vertices,
at least one of which is within distance $(r-1)/2$ of $v_{b_1}$.  Since we know
the short maximal batons, we know whether $G$ contains $B_{j:d_1,d_1}$ for some
$j$ bounded by $(r-1)/2$, even though this subgraph may be too big to see in a
card.  If such a baton exists, then we know $b_2=b_1+j$.  If no such baton
exists, then $b_2>b_1+(r-1)/2$, and $b_2$ is the larger index in its level pair.
\end{proof}

We will split the analysis into cases depending on the values of $b_1$
and $b_2$.  We summarize the cases here.
	
{\bf Case 1:} 
$b_2 - b_1 \leq r/3 - 7/2$ (the vertices of maximum degree are close together).

{\bf Case 2:}
$b_1 \leq r/4 -2$ and $b_2\ne\bb_1$ (asymmetric with $b_1$ near the end of the
spine).

{\bf Case 3:}
$r/4 - 2 < b_1 < r/3 + 9/4 $ (asymmetric and $b_1$ somewhat central).

{\bf Case 4:}
$b_1 < r/3 + 9/4 $ and $b_2=\bb_1$ ($b_1$ and $b_2$ symmetric and near the
ends). 

\noindent
Note that $b_1\le\bb_2$ implies $b_2\le r+1-b_1$, and hence
$b_2-b_1\le r+1-2b_1$.  Thus if $b_1\ge r/3+9/4$, then $b_2-b_1\le r/3-7/2$
and Case 1 applies.  Therefore, the upper bounds on $b_1$ 
in Cases 3 and 4 cover all possibilities.
We will find that the argument for Case 3 needs $n\ge48$.

We next describe additional information that will
suffice to reconstruct the unknown caterpillar.  We call this the
``Climbing Lemma'' because we iteratively discover the allocation of 
higher (that is, as in Figure~\ref{3maxfig}, more central) level pairs.

\begin{lemma}[Climbing Lemma]\label{climbing}
In the asymmetric case $b_1<\min\{b_2,\bb_2\}$, the unknown caterpillar can
be reconstructed in either of the following cases:
		
(1) $b_2\le(r+1)/2$ and $(\deg(v_1), \ldots, \deg(v_{b_2}))$ is known.
		
(2) $b_2 > (r+1)/2$ and $(\deg(v_{b_2}), \ldots, \deg(v_r))$ is known.
\end{lemma}
\begin{proof}
The key technique is to detect vertex degrees at fixed distances from $b_1$ and
$b_2$.  We prove the second case in detail; the first is similar. 
Since we know $(\deg(v_{b_2}),\ldots,\deg(v_r))$, the level pairs also
tell us $(\deg(v_1), \ldots, \deg(v_{\bb_2}))$.
We aim to discover $d(v_{b_2-h})$, which tells us also $d(v_{\bb_2+h})$.
We do this in order from $h=1$ to $h=\CL{b_2-(r+1)/2}$.  We use maximal batons
of length $h$, which by Section~\ref{2maxbaton} we know since $h\le(r-1)/2$.

Having done this for distances from $1$ through $h-1$, consider the maximal
batons of the form $B_{h:a,d_1}$.  These exist when $a$ is the degree of a
vertex at distance $h$ from $v_{b_1}$ or $v_{b_2}$.  Over all choices for $a$,
since there are exactly two vertices of degree $d_1$, there are three or four
such maximal batons (only three if $h\ge b_1$).  This tells us the multiset of
four values in the list $(d(v_{b_1-h}),d(v_{b_1+h}),d(v_{b_2+h}),d(v_{b_2-h}))$,
where we treat the first as $0$ if $h\ge b_1$.  Since $b_1<\min\{b_2,\bb_2\}$,
we already know the first three values in the list, so from the multiset of
four we obtain $d(v_{b_2-h})$.  Now we also know $d(v_{\bb_2+h})$ because we
know the level pairs.

In the case where $b_2\le(r+1)/2$, we just obtain $d(v_{b_2+h)})$ before
$d(v_{\bb_2-h})$.
\end{proof}
	
We proceed to prove reconstructibility from the deck in the four cases listed
after Lemma~\ref{b1b2} for the position of the two maximum-degree vertices
$v_{b_1}$ and $v_{b_2}$ (with $b_1\le\min\{b_2,\bb_2\}$).
	
	
\begin{lemma}\label{Case1}
If $b_2 -b_1 \leq r/3 - 7/2$ (Case 1), then the caterpillar is reconstructible.
\end{lemma} 
\begin{proof}
We first obtain the list $(d(v_{b_1}),\ldots,d(v_{b_2}))$, up to reversal.
The idea is that a certain non-maximal baton whose key vertices can only be
$v_{b_1}$ and $v_{b_2}$ fits in a card with room for two more vertices.  We 
then use the small-diameter method from Section~\ref{lowr}.

Let $h=b_2-b_1$.  The number of vertices in $B_{h:d_3+1,d_3+1}$ is
$h+1+2d_3$.  The baton fits in a card with room for two more vertices
if $h+3+2d_3\le(n+2)/2$.  From $r\ge(n-5)/2$ we obtain $n-4r\le-2r+5$.
Using the bound on $d_3$ from Lemma~\ref{2maxrd3}, we compute
$$
4d_3\le n+\FR{n-4r+4}3\le n-\FR{2r}3+3.
$$
Now when $h\le r/3-7/2$, we have the desired bound:
$$
2h+6+4d_3\le \FR{2r}3-7+6+n-\FR{2r}3+3=n+2.
$$
Hence we can see and count the desired subgraphs.

Let $H=B_{h:d_3+1,d_3+1}$, and let $P$ denote the path in $H$ connecting the
two key vertices.  Let $H_i$ be the graph obtained from $H$ by
adding a pendant edge to the vertex at distance $i$ along $P$.
Let $H_{i,j}$ be the graph obtained from $H$ by adding one pendant edge
each to the vertices at distances $i$ and $j$ along $P$ from one end.

We can see copies of $H_i$ and $H_{i,j}$ in the deck, and in each copy the
key vertices must be $v_{b_1}$ and $v_{b_2}$.  Let $M=\CH{d_1-1}{d_3}^2$.
Letting $d'(v)$ denote the number of leaf neighbors of $v$ as in
Section~\ref{lowr}, we have $\#H_i=(d'(v_{b_1+i})+d'(v_{b_2-i}))M$ and
$\#H_{i,h-i}=d'(v_{b_1+i})d'(v_{b_2-i})M$ (except $\#H_i=d'(v_{b_1+h/2})M$
when $i=h/2$).  Knowing the sum and product, we can find the pairs
$\{d'(v_{b_1+i}),d'(v_{b_2-i})\}$.  Let $x_i$ and $x'_i$ denote these two
values.

If $x_i=x'_i$ for $1\le i\le h/2$, then we already know the list, which is
unchanged by reversal.  Otherwise, it suffices to know, for each $i$ and $j$
with $1\le i<j\le\FL{h/2}$ such that $x_i\ne x'_i$ and $x_j\ne x'_j$,
which of $\{x_i,x'_i\}$ is in the same half of the segment
$(d'(v_{b_1}),\ldots,d'(v_{b_2}))$ with which of $\{x_j,x'_j\}$.  To determine
this for a given pair $(i,j)$, consider $\#H_{i,j}/M$.  The value is
either $x_ix_j+x'_ix'_j$ or $x_ix'_j+x'_ix_j$.  Since $x_i\ne x'_i$ and
$x_j\ne x'_j$, the value when the larger elements in each pair are multiplied
is bigger, so we can determine the correct pairing.
We now know the list $(d(v_{b_1}),\ldots,d(v_{b_2}))$ up to reversal.

We want to extend this list outward, still only up to reversal.  Having
extended to $(v_{b_1-i+1},\dots,v_{b_2+i-1})$, consider the next step.
Since $i< b_1\le(r-1)/2$ and by Section~\ref{2maxbaton} we know the short
maximal batons, we know the values in the set
\begin{equation*}
\{\deg(v_{b_1+i}),\deg(v_{b_1-i}),\deg(v_{b_2+i}),\deg(v_{b_2-i}) \}. 
\end{equation*}
Since $b_1<b_1+i<b_2+i$ and $b_1-i<b_2-i<b_2$, in this set we already know
the pair $\{d(v_{b_1+i}),d(v_{b_2-i})\}$.  Hence we only need to determine
which value is which in the known pair $\{d(v_{b_1-i}),d(v_{b_2+i})\}$,
subject to the ordering of $(d(v_{b_1-i+1}),\ldots,d(v_{b_2+i-1}))$.
If the two values are equal, or if the list before this point is symmetric,
then there is nothing to do, since we are only determining the list up
to reversal.

Let $j$ be the least value such that $d(v_{b_1+j})\ne d(v_{b_2-j})$.
(If there is no such value of $j$ at most $h/2$, then see the variation
discussed below.)  By analogy with the earlier use of $H_{i,j}$, let
$F_{i,j}$ be the triton $B_{i,j:2,d_3+1,2}$.  Because $j<h/2$ and
$b_1<\min\{b_2,\bb_2\}$, we have $i+j<(r-1)/2$.  Since also $d_3\le(n-r+1)/3$,
the triton $F_{i,j}$ fits in a card, and we can count its occurrences.

The central key vertex in a copy of $F_{i,j}$ must be $v_{b_1}$ or $v_{b_2}$.
Letting in general $x_t=d(v_t)-1$, when $i\ne j$ we have
$$
\#F_{i,j}=\CH{d_1-2}{d_3-1}[x_{b_1-i}x_{b_1+j}+x_{b_1+i}x_{b_1-j}
+x_{b_2-i}x_{b_2+j}+x_{b_2+i}x_{b_2-j}],
$$
while
$$
\#F_{j,j}=\CH{d_1-2}{d_3-1}[x_{b_1-j}x_{b_1+j}+x_{b_2-j}x_{b_2+j}].
$$
Since $x_{b_1+j}\ne x_{b_2-j}$, the count $\#F_{j,j}$ distinguishes
between $x_{b_1-j}$ and $x_{b_2+j}$.
Since we also know the set $\{d(v_{b_1-i}),d(v_{b_2+i})\}$ and the ordered pair
$(d(v_{b_1+i}),d(v_{b_2-i}))$ for the fixed ordering
$(d(v_{b_1}),\ldots,d(v_{b_2}))$, from the count $\#F_{i,j}$ we can now
determine which way to allocate $\{d(v_{b_1-i}),d(v_{b_2+i})\}$.

If there is no such $j$, then the portion of the caterpillar between 
$v_{b_1}$ and $v_{b_2}$ is symmetric.  Let $j$ be the least index such that
$d(v_{b_1-j})\ne d(v_{b_2+j})$.    Letting $i$ be a larger index such that 
$d(v_{b_1-i})\ne d(v_{b_2+i})$, we use the triton $B_{i-j,j:2,3,d_3+1}$ in a
argument like that above.

When we reach $i=b_1$, there is no vertex $v_{b_1-i}$, and the term
$x_{b_1-i}x_{b_1+j}$ in the count for $F_{i,j}$ is replaced by $0$.
This allows us to fix the orientation (forward or backward) for 
$(d(v_1),\ldots,d(v_{b_2+b_1-1}))$.  If $b_2> (r+1)/2$, then the level pairs
now complete the degree list, while if $b_2\le(r+1)/2$ then the Climbing Lemma
(Lemma~\ref{climbing}) applies.
\end{proof}

	

\begin{lemma}\label{Case2}
If $b_1 \leq r/4 - 2$ and $b_2\ne\bb_1 $ (Case 2), then the caterpillar is
reconstructible.
\end{lemma}
\begin{proof}
By Lemma~\ref{b1b2}, we know $b_1$ and $b_2$.  By their definition, $b_2<\bb_1$.
Again by Sections~\ref{2maxbaton} and~\ref{highr} we know the short maximal
batons and the level pairs.  In particular, we know the short maximal batons
having a key vertex of degree $d_1$.  Therefore, for $i\le(r-1)/2$ we know
the existing values in the multiset
\begin{equation}\label{4deg}
\{\deg(v_{b_1+i}),\deg(v_{b_1-i}),\deg(v_{b_2+i}),\deg(v_{b_2-i}) \}. 
\end{equation}
By ``existing'', we mean that we drop $d(v_{b_1-i})$ from the set when
$i\ge b_1$, we drop $d(v_{b_2}-i)$ when $b_2\le i\le(r-1)/2$, and we drop
$d(v_{b_2+i})$ when $b_2>(r+1)/2$ and $i\ge \bb_2$.

Let $h=\min\{\FL{r/4-1},\bb_2-1\}$; note that $h>b_1$.  We can assume
$b_2-b_1\ge r/3-3$ by not being in Case 1, so automatically $h<b_2$.
If we allowed $h\ge\bb_2$ (that is, $b_2+h>r$), then we could not have a
baton or triton with key vertex $v_{b_2+h}$.

Consider the triton $H$ given by $H=B_{h,h:2,d_3+1,2}$.  The
high-degree vertex in a copy of $H$ cannot be $v_{b_1}$ because $v_{b_1}$ is
too close to $v_1$.  Hence it must be must be $v_{b_2}$, and
$\#H=\CH{d_1-2}{d_3-1}(d(v_{b_2-h})-1)(d(v_{b_2+h})-1)$.
We can compute $\#H$ because $H$ is small enough to fit in a card.
Using $d_3\le(n-r+1)/3$ and $r\le n-8$ from Lemma~\ref{2maxrd3},
\begin{equation}\label{Hbound}
|V(H)|=2h+2+d_3\le\FR r2+\FR{n-r+1}3
\le\FR {n+1}3+\FR r6\le \FR{n+1}3+\FR{n-8}6=\FR n2-1 .
\end{equation}
Thus in fact we can also count subgraphs obtained from $H$ by adding two
vertices.

Let $M=\#H/\CH{d_1-2}{d_3-1}$; we use $M$ to determine $d(v_{b_1+h})$.  
From \eqref{4deg}, we know the three values in the multiset
$\{d(v_{b_1+h})-1,d(v_{b_2-h})-1,d(v_{b_2+h})-1\}$.  Taking their product
and dividing by $M$ produces $d(v_{b_1+h})-1$.  Knowing $d(v_{b_1+h})$, we also
obtain $\{d(v_{b_2+h}),d(v_{b_2-h})\}$.

Next we find $\{d(v_{b_2+i}),d(v_{b_2-i})\}$ for $1\le i<h$.
Let $x=d(v_{b_2-i})-1$ and $y=d(v_{b_2+i})-1$.
Let $H'$ be obtained from $H$ by adding a pendant edge at one vertex at 
distance $i$ along the spine from the vertex of degree $d_3+1$.
Let $H''$ be obtained from $H$ by adding a pendant edge at both such vertices.
As observed above, $H'$ and $H''$ both fit in cards.
Now $x+y=\#H'/\#H$ and $xy=\#H''/\#H$.  Knowing their sum and product,
we find $x$ and $y$ and the desired degree pair.

Knowing these degree pairs for vertices at fixed distance from $v_{b_2}$, we
next determine the list $(d(v_{b_2-h}),\ldots,d(v_{b_2+h}))$, up to reversal.
Let $i$ be the least positive index such that $i<h$ and
$d(v_{b_2-i})\ne d(v_{b_2+i})$ (if there is no such index, then there is
nothing to do).  For $j$ with $i<j<h$ and $d(v_{b_2-j})\ne d(v_{b_2+j})$, we
need to determine which of $\{d(v_{b_2-j}),d(v_{b_2+j})\}$ to associate with
which of $\{d(v_{b_2-i}),d(v_{b_2+i})\}$ as degrees of vertices at distance
$i$ and $j$ from $v_{b_2}$ on the same side of $v_{b_2}$.  Consider the graph
$H_{i,j}$ obtained from $H$ by adding one pendant edge each at the two vertices
at distances $i$ and $j$ in the same direction from the vertex of degree
$d_3+1$.  We know that $H_{i,j}$ fits in a card, and
$$
\#H_{i,j}/\#H=(d(v_{b_2-i})-2)(d(v_{b_2-j})-2)+(d(v_{b_2+i})-2)(d(v_{b_2+j})-2).
$$
We can determine which among the values
$\{d(v_{b_2-i}),d(v_{b_2+i}),d(v_{b_2-j}),d(v_{b_2+j})\}$ belong to vertices on
the same side of $b_2$, because $ac+bd$ differs from $ad+bc$ when $a\ne b$ and
$c\ne d$.  With this we have the list $(d(v_{b_2-h+1},\ldots,d(v_{b_2+h-1}))$,
up to reversal.

If $d(v_{b_2-h})=d(v_{b_2+h})$, or if $(d(v_{b_2-h+1},\ldots,d(v_{b_2+h-1}))$
is unchanged under reversal, then we can extend the list to include $v_{b_2-h}$
and $v_{b_2+h}$.  Otherwise, we consider slightly different cards.  Using the
same value $i$ as in the previous paragraph, we have $x\ne x'$ and $y\ne y'$,
where $\{x,x'\}=\{d(v_{b_2-h})-1,d(v_{b_2+h})-1\}$ and
$\{y,y'\}=\{d(v_{b_2-i})-2,d(v_{b_2+i})-2\}$.
Let $H^*$ be the graph obtained from the triton $B_{h,h:3,d_3+1,2}$ by
adding a pendant edge at the vertex having distance $i$ from the central
key vertex in the direction away from the key vertex of degree $3$.  Note
that $H^*$ has two more vertices than $H$ and hence fits in a card.  Also,
the high-degree vertex in any copy of $H^*$ must be $v_{b_2}$.  Thus
$$
\#H^*/\CH{d_1-2}{d_3-1}= \CH x2y'x'+xy\CH{x'}2 =\FR{xx'}2 [(x-1)y'+(x'-1)y] .
$$
Again since $ac+bd$ differs from $ad+bc$ when $a\ne b$ and $c\ne d$, this
computation tells us the proper assignment of $d(v_{b_2-h})$ and $d(v_{b_2+h})$
to the ends of the list, yielding $(d(v_{b_2-h}),\ldots,d(v_{b_2+h}))$ up to
reversal.

Now we can determine the exact list $d(v_1),\ldots,d(v_{b_1+h})$.  From the set
described in~\eqref{4deg} for $i<b_1$, delete the pair
$\{d(v_{b_2+i}),d(v_{b_2-i})\}$ to obtain the pair
$\{d(v_{b_1+i}),d(v_{b_1-i})\}$.  For $b_1\le i\le h$, in~\eqref{4deg} we have
triples rather than $4$-sets, and deleting the pair yields $d(v_{b_1+i})$
exactly.  For $i<b_1$, let $H_i=B_{i,h:2,d_3+1,2}$.
Since $H_i$ is smaller than $H$, it fits in a card and we know $\#H_i$.  In
copies of $H_i$ in $G$, the high-degree vertex must be $v_{b_1}$ or $v_{b_2}$.
Because we know $\{(d(v_{b_2-i}),d(v_{b_2+h})),(d(v_{b_2+i}),d(v_{b_2-h}))\}$,
we know the number $N$ of copies of $H_i$ in which the high-degree vertex is
$v_{b_2}$.  After subtracting those, in the remaining copies the key vertex at
distance $h$ from $v_{b_1}$ must be $v_{b_1+h}$, whose degree we know.
The other key vertex must be $d(v_{b_1-i})$.  Since we already know the pair
$\{d(v_{b_1-i}),d(v_{b_1+i}) \}$, from $\#H_i-N$ we obtain $d(v_{b_1-i})$ and
$d(v_{b_1+i})$.  Over all such $i$, we now know $d(v_1),\ldots,d(v_{b_1+h})$.


We want to expand our knowledge to $d(v_1),\ldots,d(v_{\FL{(r-4)/2}})$.
For $b_1<b_1+i\le\FL{(r-4)/2}$, let $F_i$ be the triton $B_{b_1,i:2,d_3+1,2}$.
Since $|V(F_i)|=b_1+i+2+d_3\le r/2+d_3$, and we proved $r/2+d_3\le(n-2)/2$
in~\eqref{Hbound}, subgraphs consisting of $F_i$ plus two extra vertices
fit in cards.  Since there is only one nonleaf vertex at distance $b_1$ from
$v_{b_1}$, the high-degree vertex in any copy of $F_i$ must be $b_2$.  Thus
\begin{equation}\label{Fi}
\#F_i/\CH{d_1-2}{d_3-1}
=(d(v_{b_2-b_1})-1)(d(v_{b_2+i})-1)
+(d(v_{b_2-i})-1)(d(v_{b_2+b_1})-1) .
\end{equation}
Let $F'_i$ be the graph obtained from $F_i$ by adding one pendant edge to
the vertex at distance $b_1$ from the high-degree vertex on the path of length
$i$ from that vertex.  Again the high-degree vertex in a copy of $F'_i$ in $G$
must be $b_2$.  Thus $\#F'_i/\CH{d_1-2}{d_3-1}$ equals
\begin{equation}\label{F'i}
(\deg(v_{b_2 - b_1})- 1)(\deg(v_{b_2 + b_1})- 2)(\deg(v_{b_2 + i})- 1) 
+(\deg(v_{b_2 + b_1})- 1)(\deg(v_{b_2 - b_1})- 2)(\deg(v_{b_2 - i})- 1).
\end{equation}
Summing the counts and dividing by the common factors, we have
\begin{equation}\label{FFi}
\FR{\#F_i+\#F'_i}{\CH{d_1-2}{d_3-1}(\deg(v_{b_2-b_1})-1)(\deg(v_{b_2+b_1})-1)}
=\deg(v_{b_2 + i})- 1+\deg(v_{b_2 - i})- 1.
\end{equation}
Since $b_1<h$, we know $\{d(v_{b_2-b_1}),d(v_{b_2+b_1})\}$, and hence we obtain
$x+y$ where $x=d(v_{b_2+i})-1$ and $y=d(v_{b_2-i})-1$.
By making similar counts starting with the triton $B_{b_1,i:2,d_3+1,3}$ and a
subgraph obtained by adding one vertex to it (which has two vertices more than
$B_{b_1,i:2,d_3+1,2}$), we obtain $\CH{x}2+\CH{y}2$.  When $x+y=a$ and
$\CH x2+\CH y2=b$, we have $x^2+y^2=2b+a$ and $(x+y)^2=a^2$, so
$xy=(a^2-2b-a)/2$.  Knowing the sum and product, we obtain $\{x,y\}$.

Thus we obtain $\{d(v_{b_2+i}),d(v_{b_2-i})\}$.  Deleting these two from
the triple provided by~\eqref{4deg} when $i\ge b_1$ tells us $d(v_{b_1+i})$.
If $b_2+i>r$, then the terms involving $v_{b_2+i}$ in~\eqref{Fi}
and~\eqref{F'i} disappear, and the ratio in~\eqref{FFi} simplifies to
$d(v_{b_2-i})-1$.  Now~\eqref{4deg} simplifies to the pair
$\{d(v_{b_1+i}),d(v_{b_2-i})\}$, and we still obtain $d(v_{b_1+i})$.

Since we can do this for $b_1+i\le (r-4)/2$, we now know
$(d(v_1),\ldots,d(v_{\FL{(r-4)/2}}))$.  Since we know the level pairs, we also
know the degrees in order for the last $\FL{(r-1)/2}$ vertices on the spine.
Now Lemma~\ref{3-11} applies to complete the reconstruction.
\end{proof}

	
We next consider Case 3.
We will need algebraic lemmas.  We mention only the special cases we use,
forgoing more general statements.  The first lemma is an instance of what are
called Newton's Identities or Newton--Girard Formulae, expressing elementary
symmetric functions in terms of sums of powers (see
https://en.wikipedia.org/wiki/Newton\%27s\_identities).

\begin{lemma}\label{powprod}
If the sum of the $k$th powers of nonnegative integers $x,y,z$ are known
for $k\in\{1,2,3\}$, then $xyz$ is known.
\end{lemma}
\begin{proof}
We have already used the analogue for two variables:
$x+y=\alpha$ and $x^2+y^2=\beta$ yields $xy=(\alpha^2-\beta)/2$.

For three variables, let $x+y+z=\alpha$, $x^2+y^2+z^2=\beta$, and
$x^3+y^3+z^3=\gamma$.  Straightforward computation yields
$6xyz=\alpha^3-3\alpha\beta+2\gamma.$
%
\end{proof}

\begin{lemma}\label{powbin}
The product $a^sb^t$ can be expressed as a linear combination of products
of the form $\CH ai\CH bj$ with $i\le s$ and $j\le t$.  In particular,
we use the following instances:
\begin{align*}
a^2b^2&=\left[2\CH a2+a\right]\left[2\CH b2+b\right],\\
a^3b^3&=\left[6\CH a3+6\CH a2+a\right]\left[6\CH b3+6\CH b2+b\right],\\
a^2b^4&=\left[2\CH a2+a\right]\left[24\CH b4+36\CH b3+14\CH b2+b\right].
\end{align*}
\end{lemma}
\begin{proof}
The pure powers $\{z^k\st k\ge0\}$ and the binomial coefficients
$\{\CH zk\st k\ge0\}$ both form bases for the set of polynomials in $z$.
Furthermore, within the subspace of polynomials of degree at most $k$, we can
write $z_k$ as a linear combination of $\{\CH zj\st 0\le j\le k\}$.  The linear
combination is found by successively matching the leading remaining coefficient
and restricting to smaller degree.
\end{proof}

\begin{lemma}\label{Case3}
If $b_1 \ge(r-7)/4$ and $b_2\ne\bb_1 $ (Case 3) and $n\ge48$,
then the unknown caterpillar is reconstructible.
\end{lemma}
\begin{proof}
By Lemma~\ref{Case1}, we may assume $b_2-b_1> r/3-7/2$.  From this we obtain
$b_1\le \bb_2-1<2r/3+7/2-b_1$, so $b_1<r/3+7/4$, which yields $b_1\le(r+5)/3$.
With the lower bound on $b_1$, we also have $b_2 > 7r/12-21/4>(r+1)/2$. 

As in Case 2 (Lemma~\ref{Case2}), we know the multiset of degrees of four
vertices as listed in~\eqref{4deg}.  Our initial aim is to find the multiset
of three degrees $\{d(v_{b_1+i}),d(v_{b_2-i}),d(v_{b_2+i})\}$ in order to
determine $d(v_{b_1-i})$, where $i$ is a fixed positive integer less than $b_1$.

Define integers $x_j$ and $y_j$ as follows:
\begin{center}
\begin{tabular}{c c c c}
$x_1=d(v_{2b_1})-1$& $x_2=d(v_{b_2-b_1})-1$& $x_3=d(v_{b_2+b_1})-1$& \\
$y_1=d(v_{b_1+i})-2$& $y_2=d(v_{b_2-i})-2$&
$y_3=d(v_{b_2+i})-2$&$y_4=d(v_{b_2-i})-2$.
\end{tabular}
\end{center}
Note that each $x_j$ is positive and each $y_j$ is nonnegative.
As noted above, from the batons of length $b_1$ we know the multiset
$\{x_1,x_2,x_3\}$, and from the batons of length $i$ we know
the multiset $\{y_1,y_2,y_3,y_4\}$.  If we can determine $x_1y_1x_2y_2x_3y_3$
and it is nonzero, then we can divide it into the product of all seven numbers
to obtain $y_4$.

We use special tritons.
Let $F_{s,t}=B_{b_1-i,i:1+s,2+t,d_3+1}$ for $s,t\in\{1,2,3\}$.  Note that
$$|V(F_{1,1})|=b_1+d_3+3<\FR {r+5}3+3+\FR{n-r+1}3\le \FR n3+5.$$
Thus $F_{1,1}$ fits in a card when $n\ge24$, and in general
$F_{s,t}$ fits when $n\ge 6(2+s+t)$.  We use instances with $s+t\le 6$,
which can be counted when $n\ge48$.

The high-degree vertex in a copy of $F_{s,t}$ in $G$ can only be
$v_{b_1}$ or $v_{b_2}$.  If $v_{b_1}$ then the other key vertices are
$v_{b_1+i}$ and $v_{2b_1}$.  If $v_{b_2}$, then the triton may extend to
$v_{b_2+b_1}$ or to $v_{b_2-b_1}$.  The parameters $s$ and $t$ indicate
the number of leaf neighbors of two key vertices in the triton.

Let $M=\CH{d_1-1}{d_3}$.  We compute
\begin{equation}\label{eq:triplebase}
\#F_{s,t}/M =\CH{x_1}s\CH{y_1}t+\CH{x_2}s\CH{y_2}t+\CH{x_3}s\CH{y_3}t .
\end{equation}
In particular, 
$\#F_{1,1}/M=x_1y_1+x_2y_2+x_3y_3.$
By Lemma~\ref{powbin},
\begin{equation}\label{xysq}
\Bigl[4\#F_{2,2}+2\#F_{1,2}+2\#F_{2,1}+\#F_{1,1}\Bigr]/M
=x_1^2y_1^2+x_2^2y_2^2+x_3^2y_3^2.
\end{equation}
and $x_1^3y_1^3+x_2^3y_2^3+x_3^3y_3^3$ is given by
\begin{equation*}\label{xycub}
\Bigl[36(\#F_{3,3}+\#F_{2,3}+\#F_{3,2}+\#F_{2,2})
+6(\#F_{1,3}+\#F_{3,1}+\#F_{1,2}+\#F_{2,1})+\#F_{1,1}\Bigr]/M .
\end{equation*}
We can compute these terms because $F_{3,3}$ and its
subgraphs fit in cards when $n\ge48$.

Letting $\alpha,\beta,\gamma$ respectively denote the resulting values using
first, second, and third powers, Lemma~\ref{powprod} yields
$$
x_1y_1x_2y_2x_3y_3=(\alpha^3-3\alpha\beta+2\gamma)/6.
$$
Since we know $\{x_1,x_2,x_3\}$ from the batons of length $b_1$, we can divide
by their product to obtain the product $y_1y_2y_3$.
We also know $\{y_1,y_2,y_3,y_4\}$ from the batons of length $i$.
If $y_1y_2y_3\ne0$, then we can divide to obtain $y_4$ and thus $d(v_{b_1-i})$.

If $y_1y_2y_3=0$, then we cannot perform this division, and at least one
factor is $0$.  Let $p$ and $q$ be the other two indices.  Our counts of
$F_{s,t}$ give expressions involving indices $p$ and $q$, since the third term
becomes $0$.  We obtain $x_py_p+x_qy_q$ from $\#F_{1,1}$ and
$x_p^2y_p^2+x_q^2y_q^2$ using the linear combination in~\eqref{xysq}.  Now
using the identity $2ab=(a+b)^2-(a^2+b^2)$ we obtain $x_py_px_qy_q$.

We seek also $x_py_p^2x_qy_q^2$; dividing it by $x_py_px_qy_q$ yields $y_py_q$,
and dividing this into the product of the three nonzero entries in the multiset
$\{y_1,y_2,y_3,y_4\}$ gives us $y_4$.

To obtain $x_py_p^2x_qy_q^2$, first let $a=x_py_p^2$ and $b=x_qy_q^2$
We have $[2\#F_{1,2}+\#F_{1,1}]/M=x_py_p^2+x_qy_q^2=a+b$.
Once we determine $a^2+b^2$, we again apply $2ab=(a+b)^2-(a^2+b^2)$.
To obtain $a^2+b^2$, we start with $\#F_{2,4}/M$ as in~\eqref{eq:triplebase}.
We obtain terms involving $x^2y^4$ and must eliminate lower order terms.
We can compute these terms because, as observed earlier, $F_{2,4}$ and its
subgraphs fit in cards when $n\ge48$.
Using the expression for $a^2b^4$ in Lemma~\ref{powbin},
$$
\Bigl[48\#F_{2,4}+24\#F_{1,4}+72\#F_{2,3}+36\#F_{1,3}
+28\#F_{2,2}+14\#F_{1,2}+2\#F_{2,1}+\#F_{1,1}\Bigr]/M
=x_p^2y_p^4+x_q^2y_q^4 .
$$

If two entries in $\{y_1,y_2,y_3\}$ are $0$, then we discover this
when computing $x_py_px_qy_q$.  If $y_p$ remains positive, then
then $\#F_{1,1}/M=x_py_p$ and $[2\#F_{1,2}+\#F_{1,1}]/M=x_py_p^2$.
Now the ratio of these two numbers is $y_p$, and the remaining nonzero entry in
$\{y_1,y_2,y_3,y_4\}$ is $y_4$.  If all entries in $\{y_1,y_2,y_3\}$ are $0$,
then $\#F_{1,1}/M=0$ and again we obtain $y_4$ as $\max\{y_1,y_2,y_3,y_4\}$.

This completes the discovery of $y_4$ and hence $d(v_{b_2-i})$.

After performing these computations for $i$ from $1$ to $b_1-1$,
we have the list $d(v_1),\ldots,d(v_{b_1})$ in order.
Next we want to extend our knowledge to determine $d(v_{b_1+i})$ for
$1\le i\le \bb_2-b_1$.  Letting $F'_{s,t}= B_{b_1,i:1+s,d_3+1,1+t}$, we apply
arguments similar to those above using $F'_{s,t}$ instead of $F_{s,t}$.
This portion is simpler because the $-2$ in the definition of $y_j$ is replaced
by $-1$, and all $y_j$ are positive.  Thus the arguments to obtain $y_4$ for
each $i$ remain the same as the case above where all $y_j$ are positive.

Since we know the short maximal batons, and those give us the level pairs, we
now know not only $(d(v_1),\ldots,d(v_{\bb_2}))$ but also
$(d(v_{b_2},\ldots,d(v_r))$.  With this, the Climbing Lemma
(Lemma~\ref{climbing}) completes the reconstruction.
\end{proof}

\begin{lemma}\label{Case4}
If $b_1 < r/3+9/4$ and $b_2=\bb_1$ (Case 4),
then the caterpillar is reconstructible.
\end{lemma}
\begin{proof}
We know the short maximal batons and the level pairs.
Our first goal is to determine $(d(v_{b_1}),\ldots,d(v_{b_2}))$.
If $d(v_k)=d(v_\kb)$ for $b_1<k\le(r+1)/2$, then there is nothing to do.
Otherwise, let $h$ be the least positive integer such that
$d(v_{b_1+h})\ne d(v_{b_2-h})$.  Since $d(v_{b_1})=d(v_{b_2})=d_1$, we may
assume by symmetry that $d(v_{b_1+h})< d(v_{b_2-h})$.  

We want to distinguish, for $h'$ such that $d(v_{b_1+h'})\ne d(v_{b_2-h'})$,
which degree belongs to which vertex.  There is no such $h'$ between $1$ and
$h$; consider $h'>h$.  For such $h'$ with $b_1\le h'\le(r-2)/2$, let
$H=B_{h,h'-h:d_3+1,d(v_{b_1+h})+1,2}$.  In any copy of $H$, the key vertex of
degree $d_3+1$ must be $v_{b_1}$ or $v_{b_2}$.  Since $H$ has length at least
$b_1$, the second key vertex must be $v_{b_1+h}$ or $v_{b_2-h}$.  Since it has
degree greater than the known value $d(v_{b_1+h})$, it must be $v_{b_2-h}$.
Hence the key vertex having degree $2$ in $H$ must be $v_{b_2-h'}$.
Letting $M=\CH{d_1-1}{d_3}\CH{d(v_{b_2-h})-2}{d(v_{b_1+h})-1}$, we have
$\#H=M[d(v_{b_2-h'})-1]$, which tells us $d(v_{b_2-h'})$.

To perform the computation, we must guarantee that $H$ fits in a card.
Note that $|V(H)|=h'+d_3+d(v_{b_1+h})+1$.  Since $d(v_{b_1+h})<d(v_{b_2-h})$,
we have $d(v_{b_1+h})<d_3$.  By~\eqref{d'}, any four degrees sum to at most
$n-r+6$.  Hence
$$
2d_3+2d(v_{b_1+h})\le d_1-1+d_2-1+d_3-1+d_4\le n-r+3.
$$
If equality holds throughout, then there are only four branch vertices and we
already know the full caterpillar.  Hence we can reduce the upper bound by $1$.
Thus
$
2|V(H)|\le r+(n-r+2)= n+2 ,
$
and $H$ fits in a card.  Having assigned the degrees in the level pairs
from level $2b_1$ up to $b_1+h'$ when $h'\le(r-2)/2$,
we obtain all of $(d(v_{2b_1}),\ldots,d(v_{b_2-b_1}))$, since
$b_1+(r-2)/2\ge r/2$.  (If $r$ is odd and the largest $h'$ is $(r-3)/2$ with
$b_1=1$, then there is only one element in the remaining level $(r+1)/2$.)

To enlarge the known list to $(d(v_{b_1}),\ldots,d(v_{b_2}))$, consider
$h<h'<b_1$ with $d(v_{b_1+h'})\ne d(v_{b_2-h'})$.  Let $H'$ be the graph
obtained from $B_{h,b_1-h:d_3+1,d(v_{b_1+h})+1,2}$ by adding a pendant edge at
the vertex whose distance from the key vertex of degree $d_3+1$ is $h'$.

Arguing as before, copies of $H'$ have key vertices $v_{b_2}$, $v_{b_2-h}$, and
$v_{b_2-b_1}$, and the added pendant edge is at $v_{b_2-h'}$.  Thus
$\#H'=M[d(v_{b_2-b_1})-1][d(v_{b_2-h'})-2]$.
Since we already know $d(v_{b_2-b_1})$ from the previous step, $\#H'$ determines
$d(v_{b_2-h'})$.
Also $|V(H')|=b_1+d_3+d(v_{b_1+h})+1$.  We have the same bound on the degrees
as for $H$, and the bound on $b_1$ is smaller than the earlier bound
$(r-3)/2$ on $h$, so $H'$ also fits in a card.

Now consider the levels below $b_1$.  That is, we may have $h'$ with $1<h'<b_1$
such that $d(v_{b_1-h'})\ne d(v_{b_2+h'})$.
Let $H''=B_{h',h:2,d_3+1,2}$.  In any copy of $H''$, the central key vertex
must be $v_{b_1}$ or $v_{b_2}$, but there may be several ways for $H''$ to
appear.  Note first that $h'+h\le(r-2)/2$, by the definition of $h$, so
$$
|V(H'')|\le \FR{r-2}2+d_3+1\le \FR r2+\FR{n-r+1}3=\FR{n+1}3+\FR r6
\le\FR{n+1}3+\FR{n-8}6=\FR{n-2}2.
$$
Next, let $x_i=d(v_i)-1$.  When $h\ge b_1$, the path of length $h$ from 
$v_{b_1}$ or $v_{b_2}$ in the triton must extend to higher levels.  In that
case
\begin{equation}\label{Hcase4}
\#H''=\CH{d_1-2}{d_3-1}[x_{b_1+h}x_{b_1-h'}+x_{b_2-h}x_{b_2+h'}]
\end{equation}
Since $x_{b_1+h}\ne x_{b_2-h}$ and $x_{b_1-h'}\ne x_{b_2+h'}$,
we learn which of $\{x_{b_1-h'},x_{b_2+h'}\}$ belongs to the vertex on the same
side of the caterpillar as $v_{b_1+h}$.

Now suppose $h<b_1$.  In the special case $h'=h$, whether $x_{b_1-h'}$
equals $x_{b_2+h'}$ or not, again $H''$ can appear in only two ways,
and the formula for $\#H''$ is again~\eqref{Hcase4}.  Thus we determine
$x_{b_1-h}$ and $x_{b_2+h}$.  Finally, in the general case with $h<b_1$,
we have
$$
\#H''=\CH{d_1-2}{d_3-1}[x_{b_1+h}x_{b_1-h'}+x_{b_1-h}x_{b_1+h'}+
x_{b_2-h}x_{b_2+h'}+x_{b_2+h}x_{b_2-h}].
$$
We know everything in this formula except which of $\{x_{b_1-h'},x_{b_2+h'}\}$
is which.  Again since also $x_{b_1+h}\ne x_{b_2-h}$, matching the actual
count tells us that information.

\medskip
Resolving this for $1\le h'<b_1$ completes the reconstruction in the case
where the parameter $h$ that we have defined actually exists; that is, where
there is a level pair $\{d(v_{b_1+h}),d(v_{b_2-h})\}$ with $h>0$ consisting
of distinct degrees.  If $(d(v_{b_1}),\ldots,d(v_{b_2}))$ is unchanged under
reversal, then the argument above does not distinguish the degrees in the
level pairs lower than $b_1$.

If no level pairs have distinct degrees, then there is nothing to do.
Hence in the remaining case we can choose the least $h$ such that
$\{d(v_{b_1-h}),d(v_{b_2+h})\}$ has distinct degrees.  Since all levels above
that are symmetric, we may index these so that $d(v_{b_1-h})<d(v_{b_2+h})$.
Now consider any $h'$ with $h<h'<b_1$ such that
$\{d(v_{b_1-h'}),d(v_{b_2+h'})\}$ has distinct degrees.
Let $F=B_{h,h'-h:d_3+1,3,2}$; note that $F$ fits in a card.
In any copy of $F$ in $G$, the key vertex of high degree must be $v_{b_1}$ 
or $v_{b_2}$.  The other key vertices may occur in four ways.
With $x_i=d(v_i)-1$ and $y_i=d(v_i)-2$, we have
$$
\#F=\CH{d_1-1}{d_3}[y_{b_1-h}x_{b_1-h'}+y_{b_1+h}x_{b_1+h'}
+y_{b_2+h}x_{b_2+h'}+y_{b_2-h}x_{b_2-h'}] .
$$
The only unknown quantities in the formula are $x_{b_1-h'}$ and $x_{b_2+h'}$,
though we know the unordered pair.  Since $x_{b_1-h'}\ne x_{b_2+h'}$
and $y_{b_1-h}\ne y_{b_2+h}$, matching the actual count tells us which elements
of the pairs belong to vertices in the same half of the caterpillar.
\end{proof}

\section{Unique Maximum-degree Vertex}\label{1max}

In this section, we consider the case $d_1>d_2\ge d_4\ge3$ and assume these
inequalities throughout the section, along with being
in the high-diameter case $r\ge (n-5)/2$.  We recognize these properties
from the known degree list $\VEC d1n$ in nonincreasing order.
Notation is as in the preceding sections.  One subcase requires a
restriction to $n\ge17$.

\begin{lemma}\label{cond}
Fix $q$ with $q\le(r-1)/2$.  If the deck determines all maximal batons in $G$
having $w_1$ as a key vertex and length at most $q$, then it determines all
maximal batons with length at most $q$.
\end{lemma}
\begin{proof}
Let $\alpha=\FL{(n-r+1)/4}$.  If $j\le(r-1)/2$, then
\begin{equation}\label{batfit}
\C{V(B_{j:\alpha+1,\alpha+1})}=j+2\alpha+1\le \FR{r-1}2+\FR{n-r+3}2=\FR{n+2}2.
\end{equation}
Thus every short baton with maximum degree at most $\alpha+1$ fits in a card.
If $d_2\le\alpha+1$, then we see all short batons not given by the hypothesis.
Hence we may assume $d_2>\alpha+1$.

\bigskip
{\bf Case 1:}  {\it $d_3 <d_2$}.
Here $d_1+d_2+d_3+d_4\ge 3d_3+6$.  Thus by~\eqref{d'}, $n-r-2\geq 3d_3-2$, so
$n-r\ge3d_3\ge9$.  Let $\alpha'=\FL{(n-r)/3}$.  Now when $j\le(r-1)/2$ we have
\begin{equation}\label{batfit2}
\C{V(B_{j:\alpha'+1,3})}\le\FR{r-1}2+3+\FR{n-r}3
=\FR{r+5}2+\FR{n-r}2-\FR{n-r}6\le\FR{n+5}2-\FR32=\FR{n+2}2.
\end{equation}
Thus $B_{j:\alpha'+1,3}$ fits in a card, and we can count its copies in $G$.

Since $\alpha'\ge d_3$, every copy of $B_{j:\alpha'+1,3}$ in $G$ lies in a copy
of $B_{j:d_1,3}$ or $B_{j:d_2,3}$, and similarly for copies of 
$B_{j:\alpha'+1,2}$.  For $j\leq q$, by the hypothesis we know the batons
having $w_1$ as a key vertex.  Hence we know the number of copies of
$B_{j:\alpha'+1,3}$ and $B_{j:\alpha'+1,2}$ contained in such batons.
Let $N_3$ and $N_2$, respectively, be the number of these batons remaining
(not containing $w_1$ as a key vertex).  These are contained in maximal
$j$-batons with $w_2$ as a key vertex.  Since only one vertex has degree $d_2$,
there are at most two such maximal batons.

Suppose that vertices $v_i$ and $v_{i'}$ occur at distance $j$ from $w_2$.
We obtain $N_2/\CH{d_2-1}{\alpha'}=(d(v_i)-1)+(d(v_{i'})-1)$ and
$N_3/\CH{d_2-1}{\alpha'} ={d(v_i)-1\choose 2}+{d(v_{i'})-1\choose 2}$.
Hence we obtain $\{d(v_i),d(v_{i'})\}$.  If there is no solution, then only
one of $\{v_i,v_{i'}\}$ exists, and $N_2$ determines its degree.

Now we know also the maximal $j$-batons having $w_2$ as a key vertex.
If $d_3\le\alpha+1$, then all remaining short batons fit in cards,
so we may assume $d_3\ge\alpha+2$.

If also $d_4\ge\alpha+2$, then $d_1+d_2+d_3+d_4\ge4(\alpha+2)+3$.
By~\eqref{d'}, we have
$$n-r-2\ge4\alpha+3\ge4(n-r-2)/4+3=n-r+1.$$
By this contradiction, we conclude $d_4\le\alpha+1$, and hence $d_4<d_3$.

Now $w_3$ is the only vertex of degree $d_3$.
By repeating the earlier argument with $d_3$ in place of $d_2$ and 
$d_4+1$ in place of $d_3+1$, we obtain the maximal $j$-batons whose key vertex
of largest degree is $w_3$.  Since $d_4\le\alpha+1$, we then see all the
remaining $j$-batons.

\medskip
{\bf Case 2:} {\it $d_3 =d_2$}.
If $d_4\geq \alpha+2$, then $d_1+d_2+d_3+d_4\ge4\alpha+9$.  By~\eqref{d'},
\begin{equation}\label{d234}
n-r-2\geq 4\alpha+1\geq 4(n-r-2)/4+1=n-r-1,
\end{equation}
a contradiction.  Thus $d_4\leq 1+\alpha<d_2=d_3$.

If $d_4\leq \alpha$, then each copy of $B_{j:\alpha+1,\alpha+1}$ in $G$ has both
key vertices in $W'$.  By hypothesis, we know how many of these have $w_1$ as
a key vertex.  Any remaining copy of $B_{j:\alpha+1,\alpha+1}$
must be contained in the unique maximal $j$-baton with key vertices of degree
$d_2$.  Therefore, when some copy of $B_{j:\alpha+1,\alpha+1}$ exists beyond
those with key vertex $w_1$, we know that the distance between $w_2$ and $w_3$
in $G$ is $j$.  The unique copy of $B_{j:\alpha+1,\alpha+1}$ with key vertices
$\{w_2,w_3\}$ contains $2\alpha$ copies of $B_{j:\alpha+1,\alpha}$.  Excluding
those $2\alpha$ copies (and those arising from maximal batons with key vertex
$w_1$) allows us to find the maximal $j$-batons with key vertices of degrees
$d_2$ and $\alpha$.  Continuing the exclusion argument allows us to find all
the maximal $j$-batons having a key vertex of degree $d_2$.
All remaining short batons fit in cards.

Finally, suppose $d_4= \alpha+1$.  We still have the same lower bounds on
$d_1,d_2,d_3$ as in~\eqref{d234}, so instead of that computation we get
$$
n-r-2\ge d_1+d_2+d_3+d_4-8 \geq 4\alpha\geq n-r-2.
$$
This equality requires all of the following: $r'=4$,
$d_1=1+d_2=1+d_3=2+d_4=\alpha+3$
and $\alpha=(n-r-2)/4$.

In this case
$$
\C{V(B_{j:\alpha+2,\alpha+1})}=j+2\alpha+2\le\FR{r-1}2+\FR{n-r+2}2=\FR{n+1}2,
$$
so these batons fit in cards.  Every such baton in $G$ not having $w_1$ as an
key vertex is contained either in $B_{j:d_2,d_2}$ (which occurs at most once)
or in a copy of $B_{j:d_2,\alpha+1}$ that is a maximal baton (since $r'=4$,
there are at most two of these.  Since $d_2=\alpha+2$, the instance of
$B_{j:d_2,d_2}$ (if it exists) contains $2(\alpha+1)$ copies of
$B_{j:\alpha+2,\alpha+1}$, and every instance of $B_{j:d_2,\alpha+1}$
contains only one copy of $B_{j:\alpha+2,\alpha+1}$ (namely, itself).
Note that $2(\alpha+1)\ge6$, since $w_4$ is a branch vertex.
Excluding the copies that come from larger batons, the number
of remaining copies of $B_{j:\alpha+2,\alpha+1}$ thus tells us whether we have
a copy of $B_{j:d_2,d_2}$ as a maximal baton and how many maximal batons
are copies of $B_{j:d_2,\alpha+1}$.  The remaining short maximal batons fit
in cards.
\end{proof}

In light of Lemma~\ref{cond}, our first aim is to find
the short maximal batons containing $w_1$.


\begin{lemma}\label{d2small}
If $d_2\le(n-r-3)/2$, then the unknown caterpillar is reconstructible from
$\cD$.
\end{lemma}
\begin{proof}
Initially, the level of the maximum-degree vertex $w_1$ is unknown.
By symmetry, we may call it $h$, so that $w_1=v_h$ with $h\le(r+1)/2$.

Fix $j$ with $1\le j\le(r-1)/2$, and let $H=B_{j:d_2+1,2}$ and
$H'=B_{j:d_2+1,3}$.  Since
$$
\FR{r-1}2+\FR{n-r-3}2+3=\FR{n+2}2,
$$
the batons $H'$ and
$H$ fit in cards.  Since $d_2\ge3$, each copy of $H$ or $H'$
has $w_1$ as its higher-degree key vertex.  The other key vertex is at
distance $j$ from $w_1$.  Let $M=\CH{d_1-1}{d_2}$.  When $j<h$, we have
$\#H=x-1+y-1$ and $\#H'=\CH{x-1}2+\CH{y-1}2$.  where $x$ and $y$ are the
degrees of the two vertices at distance $j$ from $w_1$.  When $j\ge h$, we have 
$\#H=x-1$ and $\#H'=\CH{x-1}2$, where $x=d(v_{h+j})$.
In particular, we have $\CH{\#H}2=\#H'$ when $j\ge h$, but
$\CH{\#H}2>\#H'$ when $j<h$, so these cases are distinguishable.

Considering this over all such $j$, we find not only $h$ (including when
$h=\CL{r/2}$, by satisfying the case $j<h$ for all $j$ up to $\FL{(r-1)/2}$),
but also the maximal baton(s) of length $j$ containing $w_1$.  By
Lemma~\ref{cond}, the deck determines all short maximal batons.
Hence by Lemma~\ref{levelpair} it also determines the level pairs
$\{d(v_k),d(v_\kb)\}$ for $1\le k\le r/2$.  We now consider three cases
depending on the value of $h$.  Let $j_1=\FL{(r-1)/2}$.

\medskip
{\bf Case 1:}  $h\leq (r-1)/2$.
As observed above, we already know $d(v_{2h}),\ldots,d(v_{h+j_1})$; note that
$h+j_1\ge(r+1)/2$.  Let $H_i=B_{i,j_1-i:d_2+1,3,2}$ for $1\le i<h$.  Note
that the triton $H_i$ has the same number of vertices as the baton $H'$ and
hence also fits in a card.  Because $H_i$ has length at least $h$, in any copy
of it the key vertices are $v_h$, $v_{h+i}$, and $v_{h+j_1}$.  Since we already
know the degree of $v_h$ and $v_{h+j_1}$, the count $\#H_i$ tells us
$d(v_{h+i})$.  Since we knew $\{d(v_{h+i}),d(v_{h-i})\}$ from
$\#H$ and $\#H'$ when $j=i$, we now also know $d(v_{h-i})$.
After doing this for $1\le i<h$, we know $(d(v_1),\ldots,d(v_{\CL{r/2}}))$,
and the level pairs complete the reconstruction.

\medskip
{\bf Case 2:}  $h=r/2$. Since
$|V(B_{r/2:d_2+1,2})|= {r/2}+d_2+2\le{r/2}+(n-r+1)/2=(n+1)/2$,
this baton fits in a card.  Its key vertices can be only $v_{r/2}$ and $v_r$.
The count thus tells us $d(v_r)$ and hence $d(v_1)$.  Next we iterate the
following from $i=1$ to $i=r/2-1$.  The count $\#B_{r/2-i:d_2+1,2}$ tells us
$d(v_{i})+d(v_{r-i})$ (minus $2$).  We already know $d(v_i)$ from the previous
step, so we learn $d(v_{r-i})$ from the count and then $d(v_{i+1})$ from
the level pair (since $\ov{r-i}=i+1$).  Continuing, we learn the degrees of all
vertices.

\medskip
{\bf Case 3:}  $h= (r+1)/2$.
If at most one level pair has distinct values, then we know the caterpillar.
Let $k$ be the highest level with unequal values and $i$ be another such level;
thus $d(v_i)\neq d(v_\ib)$ and $d(v_k)\neq d(v_\kb)$ with $i<k\le(r-1)/2$. 
Let $F=B_{h-k,k-i:d_2+1,3,2}$.  Since
$$
|V(F)|=h-i+d_2+3\leq \FR{r+1}2-1+\frac{n-r-3}{2}+3\leq\frac{n+2}{2},
$$
$F$ fits in a card.  The key vertices in any copy of $F$ are
$\{v_{h},v_k,v_{i}\}$ or $\{v_{h},v_{\kb},v_{\ib}\}$. Thus,
$$
\#F
={d_1-1\choose d_2}[(d(v_k)-2)(d(v_{i})-1)+(d(v_{\kb})-2)(d(v_{\ib})-1)].
$$
Again since $ac+bd$ differs from $ad+bc$ when $a\ne b$ and $c\ne d$, this
computation tells us which of $\{d(v_i),d(v_\ib)\}$ in level $i$ lies in the
same half of the caterpillar with the larger of $\{d(v_k),d(v_\kb)\}$ in
level $k$.  Doing this over all levels $i$ with unequal values completes
the reconstruction.
\end{proof}

When $d_2$ is too big for Lemma~\ref{cond}, the vertices $w_1$ and $w_2$
absorb almost all the leaves, and the remaining possibilities for the other
branch vertices are very restricted.

\begin{lemma}\label{uncond}
If $d_2\ge(n-r-2)/2$ and $n\ge17$, then the unknown caterpillar is
reconstructible from $\cD$.
\end{lemma}
\begin{proof}
First consider the possible degrees of branch vertices.
With $d_1>d_2$ and $d_3\ge d_4\ge3$, \eqref{d'} and the lower bound on $d_2$
yield
\begin{equation}\label{d2bound}
n-r-2\geq 1+2(d_2-2)+\SE i3r(d_i-2)\ge 1+2(d_2-2)+2\ge n-r-3 .
\end{equation}
In particular, $d_2=\FL{(n-r-1)/2}$.
Because the difference between the expressions on the ends of~\eqref{d2bound}
is $1$, we obtain
\begin{equation}\label{last}
(d_1-d_2)+\sum_{i=3}^r(d_i-2)\leq 4.
\end{equation}

With at least four branch vertices, we now have limited choices in
satisfying~\eqref{last}.  In each case below there are two main steps.
First we find the short maximal batons where $w_1$ is a key vertex, and
thus by Lemma~\ref{cond} we know all the short maximal batons.  We then
find the shorter maximal tritons, and hence by Lemma~\ref{3-11} the
caterpillar is reconstructible from $\cD$.
Let $R'$ denote the set of branch vertices.

\medskip
{\bf Case 1:} {\it $d_2=3$.}
By~\eqref{last} and $d_4\ge3$, there are only three possibilities for the
degrees of vertices in $R'$: $(5,3,3,3)$, $(4,3,3,3)$ and $(4,3,3,3,3)$. 
In all these cases, $n-r-2=\SE i1r(d_i-2)\ge5$.  Hence for $j\leq (r-1)/2$
we have
$$
\C{V(B_{j:3,3})}=j+5\le \FR{r-1}2+5 =\FR{n+2}2-\FR{n-r-7}2\le\FR{n+2}2.
$$
Thus $B_{j:3,3}$ fits in a card.
Similarly, for $j \leq (r-3)/2$ the baton $B_{j:4,3}$ fits in a card.
The vertex of degree $4$ in any copy of $B_{j:4,3}$ must be $w_1$ (which may
have degree $5$ in $G$).  Thus, $\#B_{j:4,3}$ tells us $\#B_{j:d_1,3}$ when
$j \leq (r-3)/2$.
 
It remains to consider the one integer $j$ in $\{(r-2)/2,(r-1)/2\}$.
If $\#B_{j:3,3}\leq 2$, then $B_{j:4,3}\not\esub G$, since a copy of
$B_{j:4,3}$ yields three copies of $B_{j:3,3}$ in the deck.
If $\#B_{j:3,3}=3$, then $\#B_{j:4,3}=1$, since we have at most four vertices
of degree $3$, and forming three copies of $B_{j:3,3}$ using them would put
distance $3j$ between the farthest vertices of degree $3$.  That requires
$r\ge 1+3j\ge (3r-4)/2$, simplifying to $r\le 4$, leaving no room on the spine
for $w_1$.

When $\#B_{j:3,3}\geq 4$, in all cases we can determine $\#B_{j:d_1,3}$.
Using an exclusion argument we also find $\#B_{j:d_1,2}$, since 
$B_{j:4,2}$ fits in a card.

\smallskip
{\bf Subcase 1a:} {\it Vertex degrees in $R'$ are $(4,3,3,3)$}.
Consider the shorter tritons, with $j+j'\le (r-4)/2$ and $j\le j'$.
Let $H=B_{j,j':3,3,3}$, $F=B_{j,j':4,3,3}$, $\hF=B_{j,j':3,3,4}$, and
$F'=B_{j,j':3,4,3}$, 
Here $n-r-2=\sum_{i=1}^r(d_i-2)=5$, so when $j+j'\leq (r-4)/2$ we have
$$
|V(H)|=\FR{r-4}2+6 =\frac{n+2}{2}-\frac{n-r-6}{2}<\frac{n+2}{2},
$$
and hence we know $\#H$.  The only larger tritons that may contain $H$ are $F$
and $F'$, and they contain three and two copies of $H$, respectively.  Because
$G$ has only three vertices of degree $3$, there can only be one copy of $H$
not contained in $F$ or $F'$.  Using these observations, if $\#H=1$, then
neither $F$ nor $F'$ occurs in $G$.  If $\#H=2$, then $F$ does not appear and
$F'$ must appear.  If $\#H=3$, then $F$ (or $B_{j,j':3,3,4})$ appears (once)
and $F'$ does not, except in the special situation where the degrees appear in
the order $(3,4,3,3)$ with distances $(j,j'-j,j)$, where one copy of $F'$
appears.  These possibilities are distinguished by 
$\#B_{j:4,3}$ and $\#B_{j':4,3}$.

If $\#H=4$, then $G$ has one copy of $F$ or $\hF$ and one copy of $H$ that is
a maximal triton.  The branch degrees in order are $(4,3,3,3)$ with distances
$(j,j',j)$ or $(j',j,j')$, and whether $G$ has $F$ or $\hF$ is distinguished
by $\#B_{j:4,3}$.

There are two ways that $\#H=5$ can arise.  If the branch degrees in order
are $(3,4,3,3)$ with distances $(j,j',j)$, then $\#F'=1$ and $\#\hF=1$; if
the distances are $(j',j,j')$, then $\#F'=1$ and $\#F=1$.  These two
possibilities are distinguished by which of $B_{j:3,3}$ and $B_{j':3,3}$ occurs
as a maximal baton (not using $w_1$).

When $j=j'$, the possibilities with $\#H\in\{3,4,5\}$ each reduce to one case.

It is also possible that $\#H=6$: if the branch degrees in order
are $(4,3,3,3)$ with distances $(j,j'-j,j)$, then $\#F=1$ and $\#\hF=1$,
each contributing three copies of $H$.

These cases tell us the shorter maximal tritons that do not fit in cards
(actually, only those with length $(r-4)/2$ fail to fit in cards).
All shorter tritons whose key degrees are not $(4,3,3)$ fit in cards,
so by an exclusion argument we know them.

\smallskip
{\bf Subcase 1b:} {\it Vertex degrees in $R'$ are $(4,3,3,3,3)$ or $(5,3,3,3)$.}
Define $F,\hF,F'$ as above.  Here we have $n-r-2=\sum_{i=1}^r(d_i-2)= 6$, so
$$
|V(F)|=|V(\hF)|=|V(F')|\le \FR{r-4}2+7=\FR{n+2}2-\FR{n-r-8}2\le\FR{n+2}2.
$$
Hence each of these tritons fits in a card, and when any one of them appears
it either is already a maximal triton (case $(4,3,3,3,3)$) or groups with a
fixed number of other copies into a fixed maximal triton (case $(5,3,3,3)$).
Hence we obtain the shorter maximal tritons whose key vertices are $w_1$ and
two vertices of degree $3$.  All other tritons are smaller and fit in cards,
so by an exclusion argument we know them.

\medskip
{\bf Case 2:} {\it $d_2\geq 4$ and $d_3=3$.}
Since $d_2\le(n-r-1)/2$, for $1\leq j\leq (r-1)/2$ we have
\begin{equation}\label{case2bat}
|V(B_{j:d_2,2})|<|V(B_{j:d_2,3})|
\le j+d_2+3-1\leq \frac{r-1}{2}+\frac{n-r-1}{2}+2=\frac{n+2}{2}.
\end{equation}
Thus these batons fit in cards.  The key vertex of degree $d_2$ in a copy of
$B_{j:d_2,3}$ in $\cD$ can only be $w_1$ or $w_2$.  The vertex $w_2$ lies in at
most two copies of $B_{j:d_2,3}$ whose other key vertex has degree $3$ in $G$.
Since $d_1>d_2\ge4$, the vertex $w_1$ belongs to either $0$ or at least four
such copies of $B_{j:d_2,3}$.  If $w_1$ and $w_2$ are at distance $j$ in $G$,
then $\#B_{j:d_2,3}\geq (d_1-1){d_2-1\choose 2}+{d_1-1\choose 2}$.  Thus from
$\#B_{j:d_2,3}$ we find the number of copies of $B_{j:d_1,d_2}$ and
$B_{j:d_1,3}$ that are maximal batons.  Using exclusion we also find
$\#B_{j:d_1,2}$ by looking at $\#B_{j:d_2+1,2}$, since $B_{j:d_2+1,2}$ has the
same size as $B_{j:d_2,3}$ and fits in a card.

Now that we know all the short maximal batons, we also know the level pairs,
by Lemma~\ref{levelpair}.  We consider two subcases.

\smallskip
{\bf Subcase 2a:} {\it $\SE i1r(d_i-2)\le2d_2-1$.}
Since $d_1>d_2$, the first two terms of the sum already contribute at
least $2d_2-3$, and $d_3\ge d_4\ge 3$ contributes at least $2$, so in this
subcase the degrees of the branch vertices must be $(d_2+1,d_2,3,3)$.

Knowing the level pairs, we know the levels of the four branch vertices.
Using symmetry, we may assume that $w_3=v_b$ and $w_4=v_{b'}$ with
$b\le\min\{b',\bb'\}$ (similar to the use of $b_1$ and $b_2$ in
Section~\ref{2max}).  We know $b$ but don't know whether $v_{b'}$ is also
in the first half of $P$.

If $b'=\bb$, then knowing whether $w_1$ and $w_{2}$ are in the same half of $P$
determines the caterpillar.  Since we know the short maximal batons, we can
test this by the presence or absence of $B_{h:d_1,d_2}$, where $h$ is the
difference between the levels of $w_1$ and $w_2$ (if $h=0$, then we also know
the caterpillar).  Hence we may assume $b'<\bb$.

For $i\in\{1,2\}$, let $j_i$ be a difference between the level of $w_i$ and the
level of $w_3$ or $w_4$ (there are two choices for each $i$).
If $v_b$ and $v_{b'}$ are both in the first half of $P$, then checking for
batons of the form $B_{j_1:d_1,3}$ as short maximal batons tells us whether
$w_1$ is also in that half.  We can then similarly place $w_2$, completing the
reconstruction.

Thus we may assume $\bb>{b'}>(r+1)/2$.  Let $k=(\bb'+b)/2$; note that $k$ is
the level halfway between the levels of $w_3$ and $w_4$, if they have the same
parity.  If $w_1$ or $w_2$ (say $w_i$) is not in $\{v_k,v_\kb,v_{(r+1)/2}\}$,
then again we can use the existence or absence of $B_{j_i:d_i,3}$ as a short
maximal baton to locate it.  We can then locate the remaining branch vertex by
using the same trick with $B_{j_{3-i}:d_{3-i},3}$ or by using the presence or
absence of $B_{j':d_1,d_2}$, where $j'$ is the difference between the levels of
$w_1$ and $w_2$.

The remaining possibility is $w_1,w_2\in\{v_k,v_\kb,v_{(r+1)/2}\}$.
If $b'-k\le(r-1)/2$, then checking the existence of $B_{b'-k:d_1,3}$ and
$B_{b'-k:d_2,3}$ as maximal batons in $G$ tells us whether $w_1$ or $w_2$ is
$v_k$ (or $v_\kb$), and then we can also place the other branch vertex (since
we know the levels).

Hence we may assume $b'-k\ge r/2$.  Since $b\leq \min\{b',\bb'\}$, this yields
$k\le {r}/{4}$.  Now let $B'=B_{k-b,b:d_2+1,3,2}$.  We compute
$$
|V(B')|\le\FR{r}{4}+d_2+3\leq \FR{r}{4}+\FR{n-r-1}{2}+3
=\FR{n+2}{2}-\FR{r-6}{4}.
$$
Thus $B'$ fits in a card when $r\geq 6$, which holds when ${n\ge17}$,
since $r\ge(n-5)/2$.  Since $B'$ has length $k$, it cannot appear with
key vertices $v_k$ and $v_b$.  Hence $w_1=v_\kb$ if $B'$ appears, and
$w_1=v_k$ if it does not.  This completes the reconstruction.

\smallskip
{\bf Subcase 2b:} {\it $\SE i1r(d_i-2)\ge2d_2$.}
In this subcase we find all the shorter maximal tritons and apply
Lemma~\ref{3-11}.  By~\eqref{d'}, our hypothesis yields $d_2\le(n-r-2)/2$.
Also, since $d_1>d_2\ge4$ and $d_3=3$, by~\eqref{last} we have $d_6=2$, meaning
that $G$ has at most three vertices of degree $3$.  Also, if $G$ has three
vertices of degree $3$, then $d_1=d_2+1$.

Consider $j,j'$-tritons, where $j+j'\le(r-4)/2$.
Let $H=B_{j,j':d_2+1,3,3}$ and $H'=B_{j,j':3,d_2+1,3}$.  Each of the shorter
tritons $H$ and $H'$ has at most $N$ vertices, where
\begin{equation}\label{shortermax}
N\le\frac{r-4}{2}+d_2+4=\FR{r-4}2+\FR{n-r-2}2+4=\frac{n+2}{2} .
\end{equation}
Hence $H$ and $H'$ fit in cards.  In any copy of $H$ or $H'$ in $G$, the
vertex of degree $d_2+1$ must be $w_1$.  We use $\#H$ and $\#H'$ to determine
the shorter maximal $j,j'$-tritons that may have more than $(n+2)/2$ vertices;
their key vertices have degrees $\{d_1,3,3\}$ or $\{d_1,d_2,3\}$.
We determine the larger ones first.

Let $h$ be the distance in $G$ between $w_1$ and $w_2$, and let $h'$ be the
difference between their levels.  We have $h=h'$ if and only if $G$ has
$B_{h':d_1,d_2}$ as a maximal baton.  Otherwise, $h=h'+r+1-2k$, where $k$ is
the higher of the levels of $w_1$ and $w_2$.  In particular, we know $h$.
\looseness -1

If $h\notin\{j,j',j+j'\}$, then $G$ contains no $j,j'$-triton having key
vertices with degrees $d_1$ and $d_2$.  If $h=j+j'$ and
$B_{j,j':d_1,3,d_2}\not\esub G$, then since $G$ has at most three vertices of
degree $3$, we have $\#H\leq 2$, while $B_{j,j':d_1,3,d_2}\esub G$ yields
$\#H\geq {d_2-1\choose 2}\geq 3$.  Thus $\#H$ determines whether
$B_{j,j':d_1,3,d_2}$ is present.

Now suppose $h=j$.  If $B_{j,j':d_2,d_1,3}\not\esub G$, then $H'$ can only
occur with key vertices having degrees $3,d_1,3$ in $G$, yielding
$\#H'\leq \CH{d_1-2}{d_2-1}$.  On the other hand, $B_{j,j':d_2,d_1,3}\esub G$
yields $\#H'\geq \CH{d_2-1}2\CH{d_1-2}{d_2-1}$.  Since $d_2\ge4$, we conclude
that $\#H'$ determines whether $B_{j,j':d_2,d_1,3}$ is present.

If $B_{j,j':d_1,d_2,3}\esub G$, then $\#H\ge\CH{d_1-1}{d_2}(d_2-2)$.
To prove that $\#H$ determines whether $B_{j,j':d_1,d_2,3}$ is present, we show
that $\#H$ is smaller when $B_{j,j':d_1,d_2,3}\not\esub G$.  The vertex of
degree $d_2+1$ in a copy of $H$ must be $w_1$.  Since $w_2$ has distance $j$
from $w_1$, using $w_2$ in a copy of $H$ would produce $B_{j,j':d_1,d_2,3}$,
which is forbidden.  Hence we can only produce a copy of $H$ moving in one
direction from $w_1$, distance $j$ and then $j'$ to reach two key vertices
having degree $3$ in $G$.  The result is $\#H\le \CH{d_1-1}{d_2}$, which is
less than the earlier lower bound since $d_2\ge4$.

The cases with $h=j'$ are symmetric to those with $h=j$.
Thus we have determined all shorter maximal tritons
whose key vertices have degrees $\{d_1,d_2,3\}$.

Knowing $h$ and the levels, we actually know the positions of $w_1$ and
$w_2$, where we may assume $w_1=v_b$ with $b\le(r+1)/2$.  Since $d_1\ne d_2$,
there is a $j,j$-triton with initial key vertices of degrees $d_1$ and then
$d_2$ if and only if the distances allow it to fit along the spine.
The third key vertex will have degree $2$ if and only if it does not have
degree $3$.  Since we know the shorter maximal tritons whose key vertices have
degrees $\{d_1,d_2,3\}$, we know this answer.  The same argument applies to
determine all shorter maximal tritons
whose key vertices have degrees $\{d_1,d_2,2\}$.

The remaining tritons that don't fit in cards have key vertices of degrees
$\{d_1,3,3\}$.  Furthermore, by~\eqref{shortermax}, these fail to fit in 
cards only if the degrees of the branch vertices are $\{d_2+2,d_2,3,3\}$.
Since in this case there are only two branch vertices of degree $3$, there
can only be one copy of $B_{j,j':d_1,3,3}$ or $B_{j,j':3,d_1,3}$ as a maximal
triton, and there cannot also be a $j,j'$-triton with key vertices of degrees
$\{d_1,d_2,3\}$.  After checking that we have no $j,j'$-triton with degrees
$\{d_1,d_2,3\}$, we have $B_{j,j':d_1,3,3}$ or $B_{j,j':3,d_1,3}$ in $G$ if and
only if $\#H=\CH{d_1-1}{d_2}=d_2+1$ or $\#H'=\CH{d_1-2}{d_2-1}=d_2$.

All remaining shorter tritons fit in cards.  Hence the exclusion argument
allows us to determine the remaining shorter maximal tritons.

\medskip
{\bf Case 3:} {\it $d_2>d_3\geq 4$.}
By~\eqref{last} and $d_4\ge3$, there are four branch vertices, with degrees
$(d_2+1,d_2,4,3)$, all distinct.  Using~\eqref{d'}, $d_2=(n-r-2)/2$.
By the same computation as~\eqref{case2bat} in Case 2 (improved by $1/2$), the
batons $B_{j:d_2,2}$ and $B_{j:d_2,3}$ fit in cards when $j\leq (r-1)/2$.

Since $d_1=d_2+1$, the number of copies of $B_{j:d_2,3}$ whose key vertices are
branch vertices at distance $j$ is $d_2{d_2-1\choose 2}+{d_2\choose 2}$ if the
branch vertices are $\{w_1,w_2\}$; this count simplifies to $\CH{d_2}2(d_2-1)$,
which is at least $40$ since $d_2\ge5$.  The count is $3d_2$ for $\{w_1,w_3\}$,
is $d_2$ for $\{w_1,w_4\}$, is $3$ for $\{w_2,w_3\}$, and is $1$ for
$\{w_2,w_4\}$.  These five numbers are distinct, and indeed no two subsets
of them have the same sum.  Hence we know which short maximal batons with
key vertices of degree at least $3$ are present.

All other short batons fit in cards (note that $B_{j:d_1,2}=B_{j:d_2+1,2}$,
and $B_{j:d_2+1,2}$ has the same size as $B_{j:d_2,3}$, which fits in a card).
Hence by the exclusion argument we know all short maximal batons.

To complete the reconstruction, since $w_1,\ldots,w_4$ have distinct degrees
and we know their levels, we simply use short maximal batons to test for each
of $w_2,w_3,w_4$ whether it is in the same half of the spine as $w_1$.

\medskip
{\bf Case 4:} {\it $d_2=d_3\geq 4$.}
By~\eqref{last} and $d_4\ge3$, the degrees of the branch vertices are
$(5,4,4,3)$.  Thus $n-r-2=\SE i1r(d_i-2)=8$.
For $j\le(r-1)/2$, we have
$$
|V(B_{j:4,2})|<|V(B_{j:4,3})|=j+6\le \FR{r-1}2+6
=\FR{n+2}2-\FR{n-r-9}2=\FR{n+1}2.
$$
Thus $B_{j:4,2}$ and $B_{j:4,3}$ fit in cards.

To find the short maximal batons that don't fit in cards, consider
$\#B_{j:4,3}$.  It is $4{3\choose 2}+{4\choose 2}$ (that is, $18$) for batons
with key vertices $\{w_1,w_2\}$ or $\{w_1,w_3\}$, is $4$ for $\{w_1,w_4\}$,
is $3+3$ for $\{w_2,w_3\}$, and is $1$ for $\{w_2,w_4\}$ or $\{w_3,w_4\}$.
Distinct subsets of these larger batons give distinct values of $\#B_{j:4,3}$,
except for possibly $4+1+1=6$.  However, the values $4,1,1$ cannot all occur,
because they require $w_4$ to be at distance $j$ from all of $\{w_1,w_2,w_3\}$.
Hence the numbers $\#B_{j:4,3}$ determine which batons occur with both key
vertices of degree at least $3$.

All other short batons fit in cards ($B_{j:5,2}$ has the same size as
$B_{j:4,3}$).  Hence we know the short maximal batons
(using exclusion) and also the level pairs.
We now locate the branch vertices as in Subcase 2a, with the roles of $w_2$ and
$w_4$ exchanged.  That is, we first place $w_3$ and $w_2$, with $w_3=v_b$ and
$w_2=v_{b'}$, where $b\le\min\{b',\bb'\}$.  Using batons and tritons as in
Subcase 2a, we then place $w_1$ and $w_4$ relative to $w_3$ and $w_2$.
\looseness-1
\end{proof}
 
Lemmas~\ref{cond} and~\ref{uncond} complete the proof of the theorem
for all high-diameter cases with a unique vertex of maximum degree
and at least four branch vertices.

\section{At Most Three Branch Vertices}\label{sec:3branch}

Here we reconstruct caterpillars with few branch vertices.  Trees with no
branch vertices are paths, which are reconstructible from the degree list.
The next case is also fairly easy.

\begin{lemma}\label{1branch}
When $n\ge2\ell+1$, caterpillars with one branch vertex are
$\ell$-reconstructible.
\end{lemma}
\begin{proof}
We know the degree list, so we recognize this case, with a branch vertex
having degree $d_1$ and $r=n-d_1$.  Since $r\ge(n-5)/2$ in the high-diameter
case, we may assume $d_1\le (n+5)/2$.  We seek $i$ such that $d(v_i)=d_1$.
We may assume $i\le(r+1)/2$.  We have $i=1$ if and only if
the deck does not contain the $6$-vertex caterpillar obtained by adding a
pendant edge at the central vertex of the path $P_5$.

Hence we may assume $1<i\le(r+1)/2$.  We count copies in the deck of the path
$P_j$ with $j$ vertices.  There are $r-j+1$ copies of $P_j$ contained in the
spine $\la\VEC v1r\ra$.  Excluding those, there remain $d_1-2$ copies of $P_j$
if $j>i+2$ (through indices greater than $i$), and there are $2(d_1-2)$ copies
of $P_j$ if $j\le i+2$.  Thus, as $j$ increases from $1$, the value of $i$ is
the value of $j-2$ when number of copies of $P_j$ not along the spine decreases.

We must check that $P_j$ fits in a card for the needed values of $j$.
If we can see $P_j$ when $j\le(r+5)/2$,then we can detect the branch vertex
$v_i$ for $i\le(r-1)/2$ in this way.  Indeed we have $(r+5)/2\le (n+2)/2$,
because $r\le n-3$, which holds since $d_1\ge3$.

If we do not find the branch vertex in this way, then $i=\CL{r/2}$.
\end{proof}

\begin{lemma}
When $n\ge2\ell+1$, caterpillars with two branch vertices are
$\ell$-reconstructible.
\end{lemma}
\begin{proof}
Let $v_{b}$ and $v_{b'}$ be the two branch vertices along $\VEC v1r$.
By symmetry, we may assume $b\le \min\{b',\bb'\}$.
Let $h=b'-b$.  If $j\le(n-8)/2$, then $\C{V(B_{j:3,3})}\le(n+2)/2$,
and we see whether $G$ contains $B_{j:3,3}$.  If it does, then $h=j$.
Hence we begin by determining $h$ if $h\le(n-8)/2$, and otherwise we know
$h\ge(n-7)/2$.  Note also that $b-1\le(r-1-h)/2$; this means that when we
don't know $h$ we know that it is greater than $b$.

The spine is the path $\la\VEC v1r\ra$ of nonleaf vertices.  Specifying
$v_0$ and $v_{r+1}$ as arbitary leaf neighbors gives us the {\it augmented
spine} $\la\VEC v0{r+1}\ra$.  With this definition, any branch vertex with
degree $d$ has $d-2$ leaf neighbors outside the augmented spine.
Let $Q_j$ be the path with $j+3$ vertices; note that $Q_j=B_{j:2,2}$ when
$j\ge1$.  There are always $r-j$ copies of $Q_j$ contained in the augmented
spine.  Other copies of $Q_j$ start at leaf neighbors of $v_b$ or $v_{b'}$ and
continue in either direction along the spine.  The maximum value of $\#Q_j$ is
achieved when the augmented spine is long enough in both directions from both
branch vertices to fit $Q_j$.  For $1<j\ne h$, that value is $r-j+2d_1+2d_2-8$.
When $j=h$, there is another contribution where both endpoints are leaves
outside the augmented spine, and the value is $r-j+2d_1+2d_2-8+(d_1-2)(d_2-2)$.
In both cases, let $g_j$ denote this greatest value.

We claim that $\#Q_j=g_j$ when $j<b$.  In particular, we know whether $j$
equals $h$ when $j\le b$, and in that case the contribution of
$(d_1-2)(d_2-2)$ appears in both $\#Q_j$ and $g_j$.  However, $\#Q_b<g_b$,
since the distance from $v_b$ to $v_1$ is less than $b$.  Thus we obtain $b$ as
the least value of $j$ such that $\#Q_j<g_j$.  In addition, $g_b-\#Q_b$ tells
us whether the degree of $v_b$ is $d_1$ or $d_2$; if the difference is $d_i-2$,
then the degree is $d_i$.  There is also the possibility that
$g_b-\#Q_b=d_1+d_2-4$; this happens if and only if $\bb=b'$, which determines
$G$.

Having determined $v_b$ and its degree, we again consider $h$.
Recall that we have determined $h$ if it is at most $(n-8)/2$.
In that case, we set $b'=b+h$ and know all of $G$.

In the remaining case, we have $h\ge(n-7)/2$.  If $b+(n-7)/2\ge(r+1)/2$, then
we need only find the level of $b'$ to complete the reconstruction of $G$.
Let $a=g_b-\#Q_b$.  The deficiency of $a$ continues as we increase $j$
beyond $b$.  Let $k$ be the least value of $j$ greater than $b$ such that
$\#Q_j<g_j-a$; note that $k<b+h$.  If $b+(n-7)/2\ge(r+1)/2$, then we set
$b'=\kb$ and have determined $G$.  In order to be guaranteed seeing such $Q_j$,
it suffices for $Q_j$ to fit in a card when $j\le(r-1)/2$, since if we don't
find $k$ by then $k$ is the only remaining level, $r/2$ or $(r+1)/2$.  Thus
it suffices to have $(r+5)/2\le(n+2)/2$, which simplifies to $n-r\ge3$, which
holds when $G$ has two branch vertices.

When $b+(n-7)/2<(r+1)/2$, we still need to find $b'\in\{k,\kb\}$, but very few
possibilities remain.  Since $b\ge1$, the condition requires $r\ge n-5$.
By~\eqref{d'}, $n-r=d_1+d_2-2$.  With two branch vertices, $n-r\ge4$, so
$r\le n-4$.

Let $n-r=4+\eps$, where $\eps\in\{0,1\}$.  We have $(d_1,d_2)=(3+\eps,3)$.
We may assume $(b,k)=(1,(r-1+\eps)/2)$, since $(n-7)/2=(r-3+\eps)/2$ and
$b+h\ge(r-1+\eps)/2$.  We count copies of $Q_j$, where $j={(r-3+\eps)/2}$,
which we have observed fit in a card.  Let $D=\#Q_j-(r-j)$ in order to
ignore the $r-j$ copies contained in the augmented spine in all cases.
When $d_1=3$ we have $D=5$ if $b'=v_k$ and $D=4$ if $b'=v_\kb$.
When $d_1=4$ we have $D\in\{7,8\}$ if $b'=v_k$ and $D\in\{5,6\}$ if $b'=v_\kb$,
depending on whether $v_1$ has degree $4$ or $3$.  Hence we determine $b'$.
\end{proof}


\nobreak
The case when $G$ has exactly three branch vertices takes quite a bit more work.

\begin{lemma}\label{rec3levels}
For $n\ge19$, let $G$ be an $n$-vertex caterpillar having exactly
three branch vertices.  If the levels of the branch vertices are known
and the sum of the degrees of the branch vertices is $10$,
then $G$ is reconstructible from the deck.  The same holds when those degrees
sum to $9$ if $r$ is odd or $\#B_{(r-2)/2:3,3}$ is known.
\end{lemma}
\begin{proof}
Let $R'$ the set of branch vertices.
By~\eqref{d'}, $n-r= 2+\SE i13 (d_i-2)\ge5$.  Furthermore, $n-r\ge6$ unless
the degrees in $R'$ all equal $3$; call that the ``$3$-case''.
Since the inequalities $r/2+4\le (n+2)/2$ and $(r-2)/2+5\le (n+2)/2$
reduce to $n-r\ge 6$, except in the $3$-case we have
\begin{equation}\label{j32fit}
\mbox{\em $B_{j:3,2}$ fits in a card when $j\le r/2$,}
\end{equation}
and
\begin{equation}\label{j33fit}
\mbox{\em $B_{j:3,3}$ fits in a card when $j\le(r-2)/2$.}
\end{equation}
In the $3$-case the requirements are $j\le(r-1)/2$ for $B_{j:3,2}$ and
$j\le(r-3)/2$ for $B_{j:3,3}$, but in that case we are given
$\#B_{(r-2)/2:3,3}$, so we know that quantity in all cases.

Let $s=(r+1)/2$.  Let the {\em depth} $f(v)$ of a vertex $v$ in $P$ be the
distance of $v$ from the middle of $P$.  That is, the level $k$ of $v_i$
is $\min\{i,\ib\}$, and $f(v_i)=s-k$; note that $f(v_i)$ is a half-integer when
$r$ is even.  By hypothesis, we are given the depths of all three branch
vertices.
Let an {\it $R'$-baton} be a maximal baton whose key vertices are both branch
vertices.

An easy case is $v_s\in R'$.  Let $b$ be the difference between the levels
of the other two branch vertices.  Since we know all the depths, we know the
$R'$-batons having $v_s$ as a key vertex.  Excluding those, we check whether
$B_{b:3,3}$ occurs to determine whether the other two branch vertices are in
the same half of $P$, completing the reconstruction.  Since the other branch
vertices are not $v_s$, we have $b\le(r-3)/2$, so $B_{b:3,3}$ fits in a
card, by~\eqref{j33fit}.

Another easy case is when two branch vertices have the same depth, with
degrees $d$ and $d'$, and the third branch vertex (with degree $d^*$) has
some positive depth.  When $d=d'$ we already know $G$, but when they differ
we must determine which half of the spine contains the third vertex.
Letting $b$ be the difference between the two depths, we simply check whether
$\#B_{b:3,3}$ is $(d-1)(d^*-1)$ or $(d'-1)(d^*-1)$.  We can do this when
$b\le(r-3)/2$ by the computation above.  It is also possible that
$b=(r-2)/2$ and $\C{V(B_{b:3,3})}=(n+3)/2$, but this occurs only in the
$3$-case, and then $d=d'$.

Hence we may assume that the vertices in $R'$ have distinct depths, all
positive.  Analogously to Definition~\ref{W3def}, rename the vertices of $R'$
as $u_1,u_2,u_3$ in order of depth, with
$$f(u_1)=a<f(u_2)=a+b<f(u_3)=a+b+c$$
(see Figure~\ref{3maxfig}).  Let $h=s-a-b-c$, so $u_3\in \{v_h,v_\hb\}$.  

As in Lemma~\ref{W3}, say that $R'$ is {\it Type $0$} if all of $R'$
is in the same half of $P$ and {\it Type $i$} for $i\in\{1,2,3\}$ if $u_i$ is
in the half opposite $R'-\{u_i\}$.  There must be at least two vertices
in the same half of the spine.  Since $a\ge1/2$, any $R'$-baton with key
vertices in the same half of the spine has length at most $(r-2)/2$.
By~\eqref{j33fit}, we can discover the lengths of such $R'$-batons by
checking for $B_{j:3,3}$, or in the $3$-case we are given whether
$B_{(r-2)/2:3,3}$ occurs.  These lengths lie in $\{b,c,b+c\}$.

We consider subcases by the number of $R'$-batons with length at most $(r-2)/2$.
For $j\le(r-2)/2$, we know $\#B_{j:3,3}$ by~\eqref{j33fit} or the information
given in the $3$-case.  In the $3$-case, all copies of $B_{j:3,3}$ are 
$R'$-batons.  In the case where the branch vertices have degrees $(4,3,3)$,
we see $B_{j:4,3}$ when it occurs if $j\le(r-4)/2$,
and $B_{(r-2)/2:4,3}$ occurs if and only if $\#B_{(r-2)/2:3,3}\ge3$
(and there are two copies if $\#B_{(r-2)/2:3,3}=6$).

\smallskip
{\bf Case 1:} {\it Only one $R'$-baton, say $B$.}   
Since $R'$ has two vertices in the same half of $P$, the length of $B$ is
$b$, $c$, or $b+c$, and $R'$ is respectively Type $3$, $1$, or $2$
if those lengths are distinct.  If $b=c$ and $B$ has length $b$, then we still
must distinguish Type $3$ and Type $1$.  In either case, $u_1$ and $u_3$ are in
different halves of $P$; we may assume $u_1=v_{s+a}$ and $u_3=v_h$.
Now Type $3$ means $u_2=v_{s+a+b}$ and Type $1$ means $u_2=v_{s-a-b}$.

If $b+2a\leq (r-2)/2$, then since we do not also have an $R'$-baton
of length $b+2a$, the arrangement is Type $3$.  (Note: $c=2a$ is possible,
but the two vertices of $R'$ in the same half already give an $R'$-baton
with length $b+c$, and another with key vertices $u_2$ and $u_1$ of length
$b+2a$ would count as a second $R'$-baton.)

Hence we may assume $b+2a\ge (r-1)/2$.  Now we distinguish Type $1$ from Type
$3$ by counting copies of $B_{c,h:3,3,2}$.  When $u_2=v_{s+a+b}$, the triton
can extend from $v_{s+a}$ through increasing indices or from $v_{s+a+b}$
through decreasing indices.  When $u_2=v_{s-a-b}$, the triton occurs only by
extending from $v_h$ through increasing indices, since the path through 
decreasing indices is not long enough.  Thus in the $3$-case there are two
copies when $u_2=v_{s+a+b}$ and only one when $u_2=v_{s-a-b}$.  In the other
case we are given the level of the vertex of degree $4$.  We have the following
counts to distinguish Type $1$ from Type $3$.
\begin{center}
\begin{tabular}{c c c c }
 & $d(u_1)=4$& $d(u_2)=4$& $d(u_3)=4$\\
$\#B_{c,h:3,3,2}$ when $u_2=v_{s+a+b}$& 5&5&2\\
$\#B_{c,h:3,3,2}$ when $u_2=v_{s-a-b}$& 1&2&3
\end{tabular}
\end{center}

Thus it suffices to show that $B_{c,h:3,3,2}$ fits in a card.  From
$b+2a\ge(r-1)/2$ we obtain
$$2(h-1)+2c+b=r-1-b-2a\leq (r-1)/2$$
and so $h+c\leq (r+1)/4$ (since $b\ge1$).  Using $n-r\ge5$, we have
$$
\C{V(B_{c,h:3,3,2})}=c+h+8-3\le\frac{r+1}{4}+5=\frac{r+11}{4}+\FR52\le
\FR{r+11}4+\frac{n-r}{2}=\frac{n+2}{2}+\FR74-\frac{r}{4}.
$$
With $r\ge7$, which happens when $n\ge19$ (since $r\ge(n-5)/2$ in the
high-diameter case), $B_{c,h:3,3,2}$ fits in a card.

\smallskip
{\bf Case 2:} {\it Two $R'$-batons.}
Having only two $R'$-batons means that the two extreme vertices of $R'$ are
too far apart to see their $R'$-baton in a card.  Hence $R'$ is not Type $0$.
If $R'$ is Type $1$, then the lengths of the $R'$-batons we see are
$\{2a+b,c\}$.  If Type $2$, then they are $\{2a+b,b+c\}$.  If Type $3$, then
they are $\{2a+b+c,b\}$.  Since the pairs are distinct, we know $G$.
\looseness-1

\smallskip
{\bf Case 3:} {\it Three $R'$-batons.}  If $R'$ is Type $0$, then the triple of 
lengths of the batons is $\{b,c,b+c\}$.  If Type $1$, then it is
$\{c,2a+b,2a+b+c\}$.  If Type $2$, then it is $\{b+c,2a+b,2a+2b+c\}$.
If Type $3$, then it is $\{b,2a+b+c,2a+2b+c\}$.  Since these triples are
distinct, we know $G$. 
\end{proof}

We now consider cases by the sum of the degrees of the three branch vertices.
When the sum is at most $10$, we apply Lemma~\ref{rec3levels} and hence
use the threshold for $n$ assumed there.

\begin{lemma}\label{3bsmallsum}
For $n\ge19$, if the unknown caterpillar $G$ has exactly three branch vertices
and their degrees sum to $9$, then $G$ is reconstructible from the deck.
\end{lemma}
\begin{proof}
Let $R'$ denote the set of branch vertices in $G$; we have
$R'=\{w_1,w_2,w_3\}$, with $d(w_i)=3$.  By~\eqref{d'}, $\SE i13 d_i=9$
implies $n-r=5$, so $r=n-5$.  Hence 

\begin{equation}\label{batfit32}
\mbox{\em $B_{j:3,2}$ fits in a card when $j\le(r-1)/2$,}
\end{equation}
and
\begin{equation}\label{batfit3}
\mbox{\em $B_{j:3,3}$ fits in a card when $j\le(r-3)/2$.}
\end{equation}
Let $s=(r+1)/2$.
Let the {\em depth} $f(v)$ of a vertex $v$ in $P$ be the distance of $v$ from
$v_s$ (depth is $(r+1)/2$ minus level and is a half-integer when $r$ is even).
By Lemma~\ref{rec3levels}, it suffices to determine the levels of the
branch vertices and, if $r$ is even, $\#B_{(r-2)/2:3,3}$.

Let $t=\FL{(r-1)/2}$.  The only short baton not fitting in a card is
$B_{t:3,3}$.  Every such baton contains four copies of $B_{t:3,2}$, which
we see in cards.  We obtain $\#B_{t:3,3}$ from $\#B_{t:3,2}$ as follows:
\begin{center}
\begin{tabular}{c c c c c c c}
$\#B_{t:3,2}$&8&7&6&5&4&3\\
$\#B_{t:3,3}$&2&1&1&1$\vert$0&0&0
\end{tabular}
\end{center}
A branch vertex with depth at most $1/2$ can belong to copies of $B_{t:3,2}$
extending in both directions along $P$; branch vertices in other positions
extend in one direction.  Thus when $B_{t:3,3}$ appears, the third branch
vertex contributes at least one more copy of $B_{t:3,2}$. 
Hence when $\#B_{t:3,3}=1$ we have $\#B_{t:3,2}\ge5$.  We have $\#B_{t:3,3}=2$
if and only if copies of $B_{t:3,3}$ extend in both directions from a branch
vertex with depth at most $1/2$, and then $\#B_{t:3,2}=8$.

The only possible confusion is when we see $\#B_{t:3,2}=5$; we may also have
$\#B_{t:3,2}=0$.  This requires two branch vertices that extend to copies
of $B_{t:3,2}$ in both directions, which occurs only for branch vertices with
depth at most $1/2$.  When $r$ is odd the only such vertex is $v_{(r+1)/2}$.
Hence $\#B_{t:3,3}=5$ with $\#B_{t:3,3}=0$ occurs only when $r$ is even and
$v_{r/2}$ and $v_{r/2+1}$ are both branch vertices; call this Type $0$.
Meanwhile, if $\#B_{t:3,3}=1$ and there is a branch vertex in
$\{v_{r/2},v_{r/2+1}\}$, then $\#B_{t:3,2}\ge6$.  Therefore, $\#B_{t:3,3}=5$
with $\#B_{t:3,3}=1$ when $r$ is even can only occur when
$\{v_{r/2},v_{r/2+1}\}$ contains no branch vertex; call this Type $1$.
We distinguish between Type $0$ and Type $1$ by computing $\#B_{t:2,2}$.
Note $\C{V(B_{t:2,2})}=t+3$.  Designating leaf neighbors $v_0$ and $v_{r+1}$ of
$v_1$ and $v_r$ to augment the spine, both types have $r-t$ copies of
$B_{t:2,2}$ contained in the augmented spine.  A Type $0$ caterpillar has
five additional copies of $B_{t:2,2}$, while a Type $1$ caterpillar has only
four additional copies.

Hence we know all short maximal batons.  By Lemma~\ref{levelpair}, this yields
the level pairs $\{d(v_k),d(v_\kb)\}$ for $1\leq k\leq (r-1)/2$.  The remaining
degrees belong to the remaining level, $\{d(v_s)\}$ or
$\{d(v_{r/2}),d(v_{r/2+1})\}$.  As noted earlier, Lemma~\ref{rec3levels}
applies to complete the reconstruction.
\end{proof}

\begin{lemma}\label{3br10}
For $n\ge19$, if the unknown caterpillar $G$ has exactly three branch vertices
and their degrees sum to $10$, then $G$ is reconstructible from the deck.
\end{lemma}
\begin{proof}
The set $R'$ of branch vertices is $\{w_1,w_2,w_3\}$, with degrees $4,3,3$
in order.  By~\eqref{d'}, $\SE i13 d_i= 10$
implies $n-r=6$; in particular, $r=n-6$.  Hence 
\begin{equation}\label{4batfit32}
\mbox{\em $B_{j:3,2}$ fits in a card when $j\le r/2$,}
\end{equation}
and
\begin{equation}\label{4batfit3}
\mbox{\em $B_{j:3,3}$ fits in a card when $j\le(r-2)/2$.}
\end{equation}
Since $d_1=4$, there are one (or two) copies of $B_{j:4,3}$ if and only if
$G$ has at least three (or six) copies of $B_{j:3,3}$, respectively, and the
number of copies of $B_{j:3,3}$ that are maximal batons is at most $2$.  Thus
$\#B_{j:3,3}$ determines $\#B_{j:4,3}$.  By Lemma~\ref{rec3levels}, it suffices
to determine the levels of the branch vertices.

Using~\eqref{4batfit3} and exclusion, we obtain all maximal batons with length
at most $(r-2)/2$.  By Lemma~\ref{levelpair}, this yields the level pairs
$\{d(v_k),d(v_\kb)\}$ for $1\leq k\leq (r-2)/2$.

When $r$ is even, the only remaining level is $\{v_{r/2},v_{(r+2)/2}\}$, and
its vertices have the remaining degrees.  When $r$ is odd, the remaining set is
$\{v_{(r-1)/2},v_{(r+1)/2},v_{(r+3)/2}\}$; call this $I$.  We know
$\C{R'\cap I}$ by knowing the level pairs outside $I$.
If $\C{R'\cap I}=3$, then we learn the position of $w_1$ within
$I$ by whether $\#B_{1:4,3}$ is $1$ or $2$, and this determines $G$.
If $\C{R'\cap I}=0$, then all vertices in $I$ have degree $2$.

Let $s=(r+1)/2$.  When $\C{R'\cap I}\in\{1,2\}$, knowing $d(v_s)$ determines 
the remaining level pair $\{d(v_{s-1}),d(v_{s+1})\}$, since we know the level
pairs outside $I$.  By~\eqref{4batfit32}, $B_{s-1:3,2}$ fits in a card; we use
$\#B_{s-1:3,2}$ to determine $d(v_s)$.  The nonzero contributions to
$\#B_{s-1:3,2}$ from key vertices separated by distance $s-1$ are $9$, $4$,
$3$, or $1$ when the degrees of the key vertices are $(4,3)$, $(3,3)$,
$(4,2)$, or $(3,2)$, respectively.  Vertex $v_s$ may generate such batons
in both directions, but any other vertex can be a key vertex in a copy of
$B_{s-1:3,2}$ extending in only one direction along the spine.  

Let $k$ and $k'$ with $k\le k'$ be the two least levels of branch vertices
(if $\C{R'\cap I}=2$, then $k'=s-1$).  The columns below are distinct cases,
mostly since we know $k$ and $k'$.  The second column occurs when the distance
between two branch vertices not including $v_s$ is $s-1$; in the first column
that does not occur.  Those two cases are distinguished by $\#B_{s-1:3,2}$.
\begin{center}
\begin{tabular}{c c | c c | c c | c c}
$k>1$&usual& $k>1$&extra&$k'>k=1$&& $k'=k=1$\\
$\#B_{s-1:3,2}$&$d(v_s)$& $\#B_{s-1:3,2}$&$d(v_s)$& $\#B_{s-1:3,2}$&$d(v_s)$& $\#B_{s-1:3,2}$&$d(v_s)$\\
8&4& 10&4&     13&4& 18&4\\
6&3& 11&3&11 or 8&3& 13&3\\
5&2&7 or 10&2&10 or 7 or 5&2& 5&2\\
\end{tabular}
\end{center}
The ``or'' choices depend on which branch vertex other than $v_s$ has degree
$4$.  The only ambiguity is in column $2$ with $k>1$ and $\#B_{s-1:3,2}=10$.
If $d(v_s)=4$, then the branch vertices at levels $k$ and $k'$ both have
degree $3$; in the other case their degrees are $\{4,3\}$.  Since we know the
level pairs except in $I$, ambiguity thus requires $k=2$ and $k'=s-1$, with
$d(v_2)=3$ and $v_{s+1}\in R'$ (by symmetry).  Since $\C{R'\cap I}$ is known,
the third branch vertex is $v_s$ or $v_{s-1}$.  Now we distinguish the two
cases by which of $B_{1:4,3}$ and $B_{2:4,3}$ occurs.
\end{proof}

We are now left with the case where the unknown caterpillar has exactly
three branch vertices and their degrees sum to at least $11$.
By~\eqref{d'}, this yields $n-r=2+\SE i13 (d_i-2)\ge 7$, so $r\leq n-7$.
Thus $(r-1)/2\le(n-8)/2$.  Again let $R'$ denote the set of branch vertices.  

\begin{definition}\label{distdef}
Let $u_1,u_2,u_3$ be the indices of the branch vertices (not the vertices
themselves) along the spine, where $1\leq u_1<u_2<u_3\leq r$ and the ordering
$\VEC v1r$ of the spine is chosen so that $u_2-u_1\leq u_3-u_2$.  For
$i,j\in\{1,2,3\}$, let $\delta_{i,j}=|u_i-u_j|$.

Being able to see $R'$-batons in the deck depends to a large extent on their
lengths, which are the distances between the branch vertices.  We distinguish
three ``distance cases'':\\
{\bf Case 1.} $\delta_{1,2},\delta_{2,3},\delta_{1,3}\leq (r-1)/2$.\\
{\bf Case 2.} $\delta_{1,2},\delta_{2,3}\le(r-1)/2$ and $\delta_{1,3}>(r-1)/2$.\\
{\bf Case 3.} $\delta_{1,2}\le(r-1)/2$ and $\delta_{2,3},\delta_{1,3}>(r-1)/2$.
\end{definition}

We first prove that the distance case of the unknown caterpillar is 
recognizable from the deck.  We then use the properties of each case to
determine the exact indices and complete the reconstruction.

\begin{lemma}\label{distcase}
In the situation where the unknown caterpillar $G$ has exactly three branch
vertices and their degree sum to at least $11$, the deck determines the
distance case of $G$, as defined in Definition~\ref{distdef}.
\end{lemma}
\begin{proof}
We will use the counts $\#B_{j:3,3}$ for $1\le j\le(r-1)/2$.  Since
$\C{V(B_{j,3,3})}=j+5\le(r-1)/2+5\le(n+2)/2$, all such batons fit in cards,
and we can count them.  Note that the key vertices in a copy of $B_{j:3,3}$
must be branch vertices.

For simplicity of notation in the computations, let $y_s=d(u_s)-1$ and
$x_s={y_s\choose 2}$ for $s\in\{1,2,3\}$.  For various lengths of the batons,
the exact counts are as follows:
    \begin{align*}
        \#B_{j:3,3}=\begin{cases}
            0 & j\not\in \{\delta_{1,2}, \delta_{2,3}, \delta_{1,3}\}\\
            x_1x_2 & j=\delta_{1,2}\neq \delta_{2,3}\\
            x_2x_3 & j=\delta_{2,3}\neq\delta_{1,2}\\
            x_1x_2+x_2x_3 & j=\delta_{1,2}=\delta_{2,3}\\
            x_1x_3 & j=\delta_{1,3}
        \end{cases}
    \end{align*}

Consider the number of distinct values of $j$ such that $j\le(r-1)/2$ and
$\#B_{j:3,3}\neq 0$.  There is at least one, since there must be two branch
vertices in one half of the spine.  If there are three such values, say
$j_1,j_2,j_3$ with $j_1<j_2<j_3$, then $j_1=\delta_{1,2}$, $j_2=\delta_{2,3}$,
and $j_3=j_1+j_2=\delta_{1,3}$.  Thus Case $1$ holds.

Next, suppose that there are two such values, $j_1$ and $j_2$ with $j_1<j_2$.
If $j_2\neq 2j_1$, then we must have $j_1=\delta_{1,2}$, $j_2=\delta_{2,3}$,
and $\delta_{1,3}>(r-1)/2$.  Thus Case $2$ holds.

If $j_2=2j_1$, then we must look closer.  Two types of caterpillars produce
such counts: we may have $\delta_{1,2}=\delta_{2,3}=j_1$ and $\delta_{1,3}=j_2$
(Case 1) or $\delta_{1,2}=j_1$, $\delta_{2,3}=2j_1$ and $\delta_{1,3}>(r-1)/2$
(Case 2).  We distinguish these possibilities by comparing the counts.
In the Case 1 instance, $\#B_{j_1:3,3}=x_1x_2+x_2x_3$ and
$\#B_{j_2:3,3}=x_1x_3$, while the Case 2 instance yields $\#B_{j_1:3,3}=x_1x_2$
and $\#B_{j_2:3,3}=x_2x_3$.  Although we do not know the order of the degrees
of branch vertices along the spine, we know the set of their degrees, so we
know the value $x_1x_2+x_1x_3+x_2x_3$.  If Case 1 holds here, then this value
equals $\#B_{j_1:3,3}+\#B_{j_2:3,3}$; otherwise it is larger than
$\#B_{j_1:3,3}+\#B_{j_2:3,3}$.  This distinguishes Case 1 and Case 2 when
$j_2=2j_1$.

Finally, suppose that $j_1$ is the only value at most $(r-1)/2$ such that
$\#B_{j:3,3}>0$.  By our choice of indexing, $j_1=\delta_{1,2}$.  It may happen
that $\delta_{2,3}>(r-1)/2$ (Case 3)
or that also $\delta_{2,3}=j_1$ and $\delta_{1,3}>(r-1)/2$ (Case 2).
These two possibilities appear as $T_3$ and $T_2$, respectively, in
Figure~\ref{distfig}.

    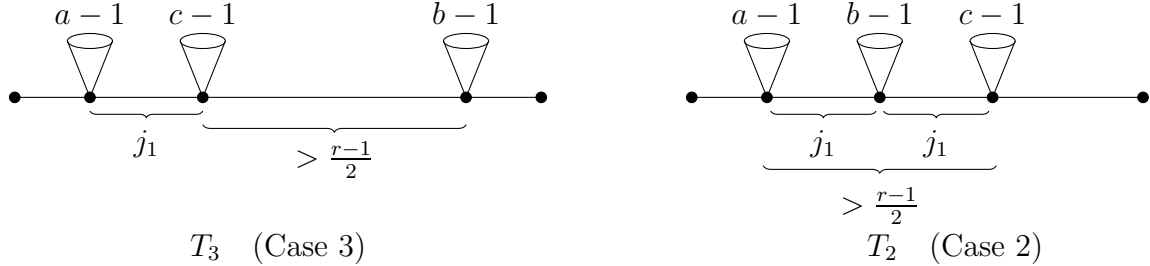
\begin{figure}[h!]
        \begin{center}
    \begin{tikzpicture}
        \draw[fill=black] (0,0) circle (2pt);
        \draw[fill=black] (1,0) circle (2pt);
        \draw[fill=black] (2.5,0) circle (2pt);
        \draw[fill=black] (6,0) circle (2pt);
        \draw[fill=black] (7,0) circle (2pt);
        \draw[thin] (0, 0)--(7,0);
        \draw[thin] (0.7,0.75)--(1,0)--(1.3,0.75);
        \draw[thin] (2.2,0.75)--(2.5,0)--(2.8,0.75);
        \draw[thin] (5.7,0.75)--(6,0)--(6.3,0.75);
        \draw (6,0.75) ellipse (0.3cm and 0.1cm);
        \draw (1,0.75) ellipse (0.3cm and 0.1cm);
        \draw (2.5,0.75) ellipse (0.3cm and 0.1cm);
        \node at (1,1.1) {$a-1$};
        \node at (2.5,1.1) {$c-1$};
        \node at (6,1.1) {$b-1$};
        \draw [decorate,decoration={brace,mirror,raise=1ex}]
          (1,0) -- (2.5,0) node[midway,yshift=-1.5em]{$j_1$};
        \draw [decorate,decoration={brace,amplitude=0.1cm,mirror,raise=2ex}]
          (2.5,0) -- (6,0) node[midway,yshift=-2em]{$>\frac{r-1}{2}$};
        \draw[fill=black] (9,0) circle (2pt);
        \draw[fill=black] (10,0) circle (2pt);
        \draw[fill=black] (11.5,0) circle (2pt);
        \draw[fill=black] (13,0) circle (2pt);
        \draw[fill=black] (15,0) circle (2pt);
        \draw[thin] (9, 0)--(15,0);
        \draw[thin] (9.7,0.75)--(10,0)--(10.3,0.75);
        \draw[thin] (11.2,0.75)--(11.5,0)--(11.8,0.75);
        \draw[thin] (12.7,0.75)--(13,0)--(13.3,0.75);
        \draw (10,0.75) ellipse (0.3cm and 0.1cm);
        \draw (11.5,0.75) ellipse (0.3cm and 0.1cm);
        \draw (13,0.75) ellipse (0.3cm and 0.1cm);
        \node at (10,1.1) {$a-1$};
        \node at (11.5,1.1) {$b-1$};
        \node at (13,1.1) {$c-1$};
        \draw [decorate,decoration={brace,mirror,raise=1ex}]
          (10.05,0) -- (11.45,0) node[midway,yshift=-1.5em]{$j_1$};
          \draw [decorate,decoration={brace,mirror,raise=1ex}]
          (11.55,0) -- (12.95,0) node[midway,yshift=-1.5em]{$j_1$};
        \draw [decorate,decoration={brace,amplitude=0.1cm,mirror,raise=5ex}]
          (9.95,0) -- (13.05,0) node[midway,yshift=-3.5em]{$>\frac{r-1}{2}$};
        \node at (3.5, -2) {$T_3$\quad(Case 3)};
        \node at (12.5, -2) {$T_2$\quad(Case 2)};
        \end{tikzpicture}
   \caption{Only one visible length of $R'$-baton\label{distfig}}
        \end{center}
    \end{figure}

In Case 3 we have $\#B_{j_1:3,3}=x_1x_2$, in Case 2
$\#B_{j_1:3,3}=x_1x_2+x_2x_3$.  However, we do not know the indexing of the
values in $\{x_1,x_2,x_3\}$.  If the product of two of those values never
equals the sum of the other two products, then $\#B_{j_1:3,3}$ distinguishes
Case 3 and Case 2.

Hence we may assume otherwise.  In particular, the product of the two largest
values in $\{x_1,x_2,x_3\}$ equals the sum of the two products involving the
smallest of the three numbers.  To compare the values, let
$b=\min\{y_1,y_2,y_3\}$, and let $a$ and $c$ be the other two values.
In Figure~\ref{distfig}, the position of the branch vertex with degree $b+1$
is fixed in the two possibilities, while the positions of those with degrees
$a+1$ and $c+1$ may be interchanged.  If Case 3 and Case 2 are
indistinguishable, then
\begin{equation}\label{32fromj33}
\CH a2\CH c2=\CH b2\left[\CH a2+\CH c2\right].
\end{equation}

If $j_1\leq (r-3)/2$,
then $B_{j_1:4,3}$ fits in a card.
Now $\#B_{j_1:4,3}=\CH a3\CH c2+\CH a2\CH c3$ in Case 3, while $\#B_{j_1:4,3}
=\CH b2\left[\CH a3+\CH c3\right]+\CH b3\left[\CH a2+\CH c2\right]$ in Case 2.
If Case 3 and Case 2 are indistinguishable, then
\begin{equation}\label{32fromj43}
\CH a3\CH c2+\CH a2\CH c3
=\CH b2\left[\CH a3+\CH c3\right]+\CH b3\left[\CH a2+\CH c2\right].
\end{equation}
Since $a,b,c\ge2$, we can divide~\eqref{32fromj43} by~\eqref{32fromj33},
which yields
$$
\FR{a-2}3+\FR{c-2}3=\FR{\CH a3+\CH c3}{\CH a2+\CH c2}+\FR{b-2}3.
$$
Multiplying both sides by ${\CH a2+\CH c2}$ and canceling like terms yields
$$
\FR{a-2}3\CH c2+\FR{c-2}3\CH a2=\FR{b-2}3\left[\CH a2+\CH c2\right].
$$
This is a contradiction, since $b<\min\{a,c\}$.  Hence~\eqref{32fromj33}
and~\eqref{32fromj43} cannot both hold, and $\#B_{j_1:3,3}$ or $\#B_{j_1:4,3}$
distinguishes Case 3 from Case 2.

We conclude that ambiguity between Case 3 and Case 2 requires $j_1\ge(r-2)/2$.
From Case 3, this requires $\delta_{1,3}=r-1$, so $j_1=(r-2)/2$.  Case 2 then
requires $\delta_{1,3}=r-2$.  The baton $B_{j_1-1:3,2}$ is small enough to fit
in a card, so we can obtain $\#B_{j_1-1:3,2}$ from the deck.  This baton is too
short to have two branch vertices as key vertices and too long to extend in
both directions from the outermost branch vertices.  Since the vertices of
degrees $a+1$ and $c+1$ can be switched, in Case 3 we have
$B_{j_1-1:3,2}\in\{\CH a2+2\CH c2+\CH b2, \CH c2+2\CH a2+\CH b2\}$.
In Case 2, $\#B_{j_1-1:3,2}=\CH a2+2\CH b2+\CH c2$.  Since $b<\min\{a,c\}$,
there are more copies in Case 3 than in Case 2, so $\#B_{j_1-1:3,2}$
distinguishes the two cases.
\end{proof}

In addition to knowing the distance case from Lemma~\ref{distcase},
from the batons in the deck we also know the actual distances, except
that in Case 3 we do not know the differences between $u_3$ and the other
indices.  Our next tast is to determine the order of the degrees of the
branch vertices along the spine.  After that we will have the smallest
subtree containing the branch vertices, and reconstruction will be completed
by determining how close that configuration is to one of the ends.

\begin{lemma}\label{3brdeg}
For an unknown caterpillar having exactly three branch vertices, whose degrees
sum to at least $11$, the deck determines the degrees of the branch vertices
in order, except that in Case $3$ it determines $d(u_3)$ and leaves the
remaining two degrees unordered.
\end{lemma}
\begin{proof}
We maintain the same notation as in the proof of Lemma~\ref{distcase}.
Since we know the degree list, we know the unordered set $\{x_1,x_2,x_3\}$.
Therefore, we know $\alpha$ and $\beta$, defined by
$\alpha=x_1x_2+x_2x_3+x_1x_3$ and $\beta=x_2x_2x_3$.  
Let $j_1=\delta_{1,2}$, $j_2=\delta_{2,3}$, and $j_3=\delta_{1,3}$; we know the
values in $(j_1,j_2,j_3)$ that are at most $(r-1)/2$.  Let $b_i=\#B_{j_i:3,3}$
for $i\in \{1,2,3\}$; we know $b_i$ when $j_i\le(r-1)/2$.

In Case $1$ or Case $2$ of the distance cases, we know $j_1,j_2,b_1,b_2$
If $j_1<j_2$, then $b_1=x_1x_2$ and $b_2=x_2x_3$.  Hence we obtain
$x_3=\beta/b_1$, $x_1=\beta/b_2$, and then $x_2=\beta/(x_1x_3)$.

If $\delta_{1,2}=\delta_{2,3}$, then $b_1=b_2=x_1x_2+x_2x_3$ and $b_3=x_1x_3$.
We can obtain $b_3$ as $\alpha-b_1$, even if $j_3>(r-1)/2$ (that is, the
computation is valid in Case $1$ or Case $2$).  We then have $x_2=\beta/b_2$.
The ordering of $x_1$ and $x_3$ can be chosen arbitrarily (see
Figure~\ref{distfig}); to make the choice definite, we make it so that
$x_1\le x_3$.

In Case $3$, we have learned only $b_1=x_1x_2$ and obtain $x_3=\beta/b_1$.  The
set $\{x_1, x_2\}$ remains unordered at this point (see Figure~\ref{distfig}).

We retrieve the degrees of the branch vertices using $x_i=\CH{d(v_{u_i})-1}2$.
\end{proof}

In Case $1$ or Case $2$ of the distance cases, we now know the subgraph
induced by the portion of the spine from $v_{u_1}$ to $v_{u_3}$ and the
neighboring leaves.  To complete the reconstruction, it suffices to determine
the index of one end of this portion, $u_1$ or $u_3$.  Case $3$ will be
somewhat harder, since there we do not yet completely know the ordering
of the degrees of branch vertices.

\begin{lemma}\label{3brrec}
If $G$ is a caterpillar having exactly three branch vertices, whose degrees
sum to at least $11$, then $G$ is determined by the deck.
\end{lemma}
\begin{proof}
We know $\#B_{j:3,3}$ when $j\le (r-1)/2$ and $\#B_{j:3,2}$ when $j\le(r+1)/2$.

Let $\#B'_{j:3,2}$ denote the number of copies of $B_{j:3,2}$ in which the
leaf neighbor of the key vertex of degree $2$ lies on the augmented spine.
We don't see $\#B'_{j:3,2}$ directly in the deck, but we will be able to
compute these values when needed using the expressions below in~\eqref{B'}.
Recall the notation $y_i=d(v_{u_i})-1$ and $x_i=\CH{y_i}2$.  The computation
uses that the key vertices in copies of $B_{j:3,2}$ not counted by
$\#B'_{j:3,2}$ must both be branch vertices.  By convention, let
$\#B'_{0:3,2}=2x_1+2x_2+2x_3$.
\begin{align}\label{B'}
  \#B'_{j:3,2}=\begin{cases}
    \#B_{j:3,2} & j\not\in \{\delta_{1,2},\delta_{2,3},\delta_{1,3}\}\\
    \#B_{j:3,2}-x_1(y_2\!-\!1)-x_2(y_1\!-\!1) & j=\delta_{1,2}\ne\delta_{2,3}\\
    \#B_{j:3,2}-x_2(y_3\!-\!1)-x_3(y_2\!-\!1) & j=\delta_{2,3}\ne\delta_{1,2}\\
    \#B_{j:3,2}-x_1(y_3\!-\!1)-x_3(y_1\!-\!1) & j=\delta_{1,3}\\
    \#B_{j:3,2}-x_1(y_2\!-\!1)-x_2(y_1\!-\!1)-x_2(y_3\!-\!1)-x_3(y_2\!-\!1) & j=\delta_{1,2}=\delta_{2,3}\\
        \end{cases}
    \end{align}
We cannot compute $\#B'_{j:3,2}$ when $j>(r+1)/2$, because then
$B_{j:3,2}$ does not fit in a card (such as when $j=\delta_{1,3}>(r+1)/2$.
It does fit when $j\le(r+1)/2$, since $\SE i13 d_i\ge11$ implies $n-r\ge7$ and
$(r+1)/2+4\le(n+2)/2$.

When $j<u_1$ and $j<\ub_3$, we have $\#B'_{j:3,2}=2x_1+2x_2+2x_3$, because
there is room for $B_{j:3,2}$ to extend in either direction from any branch
vertex.  When the length $j$ exceeds a value equal to the distance between
a branch vertex $u_1$ and $v_1$ or $v_r$, we lose the $x_i$ batons that had
been extending toward one end of the caterpillar.  That is, the values $j$ such
$\#B'_{j:3,2}<\#B'_{j-1:3,2}$ are the levels $j$ such that there is a branch
vertex in $\{v_j,v_\jb\}$.  Since there are three branch vertices, there are
two or three values of $j$ where such a difference occurs.  (Our convention
for $\#B'_{0:3,2}$ makes the computation valid also when $j=1$.)
We consider cases as in Definition~\ref{distdef}.  

\medskip
{\bf Cases 1\&2:} {\it $\delta_{1,2}\le\delta_{2,3}\le(r-1)/2$.}
By Lemma~\ref{3brdeg} we know the values $y_i$ and $x_i$ for $i\in\{1,2,3\}$,
so we obtain $\#B_{j:3,2}$ from $\#B_{j:3,2}$ as in~\eqref{B'} for
$j\le(r+1)/2$.

Let $h$ be the least $j$ such that $\#B'_{h:3,2}<2x_1+2x_2+2x_3$.  Since
$u_2>\min\{u_1,\ub_3\}$, we have $h=\min\{u_1,\ub_3\}$.  To determine whether
$h$ is $u_1$ or $\ub_3$ (or both) we compute the difference.  If $x_1\ne x_3$,
then
    \begin{align}\label{B'diff}
        \#B'_{h-1:3,2}-\#B'_{h:3,2}=\begin{cases}
            x_1 & u_1=h\\
            x_1+x_3 & u_1=\ub_3=h\\
            x_3 & \ub_3=h
        \end{cases} ~.
    \end{align}
The outcome specifies the position along the spine for at least one of
$\{v_{u_1},v_{u_3}\}$, completing the reconstruction.

When $x_1=x_3$, the difference $D$ in~\eqref{B'diff} does not distinguish
$u_1=h$ and $\ub_3=h$.  If $u_1=\ub_3=h$, which we recognize by $D=x_1+x_3$,
then the caterpillar is determined up to symmetry.  Hence we may assume
$u_1\ne \ub_3$.

Since we know $x_1,x_2,x_3$, we know the three levels of the branch vertices,
two of which may be equal.  If $\delta_{1,2}=\delta_{2,3}$, then setting
$u_1=h$ or $u_3=\hb$ yields the same caterpillar, since $x_1=x_3$.  Hence
we may assume $\delta_{1,2}<\delta_{2,3}$.  Now setting $u_1=h$ or $u_3=\hb$
interchanges the levels of $v_{u_1}$ and $v_{u_3}$, yielding the same 
contributions to the differences $\#B'_{j-1:3,2}-\#B'_{j:3,2}$ from these
two vertices.  However, when $u_1=h$ the vertex $v_{u_2}$ has level
$h+\delta_{1,2}$, and when $\ub_3=h$ the level of $v_{u_2}$ is higher than
$h+\delta_{1,2}$.  We detect this difference, since we know the levels $j$
such that $\#B'_{j:3,2}<\#B'_{j-1:3,2}$.  This distinguishes which branch
vertex has level $h$, completing the reconstruction.

\medskip
{\bf Case 3:} {\it $\delta_{1,2}\le(r-1)/2<\delta_{2,3}$.}
By Lemma~\ref{3brdeg}, we know $x_3$ and hence also the set $\{x_1, x_2\}$.
Since we do not know $\delta_{2,3}$, fixing $u_1$ does not complete the
reconstruction.  However, since $\delta_{2,3}>(r-1)/2$, we know that $v_{u_3}$
is in the second half of the spine and that $v_{u_1}$ and $v_{u_2}$ are in the
first half ($\delta_{2,3}>(r-1)/2$ yields $u_2\le r/2$).  Hence it suffices to
determine the levels of the three branch vertices and assign the values in
$\{x_1,x_2\}$ to the correct indices.

To find the levels of branch vertices, we seek $j$ such that
$\#B'_{j:3,2}<\#B'_{j-1:3,3}$ and $j\le r/2$.  Because $\delta_{2,3}>(r-1)/2$,
the only lines in~\eqref{B'} that are relevant for this computation are the
first two.  When $j=\delta_{1,2}$, we know the value of $\#B'_{j:3,2}$ because
the formula is unchanged when $(x_1,y_1)$ and $(x_2,y_2)$ are interchanged.
For all other relevant values, $\#B'_{j:3,2}=\#B_{j:3,2}$. 

Let $h$ be the least value of $j$ such that $\#B'_{j:3,2}<\#B'_{j-1:3,2}$,
and let $h'$ be the next.  These are the only two such values if and only if
two branch vertices are on the same level.  The possibilities are then
$u_1=\ub_3=h$ or $u_2=\ub_3=h'$.  We obtain
    \begin{align}\label{B'h}
        \#B'_{h-1:3,2}-\#B'_{h:3,2}=\begin{cases}
            x_1+x_3 & u_1=h=\ub_3\\
            x_1 & u_1=h\ne \ub_3
        \end{cases}
    \end{align}
We also have
    \begin{align}\label{B'h'}
        \#B'_{h'-1:3,2}-\#B'_{h':3,2}=\begin{cases}
            x_2 & u_2=h'\ne\ub_3\\
            x_2+x_3 & u_2=h'=\ub_3
        \end{cases}
    \end{align}
Thus the ordered pair of values in~\eqref{B'h} and~\eqref{B'h'}
is $(x_1+x_3,x_2)$ (two branch vertices at level $h$) or 
$(x_1,x_2+x_3)$ (two branch vertices at level $h'$).
If $x_1=x_2$, then the ordered pairs are distinguished by which
coordinate is larger, completing the reconstruction.  If $x_1\ne x_2$, then
we count appearances of $B_{\delta_{1,2},h:3,3,2}$, which fits in a card since
its length is at most $(r-2)/2$.  We have
$\#B_{\delta_{1,2},h:3,3,2}=x_1(y_2-1)$, and that value
differs from $x_2(y_1-1)$ when $x_1\ne x_2$.  Thus this count tells us whether
the degrees of $v_{u_1}$ and $v_{u_2}$ respect the order
$(x_1,x_2)$ or the order $(x_2,x_1)$.  Having determined which of $x_1$ and
$x_2$ is smaller, we can determine which of them we see in the ordered pair
from~\eqref{B'h} and~\eqref{B'h'}, completing the reconstruction.

\medskip
It remains only to consider the subcase where the three branch vertices have
distinct levels.  That is, there are distinct values $h_1,h_2,h_3$ with
$1\le h_1<h_2<h_3\le r/2$ such that $\#B'_{h_i:3,2}<\#B'_{h_i-1:3,2}$.
The caterpillar is now one of three possibilities, given by $(u_1,u_2,u_3)$
being $(h_1,h_2,\hb_3)$, $(h_1,h_3,\hb_2)$, or $(h_2,h_3,\hb_1)$.

Let $\beta_i= \#B'_{h_i-1:3,2}-\#B'_{h_i:3,2}$.
In the three options for $G$, the triple $(\beta_1,\beta_2,\beta_3)$ is
$(x_1,x_2,x_3)$, $(x_1,x_3,x_2)$, or $(x_3,x_1,x_2)$.

We first distinguish the second possibility from the others using
$\delta_{1,2}$, a value we know.  In the three options, $\delta_{1,2}$ is
$h_2-h_1$, $h_3-h_1$, or $h_3-h_2$, respectively.  Since $h_3-h_1$ is strictly
largest, that option cannot be confused with either of the others.  It also
distinguishes $x_1$ as the difference associated with $h_1$.

Ambiguity among the other two options requires $(x_1,x_2,x_3)=(x_3,x_1,x_2)$.
Here $x_3$ must equal both $x_1$ and $x_2$, so all three branch vertices have
the same degree, $d$.  Now we consider a triton.  Let $x=\CH{d-1}2$.  We have
$\#B_{\delta_{1,2},h:3,3,2}=x(d-2)$ if $u_1=h_1$, and
$\#B_{\delta_{1,2},h:3,3,2}=2x(d-2)$ if $u_1=h_2$.  Hence we can distinguish
the cases.  The value $\beta_2$ then tells us which of $\{x_1,x_2\}$ is
which, completing the reconstruction.
\end{proof}

We have completed the consideration of all cases and proved
Theorem~\ref{main}.

\paragraph{Acknowledgment.} We thank Carla Groenland and Zachary Hunter for
helpful comments.

\end{document}